\documentclass[a4paper]{amsart}
\pdfoutput=1 
\usepackage[utf8]{inputenc}
\usepackage[T1]{fontenc}
\usepackage{lmodern}
\usepackage{amssymb,amsxtra}
\usepackage{mathabx}
\usepackage{nicefrac,mathtools}
\usepackage[lite]{amsrefs}

\renewcommand*{\MR}[1]{ \href{http://www.ams.org/mathscinet-getitem?mr=#1}{MR \textbf{#1}}}
\newcommand*{\arxiv}[1]{\href{http://www.arxiv.org/abs/#1}{arXiv: #1}}
\renewcommand{\PrintDOI}[1]{\href{http://dx.doi.org/\detokenize{#1}}{doi: \detokenize{#1}}}

\BibSpec{book}{%
  +{}  {\PrintPrimary}                {transition}
  +{,} { \textit}                     {title}
  +{.} { }                            {part}
  +{:} { \textit}                     {subtitle}
  +{,} { \PrintEdition}               {edition}
  +{}  { \PrintEditorsB}              {editor}
  +{,} { \PrintTranslatorsC}          {translator}
  +{,} { \PrintContributions}         {contribution}
  +{,} { }                            {series}
  +{,} { \voltext}                    {volume}
  +{,} { }                            {publisher}
  +{,} { }                            {organization}
  +{,} { }                            {address}
  +{,} { \PrintDateB}                 {date}
  +{,} { }                            {status}
  +{}  { \parenthesize}               {language}
  +{}  { \PrintTranslation}           {translation}
  +{;} { \PrintReprint}               {reprint}
  +{.} { }                            {note}
  +{.} {}                             {transition}
  +{} { \PrintDOI}                   {doi}
  +{} { available at \url}            {eprint}
  +{}  {\SentenceSpace \PrintReviews} {review}
}

\usepackage{enumitem} 
\setlist[enumerate,1]{label=\textup{(\arabic*)}}
\setlist[enumerate,2]{label=\textup{(\alph*)}}

\usepackage{tikz}

\tikzset{cd/.style=matrix of math nodes,row sep=2em,column sep=2em, text height=1.5ex, text depth=0.5ex}
\tikzset{cdar/.style=->,auto}
\tikzset{mid/.style={anchor=mid}} 
\tikzset{narrowfill/.style={inner sep=1pt, fill=white}}

\usepackage{tikz-cd}
\usepackage{a4wide}
\usepackage{microtype}
\usepackage[pdftitle={Ideal structure and pure infiniteness of
  inverse semigroup crossed products},
pdfauthor={Bartosz Kwa\'sniewski, Ralf Meyer},
pdfsubject={Mathematics}]{hyperref}

\theoremstyle{plain}
\newtheorem{theorem}[equation]{Theorem}
\newtheorem{lemma}[equation]{Lemma}
\newtheorem{proposition}[equation]{Proposition}
\newtheorem{corollary}[equation]{Corollary}
\theoremstyle{definition}
\newtheorem{definition}[equation]{Definition}
\theoremstyle{remark}
\newtheorem{remark}[equation]{Remark}
\newtheorem{example}[equation]{Example}

\DeclareMathOperator{\Prim}{Prim}
\DeclareMathOperator{\Prime}{Prime}
\DeclareMathOperator{\Bis}{Bis}
\DeclareMathOperator{\supp}{supp}

\newcommand{\twolinesubscript}[2]{\genfrac{}{}{0pt}{}{#1}{#2}}

\newcommand{\LL}{\mathcal L}

\newcommand{\cl}[1]{\overline{#1}}

\newcommand{\CC}{\mathcal C}

\newcommand*{\nb}{\nobreakdash}
\newcommand*{\Star}{\(^*\)\nobreakdash-}

\newcommand*{\C}{\mathbb C}
\newcommand*{\Z}{\mathbb Z}
\newcommand*{\R}{\mathbb R}
\newcommand*{\N}{\mathbb N}

\newcommand*{\Gr}{G}

\newcommand*{\Locmult}{\mathcal{M}_\mathrm{loc}}
\newcommand*{\Ideals}{\mathbb I}
\newcommand*{\Null}{\mathcal N}
\newcommand*{\Bound}{\mathbb B}
\newcommand*{\Comp}{\mathbb K}
\newcommand*{\Mat}{\mathbb M}
\newcommand{\Her}{\mathbb H}
\newcommand*{\red}{\mathrm r}
\newcommand*{\ess}{\mathrm{ess}}
\newcommand*{\alg}{\mathrm{alg}} 

\newcommand*{\hull}{\mathrm{I}} 

\newcommand*{\Cst}{\textup C^*}
\newcommand*{\Mult}{\mathcal M}
\newcommand*{\UMult}{\mathcal{UM}}

\newcommand*{\Cont}{\textup C}
\newcommand*{\Contc}{\Cont_\textup c} 

\newcommand*{\prid}[1][p]{\mathfrak{#1}} 

\newcommand{\idealin}{\mathrel{\triangleleft}} 
\newcommand*{\Slice}{\mathcal S}
\newcommand*{\Id}{\textup{Id}}
\newcommand*{\Ad}{\textup{Ad}}

\newcommand*{\Hils}{\mathcal H}
\newcommand*{\Hilm}[1][E]{\mathcal #1}

\newcommand*{\A}{\mathcal A}
\newcommand*{\defeq}{\mathrel{\vcentcolon=}}
\newcommand*{\congto}{\xrightarrow\sim}

\DeclarePairedDelimiter{\norm}{\lVert}{\rVert}
\DeclarePairedDelimiterX{\braket}[2]{\langle}{\rangle}{#1\,\delimsize\vert\,\mathopen{}#2}
\DeclarePairedDelimiterX{\BRAKET}[2]{\langle}{\rangle}{\!\delimsize\langle#1\,\delimsize\vert\,\mathopen{}#2\delimsize\rangle\!}
\DeclarePairedDelimiterX{\setgiven}[2]{\{}{\}}{#1\,{:}\,\mathopen{}#2}

\newcommand*{\dual}[1]{\widehat{#1}}

\newcommand*{\s}{s}
\newcommand*{\rg}{r}

\DeclareMathOperator*{\stlim}{s-lim}

\newcommand*{\into}{\rightarrowtail}
\newcommand*{\onto}{\twoheadrightarrow}

\makeatletter
\@namedef{subjclassname@2020}{%
  \textup{2020} Mathematics Subject Classification}
\makeatother

\begin{document}

\title[Purely infinite crossed products]{Ideal structure and pure infiniteness of\\
inverse semigroup crossed products}
\author{Bartosz Kosma Kwa\'sniewski}
\email{bartoszk@math.uwb.edu.pl}
 \address{Faculty of Mathematics\\
   University  of Bia\l ystok\\
   ul.\@ K.~Cio\l kowskiego 1M\\
   15-245 Bia\l ystok\\
   Poland}

\author{Ralf Meyer}
\email{rmeyer2@uni-goettingen.de}
\address{Mathematisches Institut\\
 Georg-August-Universit\"at G\"ottingen\\
 Bunsenstra\ss e 3--5\\
 37073 G\"ottingen\\
 Germany}

\begin{abstract}
  We give efficient conditions under which a \(\Cst\)-subalgebra
  \(A\subseteq B\) separates ideals in a \(\Cst\)\nb-algebra~\(B\),
  and~\(B\) is purely infinite if every positive element in~\(A\) is
  properly infinite in~\(B\).  We specialise to the case when~\(B\)
  is a crossed product for an inverse semigroup action by Hilbert
  bimodules or a section \(\Cst\)-algebra of a Fell bundle over an
  \'etale, possibly non-Hausdorff, groupoid.  Then our theory works
  provided~\(B\) is the recently introduced essential crossed
  product and the action is essentially exact and residually
  aperiodic or residually topologically free. These last notions are developed in the article.
\end{abstract}
\thanks{The first named author was supported by the National Science Center
  (NCN), Poland, grant no.~2019/35/B/ST1/02684.   We
    thank  Alcides Buss, Diego Mart\'inez and Jonathan Taylor for
    pointing out an error in the previous version of
    Proposition~\ref{prop:functoriality}.}

\maketitle


\section{Introduction}
\label{sec:introduction}

Many authors have given sufficient criteria for crossed products by
discrete group actions or for \(\Cst\)\nb-algebras associated to
étale locally compact groupoids to be purely infinite (see, for
instance, \cites{Anantharaman-Delaroche:Purely_infinite,
  Laca-Spielberg:Purely_infinite,
  Jolissaint-Robertson:Simple_purely_infinite,
  Rordam-Sierakowski:Purely_infinite,
  Pasnicu-Phillips:Spectrally_free,
  Kirchberg-Sierakowski:Strong_pure,
  Kwasniewski-Szymanski:Pure_infinite, Boenicke-Li:Ideal,
  Rainone-Sims:Dichotomy, Ma:Purely_infinite_groupoids}).  These
articles mostly deal with the case when the bigger
\(\Cst\)\nb-algebra~\(B\) is simple or when the
\(\Cst\)\nb-subalgebra \(A\subseteq B\) on which the action takes
place is commutative and has totally disconnected spectrum.  In
addition, étale groupoids are required to be Hausdorff.  These pure
infiniteness criteria also imply that~\(A\) separates ideals
in~\(B\).  Then the ideal lattice of~\(B\) is
isomorphic to the lattice of invariant ideals in~\(A\).  Here we
 formulate sufficient conditions for~\(A\) to separate ideals
in~\(B\) and for~\(B\) to be purely infinite, which allow~\(A\) to be
noncommutative and which impose no Hausdorffness restrictions.  In
this generality, it is natural to study actions of inverse
semigroups by Hilbert bimodules
(see~\cite{Buss-Meyer:Actions_groupoids}) or, equivalently, section
algebras of Fell bundles over inverse semigroups.  This contains
Fell bundles over discrete groups and over étale groupoids
--~possibly non-Hausdorff~-- as special cases.  Another special case
are Exel's noncommutative Cartan \(\Cst\)\nb-inclusions
(see~\cites{Exel:noncomm.cartan, Kwasniewski-Meyer:Cartan}), which
generalise Renault's (commutative) Cartan subalgebras.

This article is based on our recent papers,
\cites{Kwasniewski-Meyer:Stone_duality,
  Kwasniewski-Meyer:Aperiodicity,
  Kwasniewski-Meyer:Aperiodicity_pseudo_expectations,
  Kwasniewski-Meyer:Essential}, where two key concepts are
developed.  The first one is the \emph{essential crossed product}
introduced in~\cite{Kwasniewski-Meyer:Essential}, which is a
variation on the reduced crossed product that ``always'' has the
expected ideal structure --~even for general actions of inverse
semigroups and for actions of non-Hausdorff groupoids.  The second
concept is \emph{aperiodicity}.  It is a strong regularity property,
abstracted from the work of Kishimoto, Olesen--Pedersen, and others,
defined for a general \(\Cst\)\nb-inclusion \(A\subseteq B\)
in~\cite{Kwasniewski-Meyer:Essential}.  As shown
in~\cite{Kwasniewski-Meyer:Aperiodicity_pseudo_expectations},
aperiodicity implies that there is a \emph{unique
  pseudo-expectation} --~a unique generalised conditional
expectation for \(A\subseteq B\) taking values in Hamana's injective
hull of~\(A\).  If in addition this pseudo-expectation is almost
faithful, then~\(A\) \emph{supports}~\(B\) in the sense that every
element in \(B^+\setminus\{0\}\) is supported by an element in
\(A^+\setminus\{0\}\) in the Cuntz preorder
of~\cite{Cuntz:Dimension_functions}.  This, in turn, implies
that~\(A\) \emph{detects ideals} in~\(B\), that is,
\(J\cap A \neq 0\) for any ideal \(0\neq J\idealin B\).  All these
properties are closely related.  In fact, when~\(B\) is an essential
crossed product and~\(A\) is separable or of Type~I, then the
following conditions are equivalent: aperiodicity, unique
pseudo-expectation, supporting positive elements in all intermediate
\(\Cst\)\nb-algebras, detection of ideals in all intermediate
\(\Cst\)\nb-algebras, and \emph{topological freeness} of the dual
groupoid
(see~\cite{Kwasniewski-Meyer:Aperiodicity_pseudo_expectations}).

In the present paper, we study ``residual'' versions of these
conditions, that is, when they hold for quotient inclusions
\(A/I \subseteq B/BIB\) for all ideals \(I\idealin A\) that are
restricted from~\(B\).  The relationship between ideals in~\(B\) and
\emph{restricted ideals} in~\(A\) is thoroughly studied
in~\cite{Kwasniewski-Meyer:Stone_duality}.  For a crossed product
inclusion, the restricted ideals in~\(A\) are exactly those that are
\emph{invariant} under the action that produces~\(B\) from~\(A\).
The residual version of detection of ideals is separation of ideals.
We say that~\(A\) \emph{separates ideals} in~\(B\) if
\(I\cap A = J\cap A\) for ideals \(I,J\idealin B\) implies \(I=J\).
This identifies the ideal lattice of~\(B\) with the lattice of
restricted ideals in~\(A\).  Under some extra assumptions, we may
also identify the primitive ideal space of~\(B\) with the
\emph{quasi-orbit space} of the induced action on the primitive
ideal space of~\(A\) (see~\cite{Kwasniewski-Meyer:Stone_duality}).
A residual version of supporting is closely related to the formally
stronger condition called \emph{filling}, which was used in
\cites{Kirchberg-Sierakowski:Strong_pure,
  Kirchberg-Sierakowski:Filling_families} to prove strong pure
infiniteness.  For a large class of \(\Cst\)\nb-inclusion
\(A\subseteq B\), we show that~\(A\) residually supports~\(B\) if
and only if~\(A\) fills~\(B\).  Namely, this holds for the
\emph{symmetric} inclusions defined
in~\cite{Kwasniewski-Meyer:Stone_duality}, which include all sorts
of \(\Cst\)\nb-inclusions coming from crossed products.  For general
residually supporting \(\Cst\)\nb-inclusions \(A\subseteq B\) and a
family \(\mathcal{F}\subseteq A^+\) of elements in~\(A\) that are
properly infinite in~\(B\), we give sufficient conditions for~\(B\)
to be purely infinite (see
Theorem~\ref{thm:pure_infiniteness_criteria} below).

To ensure that~\(A^+\) residually supports~\(B\) we assume that the
inclusion \(A\subseteq B\) is \emph{residually aperiodic}, in the
sense that for any restricted ideal \(I\idealin A\), the inclusion
\(A/I\to B/BIB\) is aperiodic.  We also need to assume that the
pseudo-expectations for these quotient inclusions are almost
faithful.  For actions by discrete groups this ``residual
faithfulness of conditional expectations'' is also called
\emph{exactness} (see
\cites{Sierakowski:IdealStructureCrossedProducts,
  Abadie-Abadie:Ideals, Kwasniewski-Szymanski:Pure_infinite}) and
for groupoids inner exactness (see
\cites{AnantharamanDelaroch:Weak_containment, Boenicke-Li:Ideal}).
We generalise this concept to inverse semigroup actions by Hilbert
bimodules and Fell bundles over \'etale groupoids.  
 In particular, 
we prove that the full crossed product is an exact functor and the
reduced crossed product is an injective functor, but only when restricted
to special homomorphisms between actions (Propositions
\ref{prop:functoriality} and~\ref{prop:exactness_of_universal_rep}). 
The essential crossed product is not functorial.  Therefore,
exactness for these crossed products, which we call \emph{essential
  exactness}, is more subtle.

As an appetizer, we formulate here a theorem that summarises and
illustrates some of our results.  An action of a unital inverse
semigroup~\(S\) by Hilbert bimodules on a \(\Cst\)\nb-algebra~\(A\)
is a semigroup \(\Hilm=(\Hilm_t)_{t\in S}\), where each
fibre~\(\Hilm_t\) is a Hilbert \(A\)\nb-bimodule and the semigroup
product is compatible with the internal tensor product (see
Definition~\ref{def:S_action_Cstar} below).  This induces an
\(S\)-action by partial homeomorphisms on the spectrum~\(\dual{A}\)
and the primitive ideal space~\(\check{A}\).  The corresponding
transformation groupoids \(\dual{A}\rtimes S\) and
\(\check{A}\rtimes S\) are called \emph{dual groupoids}
of~\(\Hilm\).  If the unit spaces in these groupoids are closed, we
call~\(\Hilm\) a \emph{closed action}.  Then the essential crossed
product \(A\rtimes_\ess S\) coincides with the reduced crossed
product \(A\rtimes_\red S\).  The following theorem combines
Theorems \ref{the:residually_aperiodic_characterisations},
\ref{thm:residual_aperiodic_action}
and~\ref{thm:Kirchberg_Sierakowski2} and
Corollary~\ref{cor:pure_infiniteness_for_crossed_products}:

\begin{theorem}
  \label{thm:introduction}
  Let \(\Hilm=(\Hilm_t)_{t\in S}\) be an inverse semigroup action by
  Hilbert bimodules on a \(\Cst\)\nb-algebra~\(A\).  Assume
  that~\(\Hilm\) is essentially exact \textup{(}or exact when the
  action is closed\textup{)} and residually aperiodic \textup{(}this
  holds when \(\dual{A}\rtimes S\) is residually topologically
  free\textup{)}.  Then
  \begin{enumerate}
  \item \label{en:introduction1}%
    \(A\) fills the essential crossed product \(A\rtimes_\ess S\).
    So \(A\rtimes_\ess S\) is strongly purely infinite if and only
    if every pair of elements in~\(A^+\) satisfies the matrix
    factorisation property
    of~\cite{Kirchberg-Sierakowski:Filling_families};
  \item \label{en:introduction2}%
    the ideal lattice of \(A\rtimes_\ess S\) is isomorphic to the
    lattice \(\Ideals^{\Hilm}(A)\) of \(\Hilm\)\nb-invariant ideals
    in~\(A\).  If~\(\check{A}\) is second countable, the primitive
    ideal space of \(A\rtimes_\ess S\) is homeomorphic to the
    quasi-orbit space \(\check{A}/{\sim}\) of the dual groupoid
    \(\check{A}\rtimes S\);
  \item \label{en:introduction3}%
    if \(\mathcal{F}\subseteq A^+\) residually supports~\(A\) and
    consists of residually \(\Hilm\)\nb-infinite elements
    \textup{(}see Definition~\textup{\ref{def:infinite_elements_B})},
    and \(\Ideals^{\Hilm}(A)\) is finite or the projections
    in~\(\mathcal{F}\) separate the ideals
    in~\(\Ideals^{\Hilm}(A)\), then \(A\rtimes_\ess S\) is purely
    infinite and has the ideal property.
  \end{enumerate}
\end{theorem}

The above theorem directly applies to (twisted) crossed products by
discrete groups.  As we explain in
Section~\ref{subsection:Pure_infiniteness_criteria}, it covers all
the pure infiniteness results in
\cites{Laca-Spielberg:Purely_infinite,
  Jolissaint-Robertson:Simple_purely_infinite,
  Rordam-Sierakowski:Purely_infinite,
  Pasnicu-Phillips:Spectrally_free,
  Kirchberg-Sierakowski:Strong_pure,
  Kwasniewski-Szymanski:Pure_infinite}.
Theorem~\ref{thm:introduction} has an analogue for Fell bundles over
\'etale groupoids (Corollary~\ref{cor:Ideals_in_section_algebras}).
In particular, it can be used to generalise the results in
\cites{Boenicke-Li:Ideal, Rainone-Sims:Dichotomy,
  Ma:Purely_infinite_groupoids} to twisted (not necessarily
Hausdorff) \'etale groupoids.  The twisted version covers all Cartan
inclusions of Renault~\cite{Renault:Cartan.Subalgebras} (see
Corollary~\ref{cor:purely_infinite_groupoid} and
Remark~\ref{rem:paradoxicality_for_groupoids}).  In fact,
Theorem~\ref{thm:introduction} may be applied to all regular
residually aperiodic \(\Cst\)\nb-inclusions \(A\subseteq B\) with a
residually faithful conditional expectation (see
Proposition~\ref{prop:residual_aperiodic_Cartan}).  This includes a
large class of noncommutative Cartan inclusions in the sense of
Exel~\cite{Exel:noncomm.cartan}.  See
\cite{Kwasniewski-Meyer:Cartan}*{Theorem~4.3} for a number of
equivalent characterisations of such inclusions.

The article is organised as follows.
Section~\ref{sec:detect_separate} reviews general results about
restriction and induction of ideals
from~\cite{Kwasniewski-Meyer:Stone_duality}, discusses residually
supporting and filling families and their relationship, and presents
our general pure infiniteness criteria for \(\Cst\)\nb-algebras.  In
Section~\ref{sec:isg_crossed}, we recall actions of inverse
semigroups by Hilbert bimodules and their crossed products.  In
Section \ref{sec:exactness} we discuss restrictions, functoriality
and exactness of inverse semigroup crossed products.
Section~\ref{sec:crossed_infinite} discusses our general pure
infiniteness criteria for crossed products by inverse semigroup
actions.

\section{Separation of ideals and pure infiniteness criteria for
  \texorpdfstring{$\Cst$}{C*}-inclusions}
\label{sec:detect_separate}
\numberwithin{equation}{section}

\subsection{Separation and detection of ideals}

Let \(A\subseteq B\) be a \(\Cst\)\nb-inclusion.  We recall some
results from~\cite{Kwasniewski-Meyer:Stone_duality} that relate the
ideal structure of the two \(\Cst\)\nb-algebras \(A\) and~\(B\).
(In~\cite{Kwasniewski-Meyer:Stone_duality}, we also consider the
more general situation of an inclusion \(A\to \Mult(B)\) into the
multiplier algebra of \(B\).)

\begin{definition}
  \label{def:restricted_induced_ideals}
  Let \(\Ideals(A)\) and \(\Ideals(B)\) be the complete lattices of
  (closed, two-sided) ideals in \(A\) and~\(B\), respectively.  For
  \(I\in\Ideals(A)\), let \(BIB\in\Ideals(B)\) be the ideal in~\(B\)
  generated by~\(I\).  We define the \emph{restriction map~\(r\)}
  and the \emph{induction map~\(i\)} by
  \begin{alignat*}{2}
    r&\colon \Ideals(B) \to \Ideals(A),&\qquad J&\mapsto J\cap A,\\
    i&\colon \Ideals(A) \to \Ideals(B),&\qquad I&\mapsto BIB.
  \end{alignat*}
  We call \(I\in\Ideals(A)\) \emph{restricted} if \(I=r(J)\) for
  some \(J\in\Ideals(B)\).  We call \(J\in\Ideals(B)\)
  \emph{induced} if \(J=i(I)\) for some \(I\in\Ideals(A)\).  Let
  \(\Ideals^B(A)\subseteq \Ideals(A)\) and
  \(\Ideals^A(B)\subseteq \Ideals(B)\) be the subsets of restricted
  and induced ideals, respectively.
\end{definition}

The maps \(r\) and~\(i\) form a (monotone) \emph{Galois connection},
that is, if \(I\in\Ideals(A)\), \(J\in\Ideals(B)\), then
\(I\subseteq r(J)\) if and only if \(i(I)\subseteq J\).  This
observation goes back to Green~\cite{Green:Local_twisted}.  It has a
number of consequences.  For instance, the maps \(i\) and~\(r\) are
monotone and satisfy \(r\circ i(I)\supseteq I\) and
\(i\circ r\circ i(I)=i(I)\) for all \(I\in\Ideals(A)\), and
\(i\circ r(J)\subseteq J\) and \(r\circ i\circ r(J)=r(J)\) for all
\(J\in\Ideals(B)\).  The map~\(i\) preserves joins and~\(r\)
preserves meets.  The maps
\(r\colon \Ideals(B) \to \Ideals^B(A)\subseteq \Ideals(A)\) and
\(i\colon \Ideals(A) \to \Ideals^A(B)\subseteq \Ideals(B)\) restrict
to mutually inverse isomorphisms of partially ordered sets
\[
  \Ideals^A(B)\cong \Ideals^B(A).
\]
Thus \(\Ideals^A(B)\cong \Ideals^B(A)\) are complete lattices with
inclusion as the partial order.  The map~\(r\) is injective if and
only if \(i\) is surjective, if and only if
\(\Ideals(B)=\Ideals^A(B)\).  Subalgebras with this property are
said to separate ideals.

\begin{definition}
  \label{def:detect_separate}
  We say that~\(A\) \emph{separates ideals} in~\(B\) if
  \(J_1\cap A\neq J_2\cap A\) for all \(J_1,J_2\idealin B\) with
  \(J_1\neq J_2\) or, equivalently, \(r\) is injective.  We say
  that~\(A\) \emph{detects ideals} in~\(B\) if \(J\cap A\neq0\) for
  all \(J\idealin B\) with \(J\neq0\) or, equivalently,
  \(r^{-1}(0)=\{0\}\).
\end{definition}

Detection of ideals is sometimes called the intersection property.
Separation of ideals is a residual version of detection of ideals.

\begin{lemma}[\cite{Kwasniewski-Meyer:Stone_duality}*{Proposition 2.12}]
  \label{lem:separate_induced}
  A \(\Cst\)\nb-subalgebra \(A\subseteq B\) separates ideals if and
  only if \(A/I \subseteq B/BIB\) detects ideals for all restricted
  ideals \(I\in\Ideals^B(A)\).
\end{lemma}

\subsection{Symmetric and regular
  \texorpdfstring{$\Cst$}{C*}-inclusions}

We will be mainly interested in regular inclusions, and these are
symmetric as in the following definition. 

\begin{definition}[\cite{Kwasniewski-Meyer:Stone_duality}*{Definition~5.2}]
  A \(\Cst\)\nb-inclusion \(A\subseteq B\) is \emph{nondegenerate}
  if \(AB=B\).  It is \emph{symmetric} if for all
  \(I\in\Ideals^B(A)\) the inclusion \(I\to BIB\) is nondegenerate.
  (This is equivalent to \(IB = B I\) by
  \cite{Kwasniewski-Meyer:Stone_duality}*{Lemma~5.1}.)
\end{definition}

Let \(\check{A}\) and~\(\check{B}\) denote the primitive ideal
spaces of \(A\) and~\(B\), respectively.  Let \(\Prime^B(A)\) denote
the space of prime ideals in the lattice \(\Ideals^B(A)\) of
restricted ideals.  For \(I\in \Ideals^B(A)\), let
\(U_I\defeq \setgiven{\prid \in \Prime^B(A)}{I\not\subseteq
  \prid}\).  We equip \(\Prime^B(A)\) with the topology
\(\{U_I\}_{I\in \Ideals^B(A)}\).  
The next theorem summarises several desirable properties of
symmetric inclusions.

\begin{theorem}[\cite{Kwasniewski-Meyer:Stone_duality}]
  \label{thm:quasi_orbit_map}
  Let \(A\subseteq B\) be a symmetric \(\Cst\)\nb-inclusion.
  \begin{enumerate}
  \item \label{enu:quasi_orbit_map1} If~\(\prid\) is a primitive ideal in~\(B\), then
    \(r(\prid)\in \Ideals^B(A)\) is prime, and this defines a
    continuous map \(r\colon \check{B} \to \Prime^B(A)\).
  \item \label{enu:quasi_orbit_map2} If~\(\prid\)
    is a primitive ideal in~\(A\),
    then there is a largest restricted ideal in~\(A\)
    that is contained in~\(\prid\),
    which we denote by~\(\pi(\prid)\).
    This element of\/~\(\Ideals^B(A)\)
    is prime, and the resulting map
    \(\pi\colon\check{A} \to \Prime^B(A)\)
    is continuous.  Define an equivalence relation on~\(\check{A}\)
    by \(\prid\sim \prid[q]\)
    if and only if \(\pi(\prid)=\pi(\prid[q])\).
    So~\(\pi\)
    descends to a continuous map
    \(\widetilde{\pi}\colon \check{A}/{\sim}\to\Prime^B(A)\).
  \item \label{enu:quasi_orbit_map3} If\/ \(\Prime^B(A)\) is first countable --~this always holds
    when~\(\check{A}\) is second countable~-- then~\(\pi\) is open
    and surjective and~\(\widetilde{\pi}\) is a homeomorphism.  Then
    there is a continuous map
    \[
    \varrho\colon \check{B} \to \check{A}/{\sim},\qquad
    \prid\mapsto \widetilde{\pi}^{-1}(r(\prid)).
    \]
    It is a homeomorphism if and only if~\(A\) separates ideals
    in~\(B\).
  \end{enumerate}
\end{theorem}

\begin{proof}
  By \cite{Kwasniewski-Meyer:Stone_duality}*{Corollary~5.5} we may apply  \cite{Kwasniewski-Meyer:Stone_duality}*{Lemmas~4.8 and 4.1}
		 to get \ref{enu:quasi_orbit_map1} and \ref{enu:quasi_orbit_map2}, 
		and \cite{Kwasniewski-Meyer:Stone_duality}*{Theorem~4.5}
    gives \ref{enu:quasi_orbit_map3}, see also \cite{Kwasniewski-Meyer:Stone_duality}*{Corollary 4.7}.\end{proof}

\begin{definition}[\cite{Kwasniewski-Meyer:Stone_duality}*{Definitions 4.4, 4.10}] 
  The space~\(\check{A}/{\sim}\) in
  Theorem~\ref{thm:quasi_orbit_map} is called the \emph{quasi-orbit
    space} and the map
  \(\varrho\colon \check{B} \to \check{A}/{\sim}\) is called the
  \emph{quasi-orbit map} of \(A\subseteq B\).
\end{definition}

\begin{remark}
  If~\(B\) is the full or reduced crossed product --~or an exotic
  crossed product, for an action of a discrete group~$G$ on a
  \(\Cst\)\nb-algebra~\(A\), then~\(\check{A}/{\sim}\) coincides
  with the usual quasi-orbit space of the dual action of~\(G\)
  on~\(\check{A}\).  In~\cite{Kwasniewski-Meyer:Stone_duality}, the
  quasi-orbit space is also described in several other cases.
\end{remark}

Next we turn to regular inclusions.  To link them to crossed
products for inverse semigroup actions, we describe them through
gradings by inverse semigroups.

\begin{definition}[\cite{Kwasniewski-Meyer:Stone_duality}*{Definition 6.15}]
  \label{def:S-graded}
  Let~\(S\) be an inverse semigroup with unit \(1\in S\).  An
  \emph{\(S\)\nb-graded \(\Cst\)\nb-algebra} is a
  \(\Cst\)\nb-algebra~\(B\) with a family of closed linear subspaces
  \((B_t)_{t\in S}\) such that \(B_g^* = B_{g^*}\),
  \(B_g\cdot B_h \subseteq B_{g h}\) for all \(g,h\in S\) and
  \(B_g \subseteq B_h\) if \(g \le h\) in~\(S\) (that is,
  \(g = h g^* g\)), and \(\sum B_t\) is dense in~\(B\).  We call
  \(A\defeq B_1\subseteq B\) the \emph{unit fibre} of the
  \(S\)\nb-grading.  The grading is \emph{saturated} if
  \(B_g\cdot B_h = B_{g h}\) for all \(g,h\in S\).
\end{definition}

\begin{definition}[\cites{Kumjian:Diagonals, Renault:Cartan.Subalgebras}]
  Let \(A\subseteq B\) be a \(\Cst\)\nb-subalgebra.  We call
  \(b\in B\) a \emph{normaliser} of~\(A\) in~\(B\) if
  \(b A b^*\subseteq A\) and \(b^* A b\subseteq A\).  The inclusion
  \(A\subseteq B\) is \emph{regular} if it is nondegenerate
  and~\(B\) is the closed linear span of the normalisers of~\(A\)
  in~\(B\).
\end{definition}

\begin{proposition}
  \label{prop:regular_vs_inverse_semigroups}
  Let \(A\subseteq B\) be a \(\Cst\)\nb-inclusion.  The following
  are equivalent:
  \begin{enumerate}
  \item \label{enu:non_commutative_cartans1}%
    \(A\) is a regular subalgebra of~\(B\);
  \item \label{enu:non_commutative_cartans2}%
    \(A\) is the unit fibre for some \(S\)\nb-grading on~\(B\);
  \item \label{enu:non_commutative_cartans3}%
    \(A\)
    is the unit fibre for some saturated \(S\)\nb-grading on~\(B\).
  \end{enumerate}
  If the inclusion is regular, then it is symmetric, and the set
  \(\Slice(A,B)\) of all closed linear \(A\)\nb-\hspace*{0pt}subbimodules
  \(M\subseteq B\) that consist entirely of normalisers is an
  inverse semigroup with the operations
  \(M\cdot N\defeq \cl{\operatorname{span}} {}\setgiven{m n}{m \in
    M,\ n\in N}\) and \(M^*\defeq \setgiven{m^*}{m \in M}\), and it
  gives a saturated grading on~\(B\).
\end{proposition}

\begin{proof}
  Combine \cite{Kwasniewski-Meyer:Stone_duality}*{Lemma~6.25 and
    Proposition~6.26}.
\end{proof}

\begin{definition}
  Let \((B_t)_{t\in S}\) be an \(S\)\nb-grading on~\(B\) with
  \(A=B_1\).  An ideal \(I\in \Ideals(A)\) is called
  \emph{\((B_t)_{t\in S}\)\nb-invariant} if \(I B_t =B_t I\) for all
  \(t\in S\).  An ideal \(J\in \Ideals(B)\) is called
  \emph{\(S\)\nb-graded} if
  \(J=\cl{\sum_{t\in S} (J\cap B_t)}\).
\end{definition}

\begin{proposition}[\cite{Kwasniewski-Meyer:Stone_duality}*{Propositions
    6.19 and~6.20}]
  \label{prop:invariant_ideals_in_regular}
  Let \((B_t)_{t\in S}\) be an \(S\)\nb-grading on~\(B\) with
  \(A=B_1\).  An ideal \(I\in \Ideals(A)\) is restricted,
  \(I\in \Ideals^B(A)\), if and only if it is
  \((B_t)_{t\in S}\)\nb-invariant.  An ideal \(J\in \Ideals(B)\) is
  induced, \(J\in \Ideals^A(B)\), if and only if~\(J\) is
  \(S\)\nb-graded.
\end{proposition}

The following lemma on quotients of \(\Cst\)\nb-inclusions is not
yet considered in~\cite{Kwasniewski-Meyer:Stone_duality}.

\begin{lemma}
  \label{lem:quotients_of_regular}
  Let \(A\subseteq B\) be a \(\Cst\)\nb-subalgebra and
  \(J\in \Ideals(B)\).  Let \(I\defeq J\cap A\).  Let
  \(q\colon B\to B/J\) denote the quotient map.  View~\(A/I\)
  as a \(\Cst\)\nb-subalgebra of~\(B/J\).
  \begin{enumerate}
  \item \label{enu:quotients_of_inclusions4}%
    If \((B_t)_{t\in S}\) is an \(S\)\nb-grading of~\(B\) with~\(A\)
    as unit fibre, then \((q(B_t))_{t\in S}\) is an \(S\)\nb-grading
    on~\(B/J\) with unit fibre~\(A/I\).  And
    \(q(B_t)\cong B_t/B_tI\) as Banach spaces --~and even as Hilbert
    \(A/I\)-bimodules~-- for all \(t\in S\).
  \item \label{enu:quotients_of_inclusions3}%
    If \(A\subseteq B\) is regular, then \(A/I\subseteq B/J\) is
    regular.
  \end{enumerate}
\end{lemma}

\begin{proof}
  The canonical map from~\(A/I\) to~\(B/J\) is injective because
  \(J\cap A=I\).  Thus we may view~\(A/I\) as a
  \(\Cst\)\nb-subalgebra of~\(B/J\).  We
  prove~\ref{enu:quotients_of_inclusions4}.  It is easy to see that
  \((q(B_t))_{t\in S}\) is an \(S\)\nb-grading on~\(B/J\).
  Each~\(B_t\) is naturally a right Hilbert \(A\)\nb-module with
  inner product \(\braket{a}{b}\defeq a^* b\in A\) for
  \(a,b\in B_t\) and the right multiplication in~\(B\).  The proof
  of the Rieffel correspondence between ideals in~\(A\) and
  \(\Comp(B_t)\) shows that
  \(B_t I = \setgiven{b \in B_t}{\braket{b}{b}\in I}\).  The
  quotient Banach space \(B_t/B_tI\) is a right Hilbert
  \(A/I\)\nb-module with the induced multiplication and the inner
  product
  \(\braket{a +B_tI}{b+B_tI}\defeq \braket{a}{b} +I = q(a^* b)\in
  A/I\) for \(a,b\in B_t\).  We claim that the norm defined by this
  inner product is equal to the quotient norm on \(B_t/B_tI\).  To
  show this, let~\((u_n)_{n\in N}\) be an approximate unit
  for~\(I\).  Then
  \[
    \norm{q(b)}^2
    = \lim {}\norm{b-b u_i}^2
    = \lim {}\norm{(1-u_i)b^* b(1- u_i)}
    = \norm{q(b^* b)}
    = \norm{\braket{b+ B_t I}{b+B_t I}}.
  \]
  This finishes the proof of~\ref{enu:quotients_of_inclusions4}.
  Assertion~\ref{enu:quotients_of_inclusions3} follows
  from~\ref{enu:quotients_of_inclusions4} and
  Proposition~\ref{prop:regular_vs_inverse_semigroups}.
\end{proof}

\subsection{Generalised expectations}

\begin{definition}
  \label{def:gen_expectation}
  A \emph{generalised expectation} for a \(\Cst\)\nb-inclusion
  \(A \subseteq B\) consists of another \(\Cst\)\nb-inclusion
  \(A\subseteq \tilde{A}\) and a completely positive, contractive
  map \(B \to \tilde{A}\) that restricts to the identity map
  on~\(A\).  If \(\tilde{A} = A\), \(\tilde{A} = A''\), or
  \(\tilde{A}=\hull(A)\) is Hamana's injective envelope of~\(A\),
  then we speak of a \emph{conditional expectation}, a \emph{weak
    expectation}, or a \emph{pseudo-expectation}, respectively.
\end{definition}

Let \(E\colon B\to \tilde{A} \supseteq A\) be a generalised
expectation.  It is called \emph{faithful} if \(E(b^* b)=0\) for
some \(b\in B\) implies \(b=0\), \emph{almost
  faithful} if \(E((b c)^* b c)=0\) for all \(c\in B\) and some
\(b\in B\) implies \(b=0\), and \emph{symmetric} if
\(E(b^* b)=0\) for some \(b\in B\) implies \(E(b b^*)=0\).
The largest two-sided ideal in~\(B\) contained in~\(\ker E\) is
equal to
\[
  \Null_E
  \defeq\setgiven{b\in B}{E((b c)^*b c)=0 \text{ for all }c\in B}
  = \setgiven{b\in B}{E(x b y)=0 \text{ for all }x,y\in B}
\]
(see \cite{Kwasniewski-Meyer:Essential}*{Proposition~3.6}), and
\(\Null_E=0\) if and only if \(E\) is almost faithful.

Since \(E|_A=\Id_A\) and \(E|_{\Null_E} = 0\), it follows that
\(A\cap \Null_E = 0\).  Hence the composite map
\(A \to B \to B/\Null_E\) is injective and we may identify~\(A\)
with its image in \(B_\red\defeq B/\Null_E\).  The map~\(E\)
descends to a generalised expectation
\(E_\red \colon B_\red \to \tilde{A} \supseteq A\) that we call the
\emph{reduced generalised expectation} associated to~\(E\)
(see \cite{Kwasniewski-Meyer:Essential}*{Definition~3.5}).  The
reduced generalised expectation~\(E_\red\) is always almost
faithful.  It is faithful if and only if~\(E\) is symmetric (see
\cite{Kwasniewski-Meyer:Essential}*{Corollary~3.8}).

We will mainly work with pseudo-expectations below.  The injectivity
of \(\hull(A)\) implies that any \(\Cst\)\nb-inclusion has a
pseudo-expectation.  The following lemma links detection of ideals
to almost faithfulness of pseudo-expectations:

\begin{lemma}
  \label{lem:detection_ideals_almost_faithfulness}
  Let \(A\subseteq B\) be a \(\Cst\)\nb-inclusion.  The following
  are equivalent:
  \begin{enumerate}
  \item \label{enu:detection_ideals_almost_faithfulness1}%
    \(A\) detects ideals in~\(B\);
  \item \label{enu:detection_ideals_almost_faithfulness2}%
    every generalised expectation for the \(\Cst\)\nb-inclusion
    \(A\subseteq B\) is almost faithful;
  \item \label{enu:detection_ideals_almost_faithfulness3}%
    every pseudo-expectation for \(A\subseteq B\) is almost
    faithful.
  \end{enumerate}
\end{lemma}

\begin{proof}
  Let \(E\colon B\to \widetilde{A}\supseteq A\) be a generalised
  expectation.  Since \(\Null_E\cap A=0\), we must have
  \(\Null_E=0\) if~\(A\) detects ideals in~\(B\).  That is,
  \ref{enu:detection_ideals_almost_faithfulness1}
  implies~\ref{enu:detection_ideals_almost_faithfulness2}.
  That~\ref{enu:detection_ideals_almost_faithfulness2}
  implies~\ref{enu:detection_ideals_almost_faithfulness3} is
  obvious.  We prove by contradiction
  that~\ref{enu:detection_ideals_almost_faithfulness3}
  implies~\ref{enu:detection_ideals_almost_faithfulness1}.  Assume
  that~\(A\) does not detect ideals in~\(B\).  Then there is a
  nonzero ideal~\(\Null\) in~\(B\) with \(\Null\cap A=0\).  The
  inclusion \(A\hookrightarrow B/\Null\) has a
  pseudo-expectation \(E\colon B/\Null\to\hull(A)\).  Let
  \(q\colon B\to B/\Null\) be the quotient map.  Then
  \(E\circ q\colon B\to \hull(A)\) is a pseudo-expectation for the
  inclusion \(A\subseteq B\).  It is not almost faithful because
  \(0\neq \Null\subseteq \Null_{E\circ q}\).
\end{proof}

\begin{remark}[\cite{Pitts-Zarikian:Unique_pseudoexpectation}*{Theorem~3.5}]
  Every pseudo-expectation for \(A\subseteq B\) is faithful --~not
  only almost faithful~-- if and only if~\(A\) detects ideals
  in~\(C\) for each intermediate \(\Cst\)\nb-algebra
  \(A\subseteq C \subseteq B\).
\end{remark}

Lemma~\ref{lem:separate_induced} says that~\(A\) separates ideals
in~\(B\) if and only if it ``residually detects'' ideals in~\(B\).
The residual version of
Lemma~\ref{lem:detection_ideals_almost_faithfulness} says that~\(A\)
separates ideals in~\(B\) if and only if the following happens: if
\(I\in\Ideals^B(A)\) is a restricted ideal and
\(E^I\colon B/B I B \to \hull(A/I)\) is a pseudo-expectation for the
inclusion \(A/I \hookrightarrow B/B I B\), then~\(E^I\) is almost
faithful.

In general, it is not practical to check that \emph{all}
pseudo-expectations \(B\to \hull(A)\) are faithful or almost
faithful.  And the residual version of this statement looks even
more hopeless.  There are, however, inclusions with a unique
pseudo-expectation.  This
is the case for aperiodic inclusions by \cite{Kwasniewski-Meyer:Aperiodicity_pseudo_expectations}*{Theorem 3.6}.  When the inclusion is even
``residually'' aperiodic, then the inclusions
\(A/I \hookrightarrow B/B I B\) have a unique pseudo-expectation for
all \(I\in \Ideals^B(A)\).  And if we know pseudo-expectations
\(E^I\colon B/B I B \to \hull(A/I)\) for all \(I\in \Ideals^B(A)\),
then it becomes possible to check whether they are all (almost)
faithful (see Theorem~\ref{thm:residual_faithful+aperiodic_gives_supports} below).

The residual version of
Lemma~\ref{lem:detection_ideals_almost_faithfulness} discussed above
uses pseudo-expectations for the inclusion
\(A/I \hookrightarrow B/B I B\) for \(I\in\Ideals^B(A)\).  The
following examples show that these need not be closely related to
pseudo-expectations for the original inclusion \(A\subseteq B\).  In
fact, for inclusions that are not symmetric, even a genuine
conditional expectation \(E\colon B\to A\) need not ``induce'' a
conditional expectation \(E\colon B/ B I B\to A/I\).

\begin{example}
  \label{exa:symmetric_pseudo_does_not_induce}
  Let \(B=\Bound(H)\) be the algebra of bounded operators on a
  separable Hilbert space \(H\), and let \(A=\Comp(H) +1\) be the
  minimal unitisation of the compacts.  The inclusion
  \(A\subseteq B\) is symmetric,
  \(\Ideals^{B}(A)=\{0, \Comp(H), A\}\), and
  \(\hull(A)=\Bound(H)=B\).  We may take the identity map as the
  pseudo-expectation \(E\colon B\to B\) for \(A\subseteq B\).  Let
  \(I\defeq \Comp(H)\).  Then, on the one hand, \(E\) descends to
  the identity map \(E^I\colon B/I\to B/I\) on the Calkin algebra.
  On the other hand, pseudo-expectations for
  \(\C\cong A/I\subseteq B/I\) are just states on the Calkin
  algebra.  It seems that there is no universal way how to produce a
  state from the identity map.
\end{example}

\begin{example}[see \cite{Kwasniewski-Meyer:Stone_duality}*{Examples
    2.16 and~7.9}]   \label{exa:join_not_restricted}
	  Let \(B\defeq \Mat_2(\C)\oplus\Mat_2(\C)\) and consider
 the commutative \(\Cst\)\nb-subalgebra 	\(A\subseteq B\) spanned
  by the orthogonal diagonal projections \((P_{00},0)\),
  \((0,P_{00})\), and \((P_{11},P_{11})\).  Let
  \(E\colon B\to A\subseteq B\) be any faithful conditional
  expectation.  For instance,
  \[
  E\left(
    \begin{pmatrix}
      a_{00} & a_{01} \\
      a_{10} & a_{11}
    \end{pmatrix}
    \oplus
    \begin{pmatrix}
      b_{00} & b_{01} \\
      b_{10} & b_{11}
    \end{pmatrix}
  \right) \defeq
  \frac{1}{2}
  \begin{pmatrix}
    2 a_{00} & 0\\
    0 & a_{11}+b_{11}
  \end{pmatrix}
  \oplus
  \frac{1}{2}
  \begin{pmatrix}
    2 b_{00} &0 \\
    0 & a_{11}+b_{11}
  \end{pmatrix}.
  \]
  Let \(J=\Mat_2(\C)\oplus0\).  Then
  \(I\defeq J\cap A = \C\cdot (P_{00},0)\) is a restricted ideal in
  \(A\cong\C^3\).  We have \(J=B I B\) and
  \(E(J)=\setgiven{\lambda_1(P_{00},0) +\lambda_2(P_{11},P_{11})}
  {\lambda_1,\lambda_2 \in \C} \not\subseteq I\).  Hence
  \(E\colon B\to A\) does not factor through a map \(B/J\to A/I\).
	Note that the inclusion \(A\subseteq B\) is not symmetric.
      \end{example}

\begin{lemma}
  \label{lem:symmetric_quotients_conditional_expectations}
  Let \(E\colon B\to A\) be a conditional expectation for a
  symmetric \(\Cst\)\nb-inclusion \(A\subseteq B\) and let
  \(I\in \Ideals^B(A)\).  Then~\(E\) descends to a conditional
  expectation \(E^I\colon B/BIB \to A/I\),
  \(b+ J\mapsto E(b) + I \), for the inclusion \(A/I \to B/BIB\).
\end{lemma}

\begin{proof}
  Let \(J\in\Ideals^A(B)\)
  and put \(I=J\cap A\in\Ideals^B(A)\).
Since  \(A\subseteq B\) is symmetric we have \(J=I B I\).
  Thus \(E(J) = E(I B I)\subseteq I E(B) I =I\)
  because~\(E\)
  is \(A\)\nb-bilinear.
  Since \(I\subseteq E(J)\)
  always holds, this is equivalent to \(I=E(J)\).  Then~\(E^I\) is
  well defined.
\end{proof}

Many \(\Cst\)\nb-algebras, including section algebras for Fell bundles
over Hausdorff \'etale groupoids, are naturally equipped with a
conditional expectation, which is residually symmetric in the sense
described in the following proposition.  Then the residual
faithfulness of~\(E\) is called exactness of the corresponding
action \cites{Sierakowski:IdealStructureCrossedProducts,
  Abadie-Abadie:Ideals, Kwasniewski-Szymanski:Pure_infinite,
  AnantharamanDelaroch:Weak_containment, Boenicke-Li:Ideal}.

\begin{proposition}
  \label{prop:separate_ideal_vs_conditional_exp}
  Let \(A\subseteq B\) be a symmetric \(\Cst\)\nb-inclusion with a
  conditional expectation \(E\colon B\to A\) which is residually
  symmetric, that is, for each \(I\in \Ideals^B(A)\), the
  conditional expectation~\(E^I\) in
  Lemma~\textup{\ref{lem:symmetric_quotients_conditional_expectations}}
  is symmetric.  Then~\(A\) separates ideals in~\(B\) if and only
  if~\(E\) is residually faithful --~that is, \(E^I\) is faithful
  for each \(I\in \Ideals^B(A)\)~-- and~\(E\) preserves ideals
  --~that is, \(E(J)\subseteq J\) for each \(J\in \Ideals(B)\).
\end{proposition}
\begin{proof}
  Assume first that~\(A\) separates ideals in~\(B\).  Lemmas
  \ref{lem:separate_induced}
  and~\ref{lem:detection_ideals_almost_faithfulness} imply
  that~\(E^I\) is faithful for all \(I\in \Ideals^B(A)\).  Since we
  assume that \(\Ideals(B)=\Ideals^A(B)\) and \(A\subseteq B\) is
  symmetric, this implies \(E(J)\subseteq J\) by
  Lemma~\ref{lem:symmetric_quotients_conditional_expectations}.
  Conversely, assume that~\(E\) is residually faithful and preserves ideals.

  Let \(J\in \Ideals(B)\) and \(I\in\Ideals(A)\) satisfy
  \(BIB \subsetneq J\).  Then \(E(J) \neq I\) because~\(E^I\) is
  faithful.  Since~\(E\) preserves ideals, \(E(J)=J\cap A\).  Thus
  \(I \subsetneq J\cap A\).  This shows that \(A/I \subseteq B/BIB\)
  detects ideals for any \(I\in\Ideals^B(A)\).  This implies
  that~\(A\) separates ideals in~\(B\) by
  Lemma~\ref{lem:separate_induced}.
\end{proof}

\subsection{Supporting and filling families}
\label{subsec:filling_aperiodic_conditional_expectations}

Let~\(B^+\) be the set of positive elements in a
\(\Cst\)\nb-algebra~\(B\).  We equip~\(B^+\) with the
\emph{Cuntz preorder}~\(\precsim\) introduced
in~\cite{Cuntz:Dimension_functions}: for
\(a, b\in B^+\), we write \(a \precsim b\) and say
that~\(a\) \emph{supports}~\(b\) if, for every \(\varepsilon>0\),
there is \(x \in B\) with \(\norm{a-x^* b x} <\varepsilon\).  We
call \(a,b\in B^+\) \emph{Cuntz equivalent} and write
\(a\approx b\) if \(a \precsim b\) and \(b\precsim a\).  We call
\(a,b\in B^+\) \emph{Murray--von Neumann equivalent}
and write \(a\sim b\) if there is \(z\in B\) with \(a=z^*z\) and
\(b=zz^*\).  Both \(\sim\) and \(\approx\) are equivalence
relations, and \(a\sim b\) implies \(a\approx b\).  In the converse
direction, only a weaker result is true (see
\cite{Kirchberg-Rordam:Infinite_absorbing}*{Lemma 2.3(iv)}).
Namely, for \(\varepsilon >0\) and \(a\in B^+\setminus\{0\}\), let
\((a-\varepsilon)_+\in B\) be the positive part of
\(a-\varepsilon\cdot 1\in \Mult(B)\).  Let \(b B b\) be the
hereditary subalgebra generated by~\(b\), that is, the closure of
\(\setgiven{b x b}{x\in A}\).  Then \(a \precsim b\) if and only if
every \(\varepsilon\)\nb-cut-down of~\(a\) is Murray--von Neumann
equivalent to an element in~\(b B b\), that is
\[
 a \precsim b \quad \Longleftrightarrow\quad
  \forall_{\varepsilon >0}\ \exists_{z\in B}\quad
  (a-\varepsilon)_+=z^*z\ \text{and}\ zz^*\in b B b.
\]
In particular, \(a\in b B b\) implies \(a\precsim b\),
\(a B a = b B b\) implies \(a\approx b\), and \(a\precsim b\) implies
\(a\in B b B\).  A \(\Cst\)\nb-algebra~\(B\) is \emph{purely
  infinite} \cite{Kirchberg-Rordam:Non-simple_pi} if it admits no
characters and \(a,b\in B^+\setminus\{0\}\) satisfy \(a \preceq b\)
if and only if \(a\in B b B\).

\begin{definition}[compare \cite{Kwasniewski:Crossed_products}*{Definition~2.39}]
  \label{def:support}
  A subset \(\mathcal{F}\subseteq B^+\) \emph{supports~\(B\)} if,
  for each \(b \in B^+ \setminus \{0\}\), there is
  \(a\in \mathcal{F} \setminus \{0\}\) with \(a\precsim b\).  It
  \emph{residually supports~\(B\)} if, for each \(J\in \Ideals(B)\),
  the image of~\(\mathcal{F}\) in the quotient~\(B/J\)
  supports~\(B/J\).
\end{definition}

\begin{lemma}
  \label{lem:support_implies_separate}
  Let~\(B\) be a \(\Cst\)\nb-algebra and let
  \(\mathcal{F}\subseteq B^+\).
  \begin{enumerate}
  \item \label{lem:support_implies_separate1}%
    Let~\(\mathcal{F}\) support~\(B\).  If \(J \in \Ideals(B)\),
    then \(J\cap \mathcal{F}\) supports~\(J\).  So~\(\mathcal{F}\)
    detects ideals in~\(B\).
  \item \label{lem:support_implies_separate2}%
    Let~\(\mathcal{F}\) residually support~\(B\).  If
    \(J\in \Ideals(B)\), then \(J\cap \mathcal{F}\) residually
    supports~\(J\).  In addition, \(\mathcal{F}\) separates ideals
    in~\(B\).
  \end{enumerate}
\end{lemma}

\begin{proof}
  We first prove~\ref{lem:support_implies_separate1}.  Let
  \(J\in\Ideals(B)\) and \(b\in J^+\setminus\{0\}\).  Then
  \(x^* b x\in J\) for each \(x\in B\).  If \(a\precsim b\) for some
  \(a\in B^+\setminus\{0\}\), then \(a = \lim x_k^* b x_k\) for some
  sequence~\((x_k)_{k\in\N}\) in~\(B\).  Therefore, any
  \(a\in \mathcal{F}^+\setminus\{0\}\) with \(a\preceq b\) and
  \(b\in J^+\setminus\{0\}\) belongs to~\(J\).  So
  if~\(\mathcal{F}\) supports~\(B\) and \(J\neq\{0\}\), then
  \(J\cap\mathcal{F} \neq\{0\}\).

  Next we prove~\ref{lem:support_implies_separate2}.  Let
  \(J\in\Ideals(B)\) and \(I\in\Ideals(J)\).  Let
  \(q\colon B\to B/I\) be the quotient map.  For any
  \(b\in q(J)^+\setminus \{0\}\), there is \(a\in \mathcal{F}\) with
  \(0\neq q(a)\precsim b\).  The proof above shows that
  \(q(a)\in q(J)\).  Since \(q(a)\neq0\), we have
  \(a\in J\setminus I\).  Thus \(J\cap \mathcal{F}\) residually
  supports~\(J\).  If \(I, J\in \Ideals(B)\) and \(I\neq J\), then
  \(I\cap J\) is a proper ideal of \(I\) or~\(J\).  Assume, say,
  that \(I\cap J\subsetneq J\).  As we have seen above, there is
  \(a\in (J \cap\mathcal{F})\setminus I\).  Hence
  \(I\cap \mathcal{F}\neq J\cap\mathcal{F}\).
\end{proof}

\begin{corollary}
  \label{cor:residually_support_only_induced}
  Let \(A\subseteq B\) be a \(\Cst\)\nb-inclusion.  A family
  \(\mathcal{F}\subseteq A^+\) residually supports~\(B\) if and only
  if the image of~\(\mathcal{F}\) supports~\(B/J\) for each
  \(J\in \Ideals^A(B)\).
\end{corollary}

\begin{proof}
  If the image of~\(\mathcal{F}\) supports the quotient~\(B/J\) for
  every \(J\in \Ideals^A(B)\), then by
  Lemma~\ref{lem:support_implies_separate}, \(A/(J\cap A)\) detects
  ideals in~\(B/J\) for each \(J\in \Ideals^A(B)\).  Then~\(A\)
  separates ideals in~\(B\) by Lemma~\ref{lem:separate_induced}.
  Thus \(\Ideals(B)=\Ideals^A(B)\).  So our assumption already says
  that~\(\mathcal{F}\) residually supports~\(B\).
\end{proof}

A \(\Cst\)\nb-algebra~\(B\) has the \emph{ideal property} if
projections in~\(B\) separate ideals in~\(B\).

\begin{corollary}
  \label{cor:supporting_projection}
  If the set of projections in~\(B\) residually supports~\(B\),
  then~\(B\) has the ideal property.  Conversely, a purely infinite
  \(\Cst\)\nb-algebra with the ideal property is residually
  supported by its projections.
\end{corollary}

\begin{proof}
  Lemma
  \ref{lem:support_implies_separate}.\ref{lem:support_implies_separate2}
  gives the first statement.  For the second, let \(J\in\Ideals(B)\)
  and \(b\in B^+\setminus J^+\).  Let \(q\colon B\to B/J\) be the
  quotient map.  Since~\(B\) has the ideal property, the ideal
  \(BbB+J\supsetneq J\) contains a projection \(p\notin J\).  Then
  \(0 \neq q(p)\in q(BbB)=q(B)q(b)q(B)\).  The quotient~\(B/J\) is
  purely infinite by
  \cite{Kirchberg-Rordam:Non-simple_pi}*{Proposition 4.3}.  Hence we
  get \(q(p)\preceq q(b)\).
\end{proof}




Let \(\Her(B)\) be the set of all \emph{nonzero hereditary
  \(\Cst\)\nb-subalgebras} of~\(B\).

\begin{lemma}
  \label{lem:support_vs_filling}
  Let \(\mathcal{F}\subseteq B^+\).  The following conditions are
  equivalent:
  \begin{enumerate}
  \item \label{lem:support_vs_filling2}%
    \(\mathcal{F}\) supports~\(B\);
  \item \label{lem:support_vs_filling0}%
    for each \(b\in B^+\setminus\{0\}\), there is \(x\in B\) with
    \(x^*bx\in \mathcal{F}\setminus\{0\}\);
  \item \label{lem:support_vs_filling1}%
    for each \(D\in\Her(B)\), there is \(z\in B\) with
    \(z z^* \in D\) and \(z^* z \in \mathcal{F}\setminus\{0\}\).
  \end{enumerate}
\end{lemma}

\begin{proof}
  \ref{lem:support_vs_filling2}\(\Rightarrow\)%
  \ref{lem:support_vs_filling0}: Let \(b\in B^+\setminus\{0\}\).
  Let \(\delta \in (0, \norm{b})\).
  Then~\ref{lem:support_vs_filling2} gives
  \(a\in \mathcal{F}\setminus\{0\}\) with
  \(a\precsim (b-\delta)_+\).  And
  \cite{Kirchberg-Rordam:Infinite_absorbing}*{Lemma~2.4(ii)} gives
  \(x\in B\) with \(x^*bx=a\in \mathcal{F}\setminus\{0\}\).
	
  \ref{lem:support_vs_filling0}\(\Rightarrow\)%
  \ref{lem:support_vs_filling1}: Let \(b\in D^+\setminus\{0\}\)
  and choose \(x\in B\)
  with \(x^*bx\in \mathcal{F}\setminus\{0\}\)
  as in~\ref{lem:support_vs_filling0}.  Then
  \(z\defeq \sqrt{b} \cdot x\)
  satisfies \(z z^* \in D\)
  and \(z^* z \in \mathcal{F}\setminus\{0\}\),
  as required in~\ref{lem:support_vs_filling1}.

  \ref{lem:support_vs_filling1}\(\Rightarrow\)%
  \ref{lem:support_vs_filling2}: Let \(b\in B^+\setminus\{0\}\).
  Then \(D\defeq b B b\) is a nonzero hereditary
  \(\Cst\)\nb-subalgebra of~\(B\).
  Condition~\ref{lem:support_vs_filling1} gives \(z\in B\) with
  \(z z^* \in D\) and
  \(a\defeq z^* z \in \mathcal{F}\setminus\{0\}\).  Then
  \(a\sim z z^*\) and hence \(a\approx z z^*\).  And
  \(z z^* \precsim b\) by
  \cite{Kirchberg-Rordam:Non-simple_pi}*{Proposition 2.7(i)}.  So
  \(a \precsim b\).
\end{proof}

\begin{definition}[\cite{Kirchberg-Sierakowski:Filling_families}*{Definition~4.2}]
  \label{def:filling_family}
  Let~\(B\) be a \(\Cst\)\nb-algebra.  A set
  \(\mathcal{F}\subseteq B^+\) \emph{fills}~\(B\) (is a
  \emph{filling family} for~\(B\)) if, for each \(D\in \Her(B)\) and
  each \(J\in \Ideals(B)\) with \(D\not\subseteq J\), there is
  \(z\in B\setminus J\) with \(zz^*\in D\) and
  \(z^*z \in \mathcal{F}\).
\end{definition}

In this definition, we may also require \(z\in B\) with
\(zz^*\in D\) and \(z^*z \in \mathcal{F}\setminus J\) because
\(z\in J\) if and only if \(z^* z \in J\).  For \(J=0\), this is the
condition in Lemma
\ref{lem:support_vs_filling}.\ref{lem:support_vs_filling1}.  This
suggests that filling and residually supporting families are closely
related.  We are going to prove some results to this effect.

\begin{proposition}
  \label{pro:filling_residually_supports}
  If \(\mathcal{F}\subseteq B^+\)
  fills~\(B\), then it residually supports~\(B\).
\end{proposition}

\begin{proof}
  Let \(J\in \Ideals(B)\).  Let \(q\colon B\to B/J\) denote the
  quotient map.  Let \(b\in q(B)^+\setminus \{0\}\).  There is
  \(d\in B^+\setminus J\) with \(q(d)=b\).  For \(D \defeq d B d\)
  we have \(q(D)= b q(B) b\) and \(D\not\subseteq J\).
  Since~\(\mathcal{F}\) fills~\(B\), there is \(z\in B\) with
  \(z^*z\in D\) and \(zz^* \in \mathcal{F} \setminus J\).  Then
  \(q(z)^*q(z)\in b q(B) b\) and
  \(q(z)q(z)^*\in q(\mathcal{F})\setminus \{0\}\).  Hence
  \(q(z)q(z)^* \approx q(z)^*q(z)\precsim b\).  So
  \(q(\mathcal{F})\) supports~\(q(B)\).
\end{proof}

\begin{proposition}
  \label{prop:residual_supp_vs_filling}
  Let \(A\subseteq B\) be a symmetric \(\Cst\)\nb-inclusion.
  Then~\(A^+\) residually supports~\(B\) if and only if~\(A^+\)
  fills~\(B\).
\end{proposition}

\begin{proof}
  If~\(A^+\) fills~\(B\), then it residually supports~\(B\) by
  Proposition~\ref{pro:filling_residually_supports}.  Conversely,
  assume that~\(A^+\) residually supports~\(B\).  We are going to
  prove that~\(A^+\) fills~\(B\).  Pick \(D\in\Her(B)\) and
  \(J\in\Ideals(B)\) with \(D\not\subseteq J\).  We need \(z\in A\)
  with \(z^*z\in D\) and \(zz^* \in A^+ \setminus J\).  By Lemma
  \ref{lem:support_implies_separate}.\ref{lem:support_implies_separate2},
  \(A\) separates ideals in~\(B\).  Hence
  \(\Ideals^A(B)=\Ideals(B)\).  So \(J=B I B\) with
  \(I\defeq A\cap J\).  Let \(q\colon B\to B/J\) be the quotient
  map.  There is \(d\in D^+\setminus J\).  Let
  \(b\defeq q(d)\in (B/J)^+\setminus\{0\}\).
  Lemma~\ref{lem:support_vs_filling} gives \(x\in B /J\) with
  \(a\defeq x^*bx \in (A/I)^+ \setminus \{0\}\).  There are
  \(c\in A^+\) with \(q(c)=a\) and \(w\in B\) with \(q(w)=x\).  Then
  \(q(c) = x^* b x = q(w^* d w)\).  So \(c=w^* d w + v\) for some
  \(v\in J\).

  Let \(\varepsilon\defeq \norm{a}/2\).  By assumption, an
  approximate unit in~\(I\) is also one for~\(J\).  So there is
  \(f\in I^+\) with \(\norm{f}\le1\) and
  \(\norm{v-fv}<\varepsilon\).  Let~\(1\) denote the formal unit in
  the unitisation of~\(B\) and let \(g\defeq 1-f\in \Mult(A)^+\).
  Then \(\norm{g}\le 1\) and
  \[
    \norm{g w^* d w g - g c g}
    = \norm{g v g}
    \le \norm{v-f v}
    < \varepsilon.
  \]
  Now \cite{Kirchberg-Rordam:Infinite_absorbing}*{Lemma~2.2} gives
  \(h\in B\) with
  \[
    h^* (g w^* d w g) h = (g c g - \varepsilon)_+ \in  A^+.
  \]
  Let \(z\defeq d^{\nicefrac12} w g h\).  Then \(z z^*\in D\)
  because \(d\in D\), and \(z^* z = (g c g-\varepsilon)_+ \in A^+\).
  Since \(q(g c g) = q(c-c f-f c+ f c f) = q(c) = a\), we get
  \[
    \norm{q(z^* z)}
    = \norm{q\big((g c g-\varepsilon)_+\big)}
    = \norm{(a-\varepsilon)_+}
    \ge \norm{a} - \varepsilon
    = \norm{a}/2
    > 0.
  \]
  Hence \(z^* z\notin J\), which is equivalent to \(z z^*\notin J\).
\end{proof}

\begin{example}
  \label{exa:noncommutative_bases}
  Let \(B=\Cont_0(\Omega)\) be commutative and let
  \(\mathcal{F}\subseteq B^+\).  The following conditions are
  equivalent:
  \begin{enumerate}
  \item \label{en:noncommutative_bases1}%
    \(\mathcal{F}\) fills~\(B\);
  \item \label{en:noncommutative_bases2}%
    \(\mathcal{F}\) residually supports~\(B\);
  \item \label{en:noncommutative_bases3}%
    the open supports of elements of~\(\mathcal{F}\)
    form a basis of the topology of~\(\Omega\).
  \end{enumerate}
  Proposition~\ref{pro:filling_residually_supports} implies
  \ref{en:noncommutative_bases1}\(\Rightarrow\)%
  \ref{en:noncommutative_bases2}, and
  \ref{en:noncommutative_bases3}\(\Rightarrow\)%
  \ref{en:noncommutative_bases1} is straightforward.  We show
  \ref{en:noncommutative_bases2}\(\Rightarrow\)%
  \ref{en:noncommutative_bases3}.  Fix an open subset
  \(U\subseteq \Omega\) and a point \(x_0 \in U\).  There is a
  function \(b\in \Cont_0(\Omega)\) with \(b(x_0)=1\) and
  \(b|_{\Omega\setminus U}\equiv 0\).  Let~\(J\) be the ideal
  in~\(B\) consisting of functions vanishing on
  \(\Omega\setminus U \cup \{x_0\}\).  There is
  \(a \in \mathcal{F}\setminus J\) with \(a+J\precsim b+ J\).  Then
  \(V\defeq \setgiven{x\in \Omega }{a(x)>0}\) is an open subset
  of~\(U\) that contains~\(x_0\).  This
  implies~\ref{en:noncommutative_bases3}.
\end{example}

Example~\ref{exa:noncommutative_bases} suggests to view filling
families and residually supporting subsets as noncommutative
analogues of bases for topologies.

\subsection{Residual aperiodicity and criteria for pure
  infiniteness}
\label{subsec:residual_aperiodicity}

We introduce the residual version of aperiodicity and use it to
characterise when a \(\Cst\)\nb-inclusion \(A \subseteq B\)
separates ideals.  We also formulate some criteria for~\(B\) to be
purely infinite.

\begin{definition}[\cite{Kwasniewski-Meyer:Essential}*{Definition 5.14}]
  \label{def:aperiodic_inclusion}
  A \(\Cst\)\nb-inclusion \(A\subseteq B\) is \emph{aperiodic} if
  the Banach \(A\)\nb-bimodule \(B/A\) is aperiodic, that is, if for
  every \(x\in B\), \(D\in \Her(A)\) and \(\varepsilon>0\), there
  are \(a\in D^+\) and \(y\in A\) with
  \(\norm{a x a - y}<\varepsilon\) and \(\norm{a}=1\).
\end{definition}

\begin{remark}[\cite{Kwasniewski-Meyer:Aperiodicity_pseudo_expectations}*{Theorem~5.5}]
  If a \(\Cst\)\nb-inclusion \(A\subseteq B\) has the almost
  extension property introduced
  in~\cite{Nagy-Reznikoff:Pseudo-diagonals}, then \(A\subseteq B\)
  is aperiodic.  The converse implication holds if~\(B\) is
  separable.
\end{remark}

\begin{definition}
  \label{def:residual_aperiodic_inclusion}
  A \(\Cst\)\nb-inclusion \(A\subseteq B\) is \emph{residually
    aperiodic} if, for all \(I\in\Ideals^B(A)\) and \(J\defeq BIB\),
  the \(\Cst\)\nb-inclusion \(A/I\subseteq B/J\) is aperiodic.
\end{definition}

\begin{theorem}
  \label{thm:residual_faithful+aperiodic_gives_supports}
  Let \(A\subseteq B\) be a residually aperiodic
  \(\Cst\)\nb-inclusion.  Then for each \(I\in \Ideals^B(A)\) there
  is a unique pseudo-expectation \(E^I\colon B/BIB\to \hull(A/I)\)
  for the \(\Cst\)\nb-inclusion \(A/I\subseteq B/BIB\), and the
  following are equivalent:
  \begin{enumerate}
  \item \label{it:residual_faithful+aperiodic_gives_supports1}%
    the unique pseudo-expectation~\(E^I\) is almost faithful for all
    \(I\in \Ideals^B(A)\);
  \item \label{it:residual_faithful+aperiodic_gives_supports2}%
    \(A\) separates ideals in~\(B\);
  \item \label{it:residual_faithful+aperiodic_gives_supports3}%
    \(A^+\) residually supports~\(B\).
  \end{enumerate}
  If, in addition, the inclusion \(A\subseteq B\) is symmetric, then
  \ref{it:residual_faithful+aperiodic_gives_supports1}--\ref{it:residual_faithful+aperiodic_gives_supports3}
  are equivalent to
  \begin{enumerate}[resume]
  \item \label{it:residual_faithful+aperiodic_gives_supports4}%
    \(A^+\) fills~\(B\).
  \end{enumerate}
  If \(A\subseteq B\) is symmetric, \(\check{A}\) is second
  countable, and the above equivalent conditions hold, then
  \(\check{B}\cong \check{A}/{\sim}\) via the quasi-orbit map.
\end{theorem}

\begin{proof}
  It follows from
  \cite{Kwasniewski-Meyer:Aperiodicity_pseudo_expectations}*{Theorem~3.6}
  that the expectation~\(E^I\) is unique and
  that~\ref{it:residual_faithful+aperiodic_gives_supports1}
  implies~\ref{it:residual_faithful+aperiodic_gives_supports3}.
  Lemma~\ref{lem:support_implies_separate} shows that
  \ref{it:residual_faithful+aperiodic_gives_supports3}
  implies~\ref{it:residual_faithful+aperiodic_gives_supports2}.
  Next we prove
  that~\ref{it:residual_faithful+aperiodic_gives_supports2}
  implies~\ref{it:residual_faithful+aperiodic_gives_supports1}.
  Indeed, if~\(E^I\) is not almost faithful,
  then~\(\Null_{E^I}\) witnesses that~\(A/I\) does not detect
  ideals in~\(B/J\).  Then~\(A\) does not separate ideals in~\(B\)
  by Lemma~\ref{lem:separate_induced}.  In the symmetric case,
  Proposition~\ref{prop:residual_supp_vs_filling} shows
  that~\ref{it:residual_faithful+aperiodic_gives_supports3} is
  equivalent
  to~\ref{it:residual_faithful+aperiodic_gives_supports4}.  The
  claims in the last sentence follow from
  Theorem~\ref{thm:quasi_orbit_map}.
\end{proof}

We end this section with pure infiniteness criteria that use filling
and residually supporting families.  Infinite and properly infinite
elements in~\(B^+\) are defined
in~\cite{Kirchberg-Rordam:Non-simple_pi}.  We recall their
equivalent descriptions in
\cite{Kwasniewski-Szymanski:Pure_infinite}*{Lemma~2.1}.  We also
recall the notion of a separated pair of elements in~\(B^+\) from
\cite{Kwasniewski-Meyer:Aperiodicity}*{Definition~5.1} and relate it
to the matrix diagonalisation property.  We write
\(a\approx_\varepsilon b\) if \(\norm{a-b}\le \varepsilon\) for
\(a,b \in B\).

\begin{definition}
  \label{def:kinds_of_elements}
  Let~\(B\) be a \(\Cst\)\nb-algebra and \(a\in B^+\setminus\{0\}\).
  \begin{enumerate}
  \item We call \(a\in B^+\) \emph{infinite} in~\(B\) if there is
    \(b\in B^+\setminus\{0\}\) such that for all \(\varepsilon >0\)
    there are \(x,y\in a B\) with \(x^*x\approx_\varepsilon a\),
    \(y^*y\approx_\varepsilon b\) and \(x^*y\approx_\varepsilon 0\).
  \item We call \(a\in B^+\setminus\{0\}\) \emph{properly infinite}
    if for all \(\varepsilon >0\) there are \(x,y\in a B\) with
    \(x^*x\approx_\varepsilon a\), \(y^*y\approx_\varepsilon a\) and
    \(x^*y\approx_\varepsilon 0\).
  \item We call \(a,b\in B^+\) \emph{separated in~\(B\)} if for all
    \(\varepsilon >0\) there are \(x\in a B\) and \(y\in bB\) with
    \(x^*x\approx_\varepsilon a\), \(y^*y\approx_\varepsilon b\) and
    \(x^*y\approx_\varepsilon 0\).
  \end{enumerate}
\end{definition}

By \cite{Kirchberg-Rordam:Non-simple_pi}*{Theorem~4.16}, a
\(\Cst\)\nb-algebra is purely infinite if and only if each element
\(a\in B^+\setminus\{0\}\) is properly infinite.  By
\cite{Kirchberg-Rordam:Infinite_absorbing}*{Remark~5.10}, \(B\) is
\emph{strongly purely infinite} if and only if each pair of elements
\(a,b\in B^+\setminus\{0\}\) is separated in~\(B\).

We say that a pair of elements \(a,b\in B^+\) \emph{has the matrix
  diagonalisation property in}~\(B\), if for each \(x\in B\) with
\(\left(\begin{smallmatrix} a&x^*\\x&b \end{smallmatrix}\right) \in
M_2(B)^+\) and each \(\varepsilon>0\) there are \(d_1\in B\) and
\(d_2\in B\) such that
\[
  d_1^*a d_1 \approx_\varepsilon a,\qquad
  d_2^* b d_2\approx_\varepsilon b,\qquad
  d_1^* x d_2\approx_\varepsilon 0.
\]
We say that a subset \(\mathcal{F}\subseteq B^+\) is \emph{invariant
  under \(\varepsilon\)\nb-cut-downs} if
\((a-\varepsilon)_+\in \mathcal{F}\) for all \(a\in \mathcal{F}\)
and arbitrarily small \(\varepsilon > 0\).

\begin{theorem}[\cite{Kirchberg-Sierakowski:Filling_families}*{Theorem~1.1}]
  \label{thm:Kirchberg_Sierakowski}
  Let~\(\mathcal{F}\) fill~\(B\) and be invariant under
  \(\varepsilon\)\nb-cut-downs.  Then~\(B\) is strongly purely
  infinite if and only if each pair of elements
  \(a,b\in \mathcal{F}\) has the matrix diagonalisation property
  in~\(B\).
\end{theorem}

The following theorem improves upon
\cite{Kwasniewski-Meyer:Aperiodicity}*{Proposition~5.4}.

\begin{theorem}
  \label{thm:pure_infiniteness_criteria}
  Let \(A\subseteq B\) be a \(\Cst\)\nb-subalgebra for which~\(A^+\)
  residually supports~\(B\); this is the case, for instance, if
  \(A\subseteq B\) is residually aperiodic and for each
  \(I\in \Ideals^B(A)\) the unique pseudo-expectation
  \(E^I\colon B/BIB\to \hull(A/I)\) is almost faithful.  Let
  \(\mathcal{F}\subseteq A^+\) residually support~\(A\).  Assume
  that \(\Ideals^B(A)\) is finite or that the projections in
  \(\mathcal{F}\) separate the ideals in~\(\Ideals^B(A)\)
  \textup{(}this is automatic when~\(\mathcal{F}\) consists of
  projections\textup{)}. Then the following statements are
  equivalent:
  \begin{enumerate}
  \item \label{item:pure_infiniteness_criteria1}%
    \(\mathcal{F}\setminus\{0\}\) consists of elements that are
    properly infinite in~\(B\);
  \item \label{item:pure_infiniteness_criteria2}%
    \(B\) is purely infinite;
  \item \label{item:pure_infiniteness_criteria3}%
    \(B\) is purely infinite and \(\Prim B\) has topological
    dimension zero;
  \item \label{item:pure_infiniteness_criteria4}%
    \(B\) is purely infinite and has the ideal property;
  \item \label{item:pure_infiniteness_criteria6}%
    \(B\) is strongly purely infinite.
  \end{enumerate}
\end{theorem}

\begin{proof}
  The implications
  \ref{item:pure_infiniteness_criteria3}\(\Leftarrow\)%
  \ref{item:pure_infiniteness_criteria4}\(\Rightarrow\)%
  \ref{item:pure_infiniteness_criteria6} are general facts (see
  \cite{Pasnicu-Rordam:Purely_infinite_rr0}*{Propositions 2.11 and
    2.14}).  The implications
  \ref{item:pure_infiniteness_criteria3}\(\Rightarrow\)%
  \ref{item:pure_infiniteness_criteria2}\(\Rightarrow\)%
  \ref{item:pure_infiniteness_criteria1}\(\Leftarrow\)%
  \ref{item:pure_infiniteness_criteria6} are clear.  To close the
  cycles of implications, it suffices to show that
  \ref{item:pure_infiniteness_criteria1}\(\Rightarrow\)%
  \ref{item:pure_infiniteness_criteria4}.  We have
  \(\Ideals(B)=\Ideals^A(B)\cong \Ideals^B(A)\) by
  Lemma~\ref{lem:support_implies_separate}.  And~\(\mathcal{F}\)
  residually supports~\(B\) because~\(\preceq\) is transitive.  In
  particular, \(\mathcal{F}\) separates ideals in \(\Ideals^B(A)\).
  Under our assumptions, the implication
  \ref{item:pure_infiniteness_criteria1}\(\Rightarrow\)%
  \ref{item:pure_infiniteness_criteria4} follows from the proof of
  \cite{Kwasniewski-Meyer:Aperiodicity}*{Proposition~5.4}.  We
  considered there the case \(\mathcal{F}=A^+\).  The proof still
  works, however, for any \(\mathcal{F}\) that residually
  supports~\(A\).
\end{proof}

When~\(A\) separates ideals in~\(B\), the assumption that~\(A^+\)
residually supports~\(B\) seems necessary for
Theorem~\ref{thm:pure_infiniteness_criteria} to hold.  Evidence for
this is the following result abstracted from
\cite{Rordam-Sierakowski:Purely_infinite}*{Proposition~2.1} (see
also \cite{Boenicke-Li:Ideal}*{Proposition~4.1}).

\begin{proposition}
  \label{prop:separation_pure_infiniteness}
  Let \(A\subseteq B\) be a symmetric \(\Cst\)\nb-inclusion with a
  residually symmetric conditional expectation \(E\colon B\to A\) as
  in
  Proposition~\textup{\ref{prop:separate_ideal_vs_conditional_exp}}.
  If~\(A\) separates ideals in~\(B\), then \(B\) is purely infinite
  if and only if every \(a\in A^+\setminus\{0\}\) is properly
  infinite in~\(B\) and \(E(b)\precsim b\) for every \(b\in B^+\).
  If~\(B\) is purely infinite and~\(A\) separates ideals in~\(B\),
  then~\(A^+\) fills~\(B\).
\end{proposition}

\begin{proof}
  Assume that~\(A\) separates ideals in~\(B\).  Suppose first
  that~\(B\) is purely infinite.  By
  Proposition~\ref{prop:separate_ideal_vs_conditional_exp}, for any
  \(b\in B^+\setminus\{0\}\), \(E(b)\) is in the ideal in~\(B\)
  generated by~\(b\).  Then \(E(b)\precsim b\) because~\(b\) is
  properly infinite (see
  \cite{Kirchberg-Rordam:Non-simple_pi}*{Theorem~4.16}).  Now
  suppose that every \(a\in A^+\setminus\{0\}\) is properly infinite
  in~\(B\) and \(E(b)\precsim b\) for every \(b\in B^+\).  Let~\(J\)
  be the ideal in~\(B\) generated by \(b\in B^+\setminus\{0\}\).
  Proposition~\ref{prop:separate_ideal_vs_conditional_exp} implies
  \(J\cap A=E(J)\).  Since every ideal in~\(B\) is induced, this
  implies \(J=BE(J)B\).  So~\(b\) is in the ideal generated
  by~\(E(b)\).  This implies that \(b \precsim E(b)\) using that
  \(E(b)\) is properly infinite (see
  \cite{Kirchberg-Rordam:Non-simple_pi}*{Proposition 3.5(ii)}).
  Since
  \(b\oplus b \precsim E(b) \oplus E(b) \precsim E(b) \precsim b\)
  in~\(B\), we conclude that~\(b\) is properly infinite.  This shows
  that~\(B\) is purely infinite.

  If~\(A\) separates ideals in~\(B\) and~\(B\) is purely infinite,
  the same holds for all the quotient inclusions
  \(A/I \subseteq B/J\), \(I=J\cap A\), \(J\in \Ideals(B)\) (see
  \cite{Kirchberg-Rordam:Non-simple_pi}*{Proposition~4.3}).  By
  Proposition~\ref{prop:separate_ideal_vs_conditional_exp}, the
  conditional expectation \(E^I\colon B/J\to A/I\) is faithful.
  Hence the first part of the assertion shows that
  \(b \in (B/J)^+\setminus \{0\}\) implies
  \(0\neq E^I(b)\precsim b\).  Thus~\(A^+\) residually
  supports~\(B\).  Then~\(A^+\) fills~\(B\) by
  Proposition~\ref{prop:residual_supp_vs_filling}.
\end{proof}

\section{Inverse semigroup actions and  their crossed products}
\label{sec:isg_crossed}

In this section, we briefly recall inverse semigroup actions by
Hilbert bimodules and their crossed products, referring to
\cites{Buss-Exel-Meyer:Reduced, Kwasniewski-Meyer:Essential} for
more details.

\subsection{Inverse semigroup actions by Hilbert bimodules}

Throughout this paper, \(S\) is an inverse semigroup with unit
\(1\in S\).

\begin{definition}[\cite{Buss-Meyer:Actions_groupoids}]
  \label{def:S_action_Cstar}
  An \emph{action} of~\(S\) on a \(\Cst\)\nb-algebra~\(A\) (by
  Hilbert bimodules) consists of Hilbert
  \(A\)\nb-bimodules~\(\Hilm_t\) for \(t\in S\) and Hilbert bimodule
  isomorphisms
  \(\mu_{t,u}\colon \Hilm_t\otimes_A \Hilm_u\congto \Hilm_{tu}\) for
  \(t,u\in S\), such that
  \begin{enumerate}[label=\textup{(A\arabic*)}]
  \item \label{enum:AHB3}%
    for all \(t,u,v\in S\), the following diagram commutes
    (associativity):
    \[
    \begin{tikzpicture}[baseline=(current bounding box.west)]
      \node (1) at (0,1) {\((\Hilm_t\otimes_A \Hilm_u) \otimes_A \Hilm_v\)};
      \node (1a) at (0,0) {\(\Hilm_t\otimes_A (\Hilm_u \otimes_A \Hilm_v)\)};
      \node (2) at (5,1) {\(\Hilm_{tu} \otimes_A \Hilm_v\)};
      \node (3) at (5,0) {\(\Hilm_t\otimes_A \Hilm_{uv}\)};
      \node (4) at (7,.5) {\(\Hilm_{tuv}\)};
      \draw[<->] (1) -- node[swap] {ass} 			(1a);
      \draw[cdar] (1) -- node {\(\mu_{t,u}\otimes_A \Id_{\Hilm_v}\)} (2);
      \draw[cdar] (1a) -- node[swap] {\(\Id_{\Hilm_t}\otimes_A\mu_{u,v}\)} (3);
      \draw[cdar] (3.east) -- node[swap] {\(\mu_{t,uv}\)} (4);
      \draw[cdar] (2.east) -- node {\(\mu_{tu,v}\)} (4);
    \end{tikzpicture}
    \]
  \item \label{enum:AHB1}%
    \(\Hilm_1\) is the identity Hilbert \(A,A\)-bimodule~\(A\);
  \item \label{enum:AHB2}%
    \(\mu_{t,1}\colon \Hilm_t\otimes_A A\congto \Hilm_t\) and
    \(\mu_{1,t}\colon A\otimes_A \Hilm_t\congto \Hilm_t\) for
    \(t\in S\) are the maps defined by
    \(\mu_{1,t}(a\otimes\xi)=a\cdot\xi\) and
    \(\mu_{t,1}(\xi\otimes a) = \xi\cdot a\) for \(a\in A\),
    \(\xi\in\Hilm_t\).
  \end{enumerate}
\end{definition}

Any \(S\)\nb-action by Hilbert bimodules comes with canonical
involutions \(J_t\colon \Hilm_t^*\to \Hilm_{t^*}\) and inclusion
maps \(j_{u,t}\colon \Hilm_t\to\Hilm_u\) for \(t \le u\) that
satisfy the conditions required for a saturated Fell bundle
in~\cite{Exel:noncomm.cartan} (see
\cite{Buss-Meyer:Actions_groupoids}*{Theorem~4.8}).  Thus
\(S\)\nb-actions by Hilbert bimodules are equivalent to saturated
Fell bundles over~\(S\).  A nonsaturated Fell bundle over~\(S\) is
turned into a saturated Fell bundle over another inverse semigroup
in~\cite{BussExel:InverseSemigroupExpansions}, such that the full
and reduced section \(\Cst\)\nb-algebras stay the same.  Therefore,
we usually restrict attention to saturated Fell bundles, which we
may replace by inverse semigroup actions as in
Definition~\ref{def:S_action_Cstar}.
Definition~\ref{def:S_action_Cstar} contains (twisted) actions by
partial automorphisms.

\begin{example}[Twisted actions of inverse semigroups, see
  \cite{BussExel:Regular.Fell.Bundle}*{Definition~4.1}]
  \label{ex:twisted actions}
  A \emph{twisted action of an inverse semigroup}~\(S\)
  \emph{by partial automorphisms} on a \(\Cst\)\nb-algebra~\(A\)
  consists of partial automorphisms
  \(\alpha_t\colon D_{t^*}\to D_t\)
  of~\(A\)
  for \(t\in S\)
  -- that is, \(D_t\)
  is an ideal in~\(A\)
  and~\(\alpha_t\)
  is a \Star{}isomorphism -- and unitary multipliers
  \(\omega(t, u)\in \UMult(D_{t u})\)
  for \(t,u \in S\),
  such that \(D_1=A\)
  and the following conditions hold for \(r,t,u\in S\)
  and \(e,f\in E(S)\defeq 
	\setgiven{s\in S}{s^2=s}\):
  \begin{enumerate}
  \item \(\alpha_r\circ \alpha_t= \Ad_{\omega(r,t)}\alpha_{r t}\);
  \item \(\alpha_r\big(a \omega(t,u)\big) \omega(r,t u)=
    \alpha_r(a)\omega(r,t)\omega(r t,u)\) for \(a\in D_{r^*}\cap
    D_{t u}\);
  \item \(\omega(e,f) = 1_{e f}\)
    and \(\omega(r,r^*r)=\omega(r r^*,r)=1_r\),
    where~\(1_r\) is the unit of \(\Mult(D_r)\);
  \item \(\omega(t^*,e)\omega(t^*e,t)a = \omega(t^*,t) a\)
    for all \(a\in D_{t^* e t}\).
  \end{enumerate}
  Let \( ( (\alpha_t)_{t\in S}, \omega(t, u))_{t,u\in S}\)
  be a twisted action as above.
  For \(t\in S\),
  let~\(\Hilm_t\)
  be the Hilbert \(A\)\nb-bimodule associated to the partial
  homeomorphism~\(\alpha_t\);
  this is~\(D_t\)
  as a Banach space, and we denote elements by~\(b\delta_t\)
  to highlight the \(t\in S\)
  in which we view \(b\in D_t \subseteq A\)
  as an element; the Hilbert \(A\)\nb-bimodule structure is defined
  by
  \begin{alignat*}{2}
    a \cdot (b\delta_t) &\defeq (ab)\delta_t,&\qquad
    (b\delta_t) \cdot a &\defeq
    \alpha_t(\alpha_t^{-1}(b)a)\delta_t,\\
    \BRAKET{c\delta_t}{b\delta_t} &\defeq c b^*,&\qquad
    \braket{c\delta_t}{b \delta_t}&\defeq \alpha_t^{-1}(c^*b)
  \end{alignat*}
  for \(a\in A\), \(b,c\in D_t\).
  The formula
  \[
  \mu_{t,u}(a\delta_t\otimes b\delta_u)
  = \alpha_t(\alpha_t^{-1}(b)a)\omega(t,u)\delta_t
  \]
  well defines a Hilbert bimodule isomorphism
  \(\mu_{t,u}\colon \Hilm_t\otimes_A \Hilm_u\congto \Hilm_{tu}\)
  for  \(t,u\in S\) by \cite{BussExel:Regular.Fell.Bundle}*{Theorem~4.12}.
  And \((\Hilm_t, \mu_{t,u})_{t,u\in S}\)
  is an action of~\(S\)
  on~\(A\)
  by Hilbert bimodules.
  Fell bundles over~\(S\)
  that come from twisted partial actions are characterised in
  \cite{BussExel:Regular.Fell.Bundle}*{Corollary~4.16},
  where they are called ``regular.''
\end{example}

\begin{example}[Inverse semigroup gradings]
  An \(S\)\nb-grading \((B_t)_{t\in S}\) of a
  \(\Cst\)\nb-algebra~\(B\) as in Definition~\ref{def:S-graded}
  gives a Fell bundle over~\(S\), using the multiplication and
  involution in~\(B\).  This bundle is saturated if and only the
  grading is saturated.  Then it is an action of~\(S\) by Hilbert
  bimodules on \(A\defeq B_1\subseteq B\).  Thus inverse semigroup
  actions by Hilbert bimodules are inevitable in the study of
  regular inclusions (see also
  Proposition~\ref{prop:regular_vs_inverse_semigroups}).  Any
  inverse semigroup action
  \(\Hilm=\bigl((\Hilm_t)_{t\in S}, (\mu_{t,u})_{t,u\in S}\bigr)\)
  on a \(\Cst\)\nb-algebra~\(A\) comes from an \(S\)\nb-graded
  \(\Cst\)\nb-algebra.  Namely, embed the spaces~\(\Hilm_t\) for
  \(t\in S\) into, say, the full crossed product.  They form an
  \(S\)\nb-grading of the full crossed product.
\end{example}

Fell bundles over étale groupoids may be described through
\(S\)\nb-actions.

\begin{example}[Fell bundles over groupoids]
  \label{ex:fell_bundles_over_groupoids}
  Let~\(\Gr\) be an \emph{étale groupoid} with locally compact and
  Hausdorff unit space~\(X\).  So the range and source maps
  \(\rg,\s\colon \Gr \rightrightarrows X\) are local homeomorphisms.
  A \emph{Fell bundle over~\(\Gr\)} is defined, for instance, in
  \cite{BussExel:Fell.Bundle.and.Twisted.Groupoids}*{Section~2}.  It
  is an upper semicontinous bundle \(\A=(A_\gamma)_{\gamma\in \Gr}\)
  of complex Banach spaces equipped with a continuous involution
  \(^*\colon \A\to \A\) and a continuous multiplication
  \({\cdot}\colon \setgiven{(a,b)\in \A\times \A} {a\in
    A_{\gamma_1},\ b \in A_{\gamma_2},\ (\gamma_1,\gamma_2)\in
    \Gr^{(2)}} \to \A\), which satisfy some natural properties.  The
  set of (open) bisections
  \[
    \Bis(\Gr)\defeq \setgiven{U\subseteq \Gr}{U\text{ is open and }
      s|_U,r|_U \text{ are injective}}
  \]
  is a unital inverse semigroup
  with
  \(U\cdot V\defeq \setgiven{\gamma\cdot\eta}{\gamma\in U,\ \eta\in
    V}\) for \(U,V\in \Bis(\Gr)\).  Namely, \(X\in \Bis(\Gr)\) is
  the unit element and
  \(U^*\defeq \setgiven{\gamma^{-1}}{\gamma\in U}\) for
  \(U\in \Bis(\Gr)\).  Let~\(A_U\) for \(U\in \Bis(\Gr)\) be the
  space of continuous sections of \((A_\gamma)_{\gamma\in \Gr}\)
  vanishing outside~\(U\).  The spaces \(A_{r(U)}\) and~\(A_{s(U)}\)
  are closed two-sided ideals in \(A=A_X\), and~\(A_U\) becomes a
  Hilbert \(A_{r(U)}\)-\(A_{s(U)}\)-bimodule with the bimodule
  structure
  \((a\cdot \xi \cdot b)(\gamma)\defeq a(r(\gamma))\xi(\gamma)
  b(s(\gamma))\) and the right and left inner products
  \(\braket{\xi}{\eta}(x)\defeq
  \xi(s|_U^{-1}(x))^*\eta(s|_U^{-1}(x))\),
  \(\BRAKET{\xi}{\eta}(x)\defeq
  \xi(r|_U^{-1}(x))\eta(r|_U^{-1}(x))^*\).  For \(U, V\in S\), the
  formula
  \[
    \mu_{U,V}(\xi\otimes \eta) (\gamma)\defeq
    \xi\Bigl(r|_U^{-1}\bigl(r(\gamma)\bigr)\Bigr)
    \eta\Bigl(s|_V^{-1}(s\bigl(\gamma)\bigr)\Bigr), \qquad
    \gamma \in UV,
  \]
  defines a Hilbert bimodule map
  \(\mu_{U,V}\colon A_U\otimes_A A_V\to A_{UV}\).  This data defines
  a Fell bundle over~\(\Bis(\Gr)\), which is saturated if \(\A\) is
  (see \cites{BussExel:Fell.Bundle.and.Twisted.Groupoids,
    Kwasniewski-Meyer:Essential} for details).  If \(\A\) is not
  saturated, this may be naturally turned into a saturated Fell
  bundle over another inverse semigroup in a number of ways.  For
  instance, for any inverse subsemigroup \(S\subseteq \Bis(\Gr)\) we
  may let~\(\tilde{S}\) be the family of all Hilbert subbimodules
  of~\(A_U\) for \(U\in S\).  Equivalently, elements
  of~\(\tilde{S}\) are of the form \(A_U\cdot I\) for \(U\in S\) and
  \(I \triangleleft A\).  Then~\(\tilde{S}\), with operations
  defined as above, forms an inverse semigroup that acts by Hilbert
  bimodules on~\(A\) (see
  \cite{Kwasniewski-Meyer:Essential}*{Lemma~7.3}).
\end{example}

Let~\(A\) be a \(\Cst\)\nb-algebra with an action~\(\Hilm\) of a
unital inverse semigroup~\(S\).  Let \(\dual{A}\) and
\(\widecheck{A}=\Prim(A)\) be the space of irreducible
representations and the primitive ideal space of~\(A\),
respectively.  The action of~\(S\) on~\(A\) induces actions
\(\dual{\Hilm}= (\dual{\Hilm}_t)_{t\in S}\) and
\(\widecheck{\Hilm}=(\widecheck{\Hilm}_t)_{t\in S}\) of~\(S\) by
partial homeomorphisms on \(\dual{A}\) and~\(\widecheck{A}\),
respectively (see \cite{Buss-Meyer:Actions_groupoids}*{Lemma~6.12},
\cite{Kwasniewski-Meyer:Essential}*{Section~2.3}).  The
homeomorphisms
\[
  \widecheck{\Hilm}_t\colon \widecheck{\s(\Hilm)}_t \congto
  \widecheck{\rg(\Hilm)}_{t},\qquad
  \dual{\Hilm}_t\colon \dual{\s(\Hilm)}_t \congto
  \dual{\rg(\Hilm)}_{t}
\]
are given by Rieffel's correspondence and induction of
representations, respectively.  Any action by partial homeomorphisms
has a transformation groupoid, which is \'etale (see
\cite{Exel:Inverse_combinatorial}*{Section~4} or
\cite{Kwasniewski-Meyer:Essential}*{Section~2.1} for details).

\begin{definition}[\cites{Kwasniewski-Meyer:Stone_duality, Kwasniewski-Meyer:Essential}]
  \label{def:dual_groupoid}
  We call \(\dual{\Hilm}\) and~\(\widecheck{\Hilm}\) \emph{dual
    actions} to the action~\(\Hilm\) of~\(S\) on~\(A\).  The
  transformation groupoids \(\dual{A}\rtimes S\) and
  \(\widecheck{A}\rtimes S\) are called \emph{dual groupoids}
  of~\(\Hilm\).
\end{definition}

\begin{example}[Dual groupoids to Fell bundles]
  \label{exa:transformation_groupoid_Fell}
  Let~\(\A\) be a Fell bundle over an étale groupoid~\(\Gr\) with
  locally compact Hausdorff object space~\(X\).  Then~\(\Gr\) acts
  naturally both on the primitive ideal space~\(\check{A}\) and the
  spectrum~\(\dual{A}\) of the \(\Cst\)\nb-algebra
  \(A\defeq \Cont_0(X, \A)\).  More specifically, every irreducible
  representation of~\(A\) factors through the evaluation map
  \(A \to A_x\) for some \(x\in X\), and this defines a continuous
  map \(\psi\colon \dual{A} \to X\), which is the anchor map of the
  \(\Gr\)\nb-action on~\(\dual{A}\) given by the partial
  homeomorphisms
  \(\psi_\gamma\colon\dual{A_{\s(\gamma)}} \to
  \dual{A_{\rg(\gamma)}}\) induced by the Hilbert
  \(A_{\rg(\gamma)}\)-\(A_{\s(\gamma)}\)-bimodules \(A_\gamma\) for
  \(\gamma\in \Gr\) (see, for instance,
  \cite{Ionescu-Williams:Remarks_ideal_structure}*{Section~2}).  The
  corresponding transformation groupoid is
  \[
    \dual{A}\rtimes \Gr\defeq \setgiven{([\pi],  \gamma)\in
      \dual{A}\times  \Gr}{\psi([\pi])=s(\gamma)}.
  \]
  Two elements \(([\rho],\eta)\) and \(([\pi], \gamma)\) are
  composable if and only if \([\rho]=\psi_{\gamma}([\pi])\), and
  then their composite is \(([\pi],\eta\gamma)\).  The inverse is
  \(([\pi],\gamma)^{-1} = (\psi_{\gamma}([\pi]),\gamma^{-1})\).  The
  maps \(\psi\colon \dual{A} \to X\) and
  \(\psi_\gamma\colon\dual{A_{\s(\gamma)}} \to
  \dual{A_{\rg(\gamma)}}\) for \(\gamma\in \Gr\), factor through to
  a \(\Gr\)\nb-action on~\(\check{A}\), which defines a
  transformation groupoid \(\check{A}\rtimes \Gr\)
  (see~\cite{Ionescu-Williams:Remarks_ideal_structure}).  These
  actions give rise to transformation groupoids
  \(\check{A}\rtimes \Gr\) and \(\dual{A}\rtimes \Gr\).  Let
  \(S\subseteq \Bis(\Gr)\) be a unital, inverse subsemigroup of
  bisections of~\(\Gr\) which is \emph{wide} in the sense that
  \(\bigcup S = \Gr\) and \(U\cap V\) is a union of bisections
  in~\(S\) for all \(U,V\in S\).  Then turning this into the inverse
  semigroup action~\(\tilde{S}\) on~\(A\) described in
  Example~\ref{ex:fell_bundles_over_groupoids}, we have natural
  isomorphisms of groupoids (see
  \cite{Kwasniewski-Meyer:Essential}*{Remark~7.4})
  \[
    \dual{A}\rtimes \Gr\cong \dual{A}\rtimes\tilde{S},\qquad
    \dual{A}\rtimes \Gr\cong \dual{A}\rtimes\tilde{S}.
  \]
  We call \(\check{A}\rtimes \Gr\) and \(\dual{A}\rtimes \Gr\)
  \emph{dual groupoids} for~\(\A\).
\end{example}

\subsection{Crossed products}

Fix an action
\(\Hilm=\bigl((\Hilm_t)_{t\in S}, (\mu_{t,u})_{t,u\in S}\bigr)\)
of~\(S\) on~\(A\).  For any \(t\in S\), let \(\rg(\Hilm_t)\) and
\(\s(\Hilm_t)\) be the ideals in~\(A\) generated by the left and
right inner products of vectors in~\(\Hilm_t\), respectively.
Thus~\(\Hilm_t\) is an
\(\rg(\Hilm_t)\)-\(\s(\Hilm_t)\)-imprimitivity bimodule.  If
\(v\le t\), then the inclusion map~\(j_{t,v}\) restricts to a
Hilbert bimodule isomorphism
\(\Hilm_v \congto \rg(\Hilm_v) \cdot \Hilm_t = \Hilm_t\cdot
\s(\Hilm_v)\).  For \(t,u\in S\) and for all \(v\le t,u\) this gives
Hilbert bimodule isomorphisms
\(\vartheta^v_{u,t}\colon \Hilm_t\cdot \s(\Hilm_v)
\xleftarrow[\cong]{j_{t,v}} \Hilm_v \xrightarrow[\cong]{j_{u,v}}
\Hilm_u\cdot \s(\Hilm_v)\).  Let
\begin{equation}
  \label{eq:Itu}
  I_{t,u} \defeq \cl{\sum_{v \le t,u} \s(\Hilm_v)}
\end{equation}
be the closed ideal generated by~\(\s(\Hilm_v)\) for \(v\le t,u\).
It is contained in \(\s(\Hilm_t)\cap\s(\Hilm_u)\), and the inclusion
may be strict.  There is a unique Hilbert bimodule isomorphism
\begin{equation}
  \label{eq:Def-thetas}
  \vartheta_{u,t}\colon \Hilm_t\cdot I_{t,u}
  \congto \Hilm_u\cdot I_{t,u}
\end{equation}
that restricts to~\(\vartheta_{u,t}^v\) on
\(\Hilm_t\cdot \s(\Hilm_v)\) for all \(v\le t,u\) by
\cite{Buss-Exel-Meyer:Reduced}*{Lemma~2.4}.  The
\emph{\Star{}algebra} \(A\rtimes_\alg S\) of the action \(\Hilm\) is
defined as the quotient vector space of
\(\bigoplus_{t\in S} \Hilm_t\) by the linear span of
\(\vartheta_{u,t}(\xi)\delta_u-\xi\delta_t\) for all \(t,u\in S\)
and \(\xi\in\Hilm_t\cdot I_{t,u}\).  The algebraic structure on
\(A\rtimes_\alg S\) is given by the multiplication
maps~\(\mu_{t,u}\) and the involutions~\(J_t\).

\begin{definition}
  The \emph{\textup{(}full\textup{)} crossed product} \(A\rtimes S\)
  of the action~\(\Hilm\) is the maximal \(\Cst\)\nb-completion of
  the \Star{}algebra \(A\rtimes_\alg S\) described above.
\end{definition}

\begin{definition}
  \label{def:representation}
  A \emph{representation} of~\(\Hilm\) in a
  \(\Cst\)\nb-algebra~\(B\) is a family of linear maps
  \(\pi_t\colon \Hilm_t\to B\) for \(t\in S\) such that
  \(\pi_{tu}(\mu_{t,u}(\xi\otimes\eta)) = \pi_t(\xi)\pi_u(\eta)\),
  \(\pi_t(\xi_1)^*\pi_t(\xi_2) = \pi_1(\braket{\xi_1}{\xi_2})\) and
  \(\pi_t(\xi_1)\pi_t(\xi_2)^* = \pi_1(\BRAKET{\xi_1}{\xi_2})\) for
  all \(t,u\in S\), \(\xi, \xi_1,\xi_2 \in\Hilm_t\),
  \(\eta\in\Hilm_u\).  The representation is called \emph{injective}
  if~\(\pi_1\) is injective; then all the maps~\(\pi_t\) for
  \(t\in S\) are isometric.
\end{definition}

Any representation~\(\pi\) of~\(\Hilm\) in~\(B\) induces a
\Star{}homomorphism \(\pi\rtimes S\colon A\rtimes S\to B\).
 Conversely, every \Star{}homomorphism
\(A\rtimes S\to B\) is equal to \(\pi\rtimes S\) for a unique
representation~\(\pi\) (compare
\cite{Buss-Exel-Meyer:Reduced}*{Proposition~2.9}).  This universal
property determines \(A\rtimes S\) uniquely up to isomorphism.

The \(\Cst\)\nb-algebra \(A\rtimes S\) is canonically isomorphic to
the full section \(\Cst\)\nb-algebra of the Fell bundle over~\(S\)
corresponding to~\(\Hilm\).  The reduced section \(\Cst\)\nb-algebra
of a Fell bundle over~\(S\) was first defined using inducing pure states, 
see~\cite{Exel:noncomm.cartan}. 
An equivalent definition appears
in~\cite{Buss-Exel-Meyer:Reduced}, where this is called the reduced
crossed product \(A\rtimes_\red S\) of the action~\(\Hilm\) 
(the reduced \(\Cst\)\nb-algebra obtained in \cite{MR3690239},  
using a regular representation, is in general different).  
  The
main ingredient in the construction
in~\cite{Buss-Exel-Meyer:Reduced} is the weak conditional
expectation for \(A\rtimes S \to A''\) described in
\cite{Buss-Exel-Meyer:Reduced}*{Lemma~4.5} through the formula
\begin{equation}
  \label{eq:formula-cond.exp}
  E(\xi\delta_t)=\stlim_i\vartheta_{1,t}(\xi\cdot u_i),
\end{equation}
where \(\xi\in \Hilm_t\), \(t\in S\), \((u_i)\) is an approximate
unit for~\(I_{1,t}\) and \(\stlim\) denotes the limit in the strict
topology on \(\Mult(I_{1,t})\subseteq A''\).  This weak expectation is
symmetric by \cite{Kwasniewski-Meyer:Essential}*{Theorem~3.22}.

\begin{definition}
  \label{def:reduced_crossed}
  The \emph{reduced crossed product} is the quotient
  \(\Cst\)\nb-algebra \(A\rtimes_\red S \defeq (A\rtimes S)/\Null_E\).
  Hence there is a canonical surjection
  \(\Lambda\colon A\rtimes S \to A\rtimes_\red S\), and
  \begin{equation}
    \label{eq:kernel}
    \ker \Lambda =\Null_E= \setgiven{b\in A\rtimes S}{E(b^* b)=0}.
  \end{equation}
  So the induced weak expectation~\(E_\red\) on \(A\rtimes_\red S\)
  is faithful.
\end{definition}

\begin{remark}
  \label{rem:cross_product_graded}
  The canonical maps from \(A\rtimes_\alg S\) to \(A\rtimes S\) and
  to \(A\rtimes_\red S\) are injective by
  \cite{Buss-Exel-Meyer:Reduced}*{Proposition~4.3}.  In particular,
  both \(A\rtimes S\) and \(A\rtimes_\red S\) are naturally
  \(S\)\nb-graded with the same Fell bundle \((\Hilm_t)_{t\in S}\)
  over~\(S\).
\end{remark}

\begin{definition}
  \label{def:Hausdorffness}
  The action~\(\Hilm=(\Hilm_t)_{t\in S}\) is called \emph{closed} if
  the weak expectation \(E\colon A\rtimes S\to A''\) given
  by~\eqref{eq:formula-cond.exp} is \(A\)\nb-valued.  So it is a
  genuine conditional expectation
  \(A\rtimes S\to A\subseteq A\rtimes S\).
\end{definition}

\begin{remark}
  \label{rem:closed_actions_vs_groupoids}
  The action~\(\Hilm\) is closed if and only if the unit space
  \(\dual{A}\) is closed in the dual groupoid \(\dual{A}\rtimes S\)
  or, equivalently, \(\check{A}\) is closed
  in~\(\check{A}\rtimes S\) (see
  \cite{Buss-Exel-Meyer:Reduced}*{Theorem~6.5} and
  \cite{Kwasniewski-Meyer:Essential}*{Proposition~3.20}).  This
  explains the name.  By
  \cite{Buss-Exel-Meyer:Reduced}*{Proposition~6.3}, \(\Hilm\) is
  closed if and only if the ideal~\(I_{t,1}\) defined
  in~\eqref{eq:Itu} is complemented in~\(s(\Hilm_t)\) for each
  \(t\in S\).
\end{remark}

\begin{example}[\(\Cst\)\nb-algebras of Fell bundles over groupoids]
  \label{ex:full_reduced_section_algebras}
  Retain the notation from
  Example~\ref{ex:fell_bundles_over_groupoids}.  The \Star{}algebra
  associated to the Fell bundle~\(\A\) over the \'etale
  groupoid~\(\Gr\) is denoted by \(\mathfrak{S}(\Gr,\A)\).  It is
  the linear span of compactly supported continuous sections
  \(A_{U}=\Contc(U,\A)\) for all bisections \(U\in \Bis( \Gr)\) with
  a convolution and involution given by
  \[
    (f*g)(\gamma) \defeq \sum_{\rg(\eta) = \rg(\gamma)}
    f(\eta)\cdot g(\eta^{-1}\cdot \gamma),\qquad
    (f^*)(\gamma) \defeq f(\gamma^{-1})^*
  \]
  for all \(f,g\in \mathfrak{S}(\Gr,\A)\), \(\gamma\in \Gr\).  The
  \emph{full section \(\Cst\)\nb-algebra} \(\Cst(\Gr,\A)\) is
  defined as the maximal \(\Cst\)\nb-completion of the
  \Star{}algebra \(\mathfrak{S}(\Gr,\A)\).  This \(\Cst\)\nb-algebra
  contains \(A\defeq \Cst(X,\A)\) as a \(\Cst\)\nb-subalgebra.  It
  is equipped with a generalised expectation
  \(E\colon \Cst(\Gr,\A) \to \mathfrak{B}(X,\A)\), where
  \(\mathfrak{B}(X,\A)\) is the \(\Cst\)\nb-algebra of bounded Borel
  sections of the \(\Cst\)\nb-bundle~\(\A|_X\) and~\(E\) on
  \(\mathfrak{S}(\Gr,\A)\) restricts sections
  to~\(X\).  The \emph{reduced section \(\Cst\)\nb-algebra} can be
  defined as the quotient
  \[
    \Cst_\red(\Gr,\A)\defeq  \Cst(\Gr,\A)/\Null_E.
  \]
  All this follows from
  \cite{Kwasniewski-Meyer:Essential}*{Proposition~7.10}.  Let
  \(S\subseteq \Bis(\Gr)\) be a unital, inverse subsemigroup of
  bisections of~\(\Gr\) which is \emph{wide} in the sense that
  \(\bigcup S = \Gr\) and \(U\cap V\) is a union of bisections
  in~\(S\) for all \(U,V\in S\).  If~\(\A\) is saturated, then the spaces
  \((A_U)_{U\in S}\), with operations inherited from
  \(\mathfrak{S}(\Gr,\A)\), form an action of~\(S\) on~\(A\) by
  Hilbert bimodules, and the associated crossed products are
  isomorphic to the corresponding section \(\Cst\)-algebras.  In
  general, we modify the construction as follows.  The convolution
  and multiplication from \(\mathfrak{S}(\Gr,\A)\) make
  \((A_U\cdot I)_{{U\in S, I \triangleleft A} }\) an action by
  Hilbert bimodules on~\(A\), and there are natural isomorphisms
  \[
    \Cst(\Gr,\A)\cong A\rtimes \tilde{S},\qquad
    \Cst_\red(\Gr,\A)\cong A\rtimes_{\red} \tilde{S}
  \]
  (see \cite{Kwasniewski-Meyer:Essential}*{Propositions 7.6 and~7.9}).
  If the groupoid~\(\Gr\) is Hausdorff, that is, the unit
  space~\(X\) is closed in~\(\Gr\), then the inverse semigroup
  action~\(\tilde{S}\) is closed.  The converse implication holds
  if~\(\A\) is a Fell line bundle.  Fell line bundles over~\(\Gr\)
  are equivalent to twists of~\(\Gr\).  If \((\Gr,\Sigma)\) is a
  twisted groupoid and~\(\LL\) the corresponding line bundle, then
  \(\Cst(\Gr,\Sigma)\cong \Cst(\Gr,\LL)\cong A\rtimes S\) and
  \(\Cst_\red(\Gr,\Sigma)\cong \Cst_\red(\Gr,\LL)\cong
  A\rtimes_{\red} S\).  That is, twisted groupoid
  \(\Cst\)\nb-algebras are also modelled by inverse semigroup
  actions.
\end{example}

\subsection{Essential crossed products}

The \emph{local multiplier algebra} \(\Locmult(A)\) of~\(A\) is the
inductive limit of the multiplier algebras \(\Mult(J)\), where~\(J\)
runs through the directed set of essential ideals in~\(A\)
(see~\cite{Ara-Mathieu:Local_multipliers}).  A key idea
in~\cite{Kwasniewski-Meyer:Essential} is a natural generalised
expectation \(EL\colon A\rtimes S\to \Locmult(A)\) with values in
\(\Locmult(A)\).  It is defined as follows: for each
\(\xi\in \Hilm_t\), \(t\in S\), the element
\(EL(\xi)\in \Mult(I_{1,t}\oplus I_{1,t}^\bot)\subseteq
\Locmult(A)\) is given by
\[
  EL(\xi)(u+v)\defeq \vartheta_{1,t}(\xi u) \in A
\]
for \(u\in I_{1,t}\), \(v\in I_{1,t}^\bot\).  This generalised
expectation is symmetric by
\cite{Kwasniewski-Meyer:Essential}*{Theorem~4.11}.  Hence~\(EL\)
factors through a faithful pseudo-expectation on the quotient
\[
  A\rtimes_\ess S\defeq (A\rtimes S)/\Null_{EL}.
\]
There is a canonical embedding
\(\iota\colon \Locmult(A) \hookrightarrow \hull(A)\) compatible with
the inclusions \(A\subseteq \Locmult(A)\) and
\(A\subseteq \hull(A)\) (see
\cite{Frank:Injective_local_multiplier}*{Theorem~1}).  Thus
the \emph{canonical \(\Locmult\)-expectation}~\(EL\)
may be viewed as a pseudo-expectation.

\begin{definition}[\cite{Kwasniewski-Meyer:Essential}*{Definition~4.4}]
  We call \(A\rtimes_\ess S\) the \emph{essential crossed product}.
\end{definition}
 For a closed action, both
\(E\) and~\(EL\) take values in~\(A\) and then
\(A\rtimes_\ess S = A\rtimes_\red S\).  In general,
\(\Null_{E}\subseteq \Null_{EL}\) and  there are surjective maps
\[
  A\rtimes S \onto  A\rtimes_\red S \onto A\rtimes_\ess S.
\]
A \(\Cst\)\nb-algebra~\(B\) with \Star{}epimorphisms
\(A\rtimes S \onto B \onto A\rtimes_\ess S\) that compose to the
canonical quotient map \(A\rtimes S \to A\rtimes_\ess S\) is called
an \emph{exotic crossed product}
(see~\cite{Kwasniewski-Meyer:Essential}).  The following
proposition characterises when the reduced and essential crossed
products coincide:

\begin{proposition}[\cite{Kwasniewski-Meyer:Essential}*{Corollary~4.17}]
  \label{prop:ess_vs_reduced}
  \(A\rtimes_\ess S = A\rtimes_\red S\) if and only if for every
  \(b\in (A\rtimes S)^+\setminus \{0\}\) there is \(\varepsilon>0\)
  such that
  \(\setgiven{\pi\in\dual{A}}{\norm{\pi''(E(b))}> \varepsilon}\) has
  nonempty interior, if and only if for every
  \(b\in (A\rtimes S)^+\setminus \{0\}\) the set
  \(\setgiven{\pi\in\dual{A}}{\norm{\pi''(E(b))}\neq0}\) is not
  meagre in \(\dual{A}\).
\end{proposition}	

\begin{example}[Essential section \(\Cst\)\nb-algebras]
  \label{ex:essential_section_algebras}
  Let~\(\A\) be a Fell bundle over an \'etale groupoid~\(\Gr\) with
  locally compact Hausdorff unit space~\(X\).  There is a canonical
  generalised expectation \(EL\colon \Cst(\Gr,\A) \to \Locmult(A)\)
  (see \cite{Kwasniewski-Meyer:Essential}*{Section~7.4}).  Let
  \[
    \mathfrak{M}(X, \A) \defeq \setgiven{f\in \mathfrak{B}(X,\A)}
    { f \text{ vanishes on a comeagre set}}.
  \]
  If the bundle~\(\A|_X\) is continuous, then there is a natural
  embedding
  \(\Locmult(A)\hookrightarrow \mathfrak{B}(X,\A) / \mathfrak{M}(X,
  \A)\), and~\(EL\) is the composite of~\(E\) and the quotient map
  \(\mathfrak{B}(X,\A) \onto \mathfrak{B}(X,\A)/\mathfrak{M}(X,
  \A)\).  If the bundle is discontinuous, we define~\(EL\) using the
  isomorphism \(\Cst(\Gr,\A)\cong A\rtimes \tilde{S}\) from
  Example~\ref{ex:full_reduced_section_algebras}.  The
  \emph{essential section \(\Cst\)\nb-algebra} is defined in
  \cite{Kwasniewski-Meyer:Essential}*{Definition~7.12} as the
  quotient
  \[
    \Cst_\ess(\Gr,\A)\defeq  \Cst(\Gr,\A)/\Null_{EL}.
  \]
\end{example}

\begin{example}[Essential twisted groupoid \(\Cst\)\nb-algebras]
  \label{ex:essential_groupoid_algebras}
  We define the \emph{essential groupoid \(\Cst\)\nb-algebra}
  \(\Cst_\ess(\Gr,\Sigma)\) of a twisted groupoid \((\Gr,\Sigma)\)
  as \(\Cst_\ess(\Gr,\LL)\) for the corresponding Fell line
  bundle~\(\LL\). Denoting by~\(X\) the unit space of~\(\Gr\),
  \cite{Kwasniewski-Meyer:Essential}*{Proposition~7.18} implies
  that the following are equivalent:
  \begin{enumerate}
  \item \(\Cst_\red(\Gr,\Sigma)=\Cst_\ess(\Gr,\Sigma)\);
  \item \(\setgiven*{x\in X}{E_\red(f)(x)\neq 0}\) is not meagre for
    every \(f\in \Cst_\red(\Gr,\Sigma)^+\setminus\{0\}\);
  \item if \(f\in \Cst_\red(\Gr,\Sigma)^+\setminus\{0\}\), then
    \(\setgiven*{x\in X}{\norm{E_\red(f)(x)}>\varepsilon}\) has
    nonempty interior for some \(\varepsilon >0\).
  \end{enumerate}
  Here \(E_\red\colon \Cst_\red(\Gr,\Sigma)\to \mathfrak{B}(X)\) is
  the canonical generalised expectation that restricts
  sections of the corresponding Fell line bundle to~\(X\).
\end{example}

\begin{example}[Twisted crossed products by partial automorphisms]
  \label{ex:essential_partial_algebras}
  Let \((\alpha,\omega)\) be a twisted action of an inverse
  semigroup~\(S\) by partial automorphisms on a
  \(\Cst\)\nb-algebra~\(A\) as in Example~\ref{ex:twisted actions}.
  By \cite{BussExel:Regular.Fell.Bundle}*{Definition~6.2}, a
  \emph{covariant representation} of \((\alpha,\omega)\) on a
  Hilbert space~\(\Hils\) is a pair \((\rho, v)\) consisting of a
  \Star{}homomorphism \(\rho\colon A\to \Bound(\Hils)\) and a family
  \(v=(v_t)_{t\in S}\) of partial isometries in \(\Bound(\Hils)\)
  such that
  \[
    \rho(\alpha_t(b))=v_t \rho(b) v_t^*,\quad
    \overline{\rho}(\omega(t,u)) = v_t v_u v_{tu}^*,\quad
    v_tv_t^*= \overline{\rho}(1_{tt^*}),\quad
    v_t^*v_t= \overline{\rho}(1_{t^*t}),
  \]
  for all \(b\in D_t\), \(t,u\in S\).  Here \(\overline{\rho}\) is
  the extension of~\(\rho\) to the enveloping von Neumann algebra of
  \(A\), so that \(\overline{\rho}(\omega(t,u))\) and
  \(\overline{\rho}(1_e)\) make sense.  By definition, the
  \emph{full crossed product} for \((\alpha,\omega)\) is the
  universal \(\Cst\)-algebra for covariant representations, and by
  \cite{BussExel:Regular.Fell.Bundle}*{Theorem~6.3} it is naturally
  isomorphic to the full crossed product \(A\rtimes S\) for the
  associated inverse semigroup action~\(\Hilm\) by Hilbert bimodules
  (see Example~\ref{ex:twisted actions}).  The \emph{reduced crossed
    product} for \((\alpha,\omega)\) may be identified with the
  reduced crossed product \(A\rtimes_{\red} S\) by
  \cite{BussExel:Regular.Fell.Bundle}*{Definition~6.6}.  We define
  the \emph{essential crossed product} for \((\alpha,\omega)\) as
  \(A\rtimes_{\ess} S\).  By
  \cite{BussExel:Regular.Fell.Bundle}*{Theorem~7.2}, for any twisted
  groupoid \((\Gr,\Sigma)\) the bisections~\(S\) of~\(\Gr\) that
  trivialise the twist~\(\Sigma\) give rise to a twisted inverse
  semigroup action \((\alpha,\omega)\) by partial automorphisms of
  \(A\defeq \Cont_0(\Gr^0)\) such that
  \(\Cst_\red(\Gr,\Sigma) \cong A\rtimes_{\red} S\).  By definition,
  this descends to an isomorphism
  \(\Cst_\ess(\Gr,\Sigma)\cong A\rtimes_{\ess} S\).
\end{example}

The reduced section \(\Cst\)\nb-algebra \(\Cst_\red(\Gr,\A)\) is
usually defined through the regular representation, which is the
direct sum of representations
\[
  \lambda_x\colon \Cst(\Gr,\A) \to \Bound\bigl(\ell^2(\Gr_x,\A)\bigr).
\]
Here \(\ell^2(\Gr_x,\A)\) is the Hilbert \(A_x\)\nb-module completion
of \(\bigoplus_{\s(\gamma)=x} A_\gamma\) with the obvious right
multiplication and the standard inner product
\(\braket{f}{g} \defeq \sum_{\s(\gamma)=x} f(\gamma)^* g(\gamma)\).
For \(f\in \mathfrak{S}(H,\A)\) and
\(g\in \bigoplus_{\s(\gamma)=x} A_\gamma\), define
\(\lambda_x(f)(g)(\gamma)\defeq \sum_{\rg(\eta) = \rg(\gamma)}
f(\eta) g(\eta^{-1}\gamma)\).  The kernel of
\(\bigoplus_{x\in X} \lambda_x\) is~\(\Null_E\).  So the reduced
\(\Cst\)\nb-algebra \(\Cst_\red(H,\A)\) is isomorphic to the
completion of \(\mathfrak{S}(H,\A)\) in the reduced norm
\(\norm{f}_\red \defeq \sup_{x\in X} {}\norm{\lambda_x(f)}\).  We
now describe essential algebras in a similar fashion:

\begin{definition}[\cite{Kwasniewski-Meyer:Essential}*{Definition~7.14}]
  \label{def:dangerous}
  Call \(x\in X\) \emph{dangerous} if there is a net~\((\gamma_n)\)
  in~\(\Gr\) that converges towards two different points
  \(\gamma\neq\gamma'\) in~\(\Gr\) with
  \(\s(\gamma) = \s(\gamma') = x\).
\end{definition}

\begin{proposition}
  \label{prop:dangerous_give_essential}
  Let~\(\A\) be a continuous Fell bundle over an étale
  groupoid~\(\Gr\) with locally compact and Hausdorff unit
  space~\(X\).  Assume~\(\Gr\) is covered by countably many
  bisections and let \(D\subseteq X\) be the set of dangerous
  points.  Then
  \[
    \ker \bigoplus_{x\in X\setminus D} \lambda_x = \Null_{EL}.
  \]
  That is, \(\Cst_{\ess}(\Gr,\A)\) is isomorphic to the Hausdorff
  completion of \(\mathfrak{S}(\Gr,\A)\) in the seminorm
  \(\norm{f}_\ess \defeq \sup_{x\in X\setminus D}
  {}\norm{\lambda_x(f)}\).
\end{proposition}

\begin{proof}
  For each \(x\in X\), \(\gamma \in \Gr_x=s^{-1}(x)\) and
  \(f\in \Cst(\Gr,\A)\),
  \[
    \norm{\lambda_\gamma(f) 1_\gamma}^2
    = \braket{1_\gamma}{\lambda_x(f^*f)1_\gamma}
    =E(f^*f) (r(\gamma)).
  \]
  Thus
  \(\ker\lambda_x=\setgiven{f\in \Cst(\Gr,\A) }{E(f^*f)
    (r(\Gr_x))=0}\).  We claim that The set of dangerous points is
  \(\Gr\)\nb-invariant.  Indeed, if \(\eta\in s^{-1}(x)\) and there
  is a net~\((\gamma_n)\) that converges towards two different
  \(\gamma\neq\gamma'\in H\) with \(\s(\gamma) = \s(\gamma') = x\),
  then the net~\((\gamma_n\eta^{-1})\) converges to
  \(\gamma\eta^{-1}\neq\gamma'\eta^{-1}\in H\) with
  \(\s(\gamma\eta^{-1}) = \s(\gamma'\eta^{-1}) = r(\eta)\).  Hence
  \(x\in D\) implies that \(r(\Gr_x)\subseteq D\).  Then
  \[
    \ker \bigoplus_{x\in X\setminus D} \lambda_x
    = \bigcap_{x\in X\setminus D} \ker\lambda_x
    = \setgiven{f\in \Cst(\Gr,\A)}
    {E(f^*f) (x)=0 \text{ for all }x\in X\setminus D}.
  \]
  The set on the right hand side is equal to~\(\Null_{EL}\) by
  \cite{Kwasniewski-Meyer:Essential}*{Proposition~7.18}.
\end{proof}

\section{Exactness of inverse semigroup actions}\label{sec:exactness}

\subsection{Functoriality}

For group actions, the full and reduced crossed products are
functors, and the reduced one preserves injective homomorphisms.  We
extend this to inverse semigroup actions by Hilbert bimodules.

 \begin{definition} 
  \label{def:homomorphism}
  Let
  \(\Hilm=\bigl((\Hilm_t)_{t\in S}, (\mu_{t,u})_{t,u\in S}\bigr)\)
  and
  \(\Hilm[F] = \bigl( (\Hilm[F]_t)_{t\in S}, (\nu_{t,u})_{t,u\in
    S}\bigr)\) be two actions of~\(S\) by Hilbert bimodules on \(A\)
  and~\(B\), respectively.  A \emph{homomorphism}~\(\psi\)
  from~\(\Hilm\) to~\(\Hilm[F]\) is a family of linear maps
  \(\psi_t\colon \Hilm_t\to \Hilm[F]_t\) for \(t\in S\) such that   for all \(t,u\in S\), \(\xi\in\Hilm_t\), \(\eta\in\Hilm_u\) we have
\[
    \psi_{tu}(\mu_{t,u}(\xi\otimes\eta))
    = \nu_{t,u}(\psi_t(\xi)\otimes\psi_u(\eta)),
  \,\,\braket{\psi_t(\xi)}{\psi_t(\eta)}
    = \psi_1(\braket{\xi}{\eta}),\,\,
   \BRAKET{\psi_t(\xi)}{\psi_t(\eta)}
    = \psi_1(\BRAKET{\xi}{\eta}).
  \]
  The maps~\(\psi_t\) are always contractive.  We call~\(\psi\)
  \emph{injective} if~\(\psi_1\) is injective; then the
  maps~\(\psi_t\) are isometric for all \(t\in S\).  We
  call~\(\psi\) an \emph{isomorphism} if all~\(\psi_t\) are
  isomorphisms.
\end{definition}

We use the superscripts \({^{\Hilm}}\) and~\({^{\Hilm[F]}}\) to
distinguish between objects defined for the actions \({\Hilm}\)
and~\(\Hilm[F]\).

\begin{proposition}
  \label{prop:functoriality}
  Let~\(\psi\) be a homomorphism from an action
  \(\Hilm=(\Hilm_t)_{t\in S}\) on \(A\) to an action
  \(\Hilm[F] = (\Hilm[F]_t)_{t\in S}\) on~\(B\).  It induces a
  \Star{}homomorphism
  \(\psi\rtimes S\colon A\rtimes S\to B\rtimes S\) where
  \((\psi\rtimes S)(\xi)=\psi_t(\xi)\) for \(\xi\in \Hilm_t\),
  \(t\in S\). In particular, \(\psi\) respects the involution and
  inclusions maps on \(\Hilm\) and~\(\Hilm[F]\), and
  \(\psi\rtimes S\) restricts to a \Star{}homomorphism
  \(\psi\rtimes_\alg S\colon A\rtimes_\alg S\to B\rtimes_\alg S\).
  Moreover the following conditions are equivalent:
  \begin{enumerate}
  \item \label{item:functoriality0}%
    \(\psi\rtimes S\)  descends to a \Star{}homomorphism
    \(\psi\rtimes_\red S\colon A\rtimes_\red S\to B\rtimes_\red S\) that respects the canonical weak expectations, that is the following diagram
    comutes
  \[
      \begin{tikzcd}
        A\rtimes_\alg S \ar[r, "\subseteq"] \ar[d,"\psi\rtimes_\alg S"] &
        A\rtimes S \ar[r,"\Lambda^{\Hilm}"] \ar[d,"\psi\rtimes S"] &
        A\rtimes_\red S \ar[r,"E^{\Hilm}"] \ar[d,"\psi\rtimes_\red S"] &
        A'' \ar[d,"\psi_1''"] \\
        B\rtimes_\alg S \ar[r, "\subseteq"'] &
        B\rtimes S \ar[r,"\Lambda^{\Hilm[F]}"'] &
        B\rtimes_\red S \ar[r,"E^{\Hilm[F]}"'] &
        B''
      \end{tikzcd}
  \]
    where~\(\Lambda^{\Hilm}\)  (resp.~\(\Lambda^{\Hilm[F]}\)) is the regular representation
    and~\(E^{\Hilm}\) (resp.~\(E^{\Hilm[F]}\)) is the canonical weak conditional expectation 
    associated to the action \(\Hilm\) (resp.~\(\Hilm[F]\)).

  \item \label{item:functoriality3}%
    \(\psi_1''([I_{1,t}^{\Hilm}])=
    \psi_1''([s(\Hilm_t)])[I_{1,t}^{\Hilm[F]}]\), for all
    \(t\in S\), where \([I_{1,t}^{\Hilm}], [s(\Hilm_t)]\in A''\) and
    \([I_{1,t}^{\Hilm[F]}]\in B''\) are the support projections of
    the ideals \(I_{1,t}^{\Hilm}, s(\Hilm_t)\subseteq A\) and
    \(I_{1,t}^{\Hilm[F]} \subseteq B\), respectively.
\end{enumerate}
  If the above equivalent conditions hold and \(\psi\)~is injective, then so are \(\psi\rtimes_\red S\),
  \(\psi\rtimes_\alg S\), and~\(\psi_1''\).
\end{proposition}

\begin{proof}
  The Hilbert bimodules \(\Hilm[F]_t\) for \(t\in S\) embed
  into~\(B\rtimes S\).  Hence we may treat the maps
  \(\psi_t\colon \Hilm_t\to \Hilm[F]_t\) as taking values
  in~\(B\rtimes S\).  Then~\(\psi\) is a representation of~\(\Hilm\)
  in~\(B\rtimes S\).  It integrates to a \Star{}homomorphism
  \(\psi\rtimes S\colon A\rtimes S\to B\rtimes S\) by \cite{Buss-Exel-Meyer:Reduced}*{Proposition~2.9}.  This restricts
  to a \Star{}homomorphism
  \(\psi\rtimes_\alg S\colon A\rtimes_\alg S\to B\rtimes_\alg S\) and therefore
  \(\psi\) respects the induced involutions and inclusions maps on 
  \(\Hilm\) and~\(\Hilm[F]\). 
  Our first goal is to show that~\ref{item:functoriality0} is
  equivalent to
  \begin{equation}
    \label{equ:functoriality1}
    E^{\Hilm[F]}\circ \Lambda^{\Hilm[F]} \circ (\psi\rtimes S) = \psi_1'' \circ E^{\Hilm}\circ
    \Lambda^{\Hilm}.
  \end{equation}
  It is clear that~\ref{item:functoriality0}
  implies~\eqref{equ:functoriality1}.  Conversely,
  assume~\eqref{equ:functoriality1}.  For any
  \(a\in (A\rtimes S)^+\), we get
  \begin{multline}
    \label{eq:reduced_crossed_functor}
    a\in \ker\Lambda^{\Hilm}
    \iff E\bigl(\Lambda^{\Hilm} (a)\bigr)=0
    \Longrightarrow \psi_1''\Bigl(E^{\Hilm}\bigl(\Lambda^{\Hilm} (a)\bigr)\Bigr)=0
    \\
    \iff E^{\Hilm[F]}\Bigl( \Lambda^{\Hilm[F]}\bigl((\psi\rtimes S) (a)\bigr)\Bigr)=0
    \iff (\psi\rtimes S)(a)\in \ker\Lambda^{\Hilm[F]}.
  \end{multline}
  Hence \((\psi\rtimes S)(\ker\Lambda^{\Hilm}) \subseteq \ker(\Lambda^{\Hilm[F]})\).
  So \(\psi\rtimes S\) descends to a \Star{}homomorphism
  \(\psi\rtimes_\red S\) as in~\ref{item:functoriality0}.  
  If, in addition, \(\psi\) is injective, then so
  is~\(\psi_1''\), and then the only one-sided implication
  in~\eqref{eq:reduced_crossed_functor} may be reversed.  Thus
  \((\psi\rtimes S)(\ker\Lambda^{\Hilm}) = \ker(\Lambda^{\Hilm[F]})\), so that
  \(\psi\rtimes_\red S\) is injective.  Since the canonical map
  \(A\rtimes_\alg S\to A\rtimes_\red S\) is still injective, it also
  follows that \(\psi\rtimes_\alg S\) is injective.  So we get all
  assertions in~\ref{item:functoriality0}.
	
  Next, we prove that~\ref{item:functoriality3} is equivalent
  to~\eqref{equ:functoriality1}.  By passing to biduals, we get a
  weakly continuous \Star{}homomorphism
  \((\psi\rtimes S)''\colon (A\rtimes S)'' \to (B\rtimes S)''\).
  The \(\Cst\)\nb-algebra~\(A''\) and the Hilbert
  bimodules~\(\Hilm_t''\) for \(t\in S\) embed naturally
  into~\((A\rtimes S) ''\).  Similarly, \(\Hilm[F]_t''\) for
  \(t\in S\) are embedded into \((B\rtimes S)''\).  Then
  \((\psi\rtimes S)''|_{ \Hilm_t''} = \psi_t''\) for \(t\in S\).
  Moreover, if \(\xi\in \Hilm_t\subseteq A\rtimes S\), then
  \(E^{\Hilm}( \Lambda^{\Hilm}(\xi))=\xi\cdot [I_{1,t}]\in A''\),
  where the product is taken in ~\((A\rtimes S) ''\) (see
  \eqref{eq:formula-cond.exp} or the proof of
  \cite{Buss-Exel-Meyer:Reduced}*{Lemma~4.5}).
  Hence~\eqref{equ:functoriality1} is equivalent to
  \begin{equation}
    \label{equ:functoriality2}
    \psi_{t}(\xi)\psi_1''([I_{1,t}^{\Hilm}])=\psi_t(\xi)[I_{1,t}^{\Hilm[F]}],\qquad  \text{ for every } \xi\in \Hilm_t, t\in S, 
  \end{equation}
  Since \(\xi=\xi \cdot [s(\Hilm_t)]\) holds inside \((A\rtimes S) ''\),
  \ref{item:functoriality3} implies~\eqref{equ:functoriality2}:
  \[
    \psi_{t}(\xi)\psi_1''([I_{1,t}^{\Hilm}])
    = \psi_{t}(\xi)\psi_1''([s(\Hilm_t)]) \psi_1''([I_{1,t}^{\Hilm}])
    = \psi_{t}(\xi) \psi_1''([s(\Hilm_t)])[I_{1,t}^{\Hilm[F]}]
    = \psi_{t}(\xi) [I_{1,t}^{\Hilm[F]}].
  \]
  Conversely,
  \eqref{equ:functoriality2} implies
  \(\psi_{t}(\xi_1)^* \psi_{t}(\xi_2)\psi_1''([I_{1,t}^{\Hilm}]) =
  \psi_{t}(\xi_1)^* \psi_{t}(\xi_2) [I_{1,t}^{\Hilm[F]}]\) for all
  \(\xi_1,\xi_2\in \Hilm_t\).  Taking linear combinations, we may
  then replace \(\psi_{t}(\xi_1)^* \psi_{t}(\xi_2)\) by an
  approximate unit for the ideal \(\s(\Hilm_t)\).  And then we may
  take a strong limit over this approximate unit to arrive at
  \(\psi_1''([s(\Hilm_t)]) \psi_1''([I_{1,t}^{\Hilm}]) =
  \psi_1''([s(\Hilm_t)]) [I_{1,t}^{\Hilm[F]}]\).  The left hand side
  then simplifies to \(\psi_1''([I_{1,t}^{\Hilm}])\) because
  \(I_{1,t}^{\Hilm}\subseteq s(\Hilm_t)\) implies
  \(\psi_1''([I_{1,t}^{\Hilm}])\leq \psi_1''([s(\Hilm_t)])\). Hence
  \eqref{equ:functoriality2} implies \ref{item:functoriality3}.
\end{proof}

\begin{remark}
  \label{rem:some_conditions_implying_conditions}
  Retain the notation from Proposition~\ref{prop:functoriality}.
  Since~$\psi$ respects inclusions, we have
  \(\psi_1(I_{1,t}^{\Hilm})\subseteq I_{1,t}^{\Hilm[F]}\) and
  therefore
  \(\psi_1''([I_{1,t}^{\Hilm}]) \leq [I_{1,t}^{\Hilm[F]}]\).  We
  also have
  \(\psi_1''([I_{1,t}^{\Hilm}])\leq \psi_1''([s(\Hilm_t)])\) because
  \(I_{1,t}^{\Hilm}\subseteq s(\Hilm_t)\) .  Thus every
  representation satisfies
  \[\psi_1''([I_{1,t}^{\Hilm}])\le
  \psi_1''([s(\Hilm_t)])[I_{1,t}^{\Hilm[F]}] \qquad\text{ for all }t\in S.
	\]
  So condition~\ref{item:functoriality3} only asserts the inverse
  inequality.  This condition is always satisfied when \(S=G\) is a
  group, as then \([I_{1,t}^{\Hilm}]=0\) for \(t\neq 1\) (and
  \([I_{1,1}^{\Hilm}]=1\)).
  %
  %
  For general actions, \ref{item:functoriality3} holds whenever
  \(\psi_1(I_{1,t}^{\Hilm})=\psi_1(s(\Hilm_t))I_{1,t}^{\Hilm[F]}\)
  for all \(t\in S\).  The latter equality is automatic for
  inclusion and quotient homomorphisms in
  Proposition~\ref{prop:exactness_of_universal_rep} below.
\end{remark}

Condition~\ref{item:functoriality3} in
Proposition~\ref{prop:functoriality} may fail (and perhaps for some
purposes one might want to include it in the definition of a
homomorphism).  We thank Alcides Buss, Diego Mart\'inez and Jonathan
Taylor for point this to us.

\begin{example}
  Consider the actions whose crossed products are described in
  \cite{Buss-Exel-Meyer:Reduced}*{Proposition~8.5}.  Namely, let
  \(S=\{-1,0,1\}\) be the inverse semigroup with the usual number
  multiplication.  Take any \(\Cst\)\nb-algebra~\(A\) and any
  ideal~\(I\) in~\(A\) different from~\(A\).  Let
  \(\Hilm_1=\Hilm_{-1}\defeq A\) and \(\Hilm_{0}=I\) be trivial
  Hilbert bimodules over~\(A\), and let
  \(\mu_{t,u}(a\otimes b)=a\cdot b\) for \(t,u\in S\) be just the
  multiplication in~\(A\).  Then
  \(A\rtimes^{\Hilm} S=A\rtimes_\red^{\Hilm} S=A\rtimes_\alg^{\Hilm}
  S\cong A\oplus A/I\) (see \cite{Buss-Exel-Meyer:Reduced}*{(8.6)}).
  We let~\(\Hilm[F]\) be the similar action with~\(I\) replaced
  by~\(A\).  Then
  \(A\rtimes^{\Hilm[F]} S=A\rtimes_\red^{\Hilm[F]}
  S=A\rtimes_\alg^{\Hilm[F]} S\cong A\).  The inclusion maps yield a
  homomorphism~\(\psi\) from~\(\Hilm\) to~\(\Hilm[F]\) where
  \(\psi\rtimes S=\psi\rtimes_\alg S \colon A\oplus A/I \to A\) is
  given by \((a\oplus b+I)\mapsto a\).  This homomorphism is not
  injective, although~\(\psi\) is.  So conditions
  in
  Proposition~\ref{prop:functoriality} are not satisfied.  Indeed,
  note that \(\psi\rtimes_{\red} S=\psi\rtimes S\) exists in this
  example, but it does not intertwine the canonical weak
  expectations, which are given by
  \(E^{\Hilm}(a\oplus (b+I))=\frac{a+b}{2}+\frac{a-b}{2}[I]\) and
  \(E^{\Hilm[F]}(a)=a\), for \(a,b \in A\), where \([I]\in A''\) is
  the support projection of~\(I\).
\end{example}

\subsection{Restrictions of  actions}

We fix an action
\(\Hilm=\bigl((\Hilm_t)_{t\in S}, (\mu_{t,u})_{t,u\in S}\bigr)\) of
a unital inverse semigroup~\(S\) on a \(\Cst\)\nb-algebra~\(A\) by
Hilbert bimodules.

\begin{definition}
  Let
  \(\Ideals^{\Hilm}(A)\defeq \setgiven{I\in\Ideals(A)}
  {I\Hilm_t=\Hilm_t I \text{ for every }t\in S}\) be the set of
  \emph{\(\Hilm\)\nb-invariant} ideals in~\(A\).
\end{definition}

\begin{lemma}
  \label{lem:invariance_vs_duals}
  Let \(I\in\Ideals(A)\) and let~\(B\) be any \(S\)\nb-graded
  \(\Cst\)\nb-algebra with grading \((\Hilm_t)_{t\in S}\).  The
  following are equivalent:
  \begin{enumerate}
  \item \label{enu:invariance_vs_duals1}%
    \(I\) is~\(\Hilm\)\nb-invariant, that is
    \(I\in \Ideals^{\Hilm}(A)\);
  \item \label{enu:invariance_vs_duals2}%
    \(I\) is restricted, that is, \(I\in \Ideals^{B}(A)\);
  \item \label{enu:invariance_vs_duals3}%
    the open subset \(\check{I}\subseteq \check{A}\) is invariant in
    the dual groupoid \(\widecheck{A}\rtimes S\);
  \item \label{enu:invariance_vs_duals4}%
    the open subset \(\dual{I}\subseteq \dual{A}\) is invariant in
    the dual groupoid \(\dual{A}\rtimes S\).
  \end{enumerate}
\end{lemma}

\begin{proof}
  Proposition~\ref{prop:invariant_ideals_in_regular} implies that
  \ref{enu:invariance_vs_duals1} and~\ref{enu:invariance_vs_duals2}
  are equivalent.  It is easy to see that
  \ref{enu:invariance_vs_duals3} and~\ref{enu:invariance_vs_duals4}
  are equivalent.  If \(t\in S\), then
  \(I\cdot \Hilm_t = \Hilm_t\cdot I\) is equivalent to
  \(\widecheck{\Hilm_t}(\widecheck{I} \cap \widecheck{D}_t) =
  \widecheck{I} \cap \widecheck{D}_{t^*}\) (see
  \cite{Kwasniewski:Topological_freeness}*{page~645} or the proof of
  \cite{Abadie-Abadie:Ideals}*{Proposition~3.10}).  That is,
  \ref{enu:invariance_vs_duals1} and~\ref{enu:invariance_vs_duals3}
  are equivalent.
\end{proof}

Let \(I\in\Ideals^{\Hilm}(A)\) be an \(\Hilm\)\nb-invariant ideal
in~\(A\).  There are natural induced actions of~\(S\) on \(I\)
and~\(A/I\).  Namely, the family
\(\Hilm|_I \defeq (\Hilm_t I)_{t\in S}\) of Hilbert
\(I\)\nb-bimodules, with the restrictions of the
isomorphisms~\(\mu_{t,u}\) to
\(\Hilm_t I\otimes_A \Hilm_uI =\Hilm_t \otimes_A \Hilm_u I\congto
\Hilm_{tu} I\) for \(t,u\in S\), forms an inverse semigroup action
by Hilbert bimodules on the \(\Cst\)\nb-algebra~\(I\).  For
\(t\in S\), the quotient Banach space \(\Hilm_t/\Hilm_t I\) is a
Hilbert \(A/I\)\nb-bimodule in a natural way because
\(\Hilm_t I = I\Hilm_t\).  If \(t,u\in S\), then the isomorphism
\(\mu_{t,u}\colon \Hilm_t \otimes_A \Hilm_u\congto \Hilm_{tu}\)
induces an isomorphism
\[
  \widetilde{\mu}_{t,u}\colon
  (\Hilm_t \otimes_A \Hilm_u) \bigm/ (\Hilm_t \otimes_A \Hilm_u I)
  \congto \Hilm_{tu} \mathbin/\Hilm_{tu} I.
\]
There are natural isomorphisms of Hilbert bimodules
\[
  q_{t,u}\colon (\Hilm_t/\Hilm_t I)\otimes_{A/I} (\Hilm_u/\Hilm_uI)
  \congto (\Hilm_t \otimes_A \Hilm_u) \bigm/ (\Hilm_t \otimes_A \Hilm_u I)
\]
because \(I\Hilm_u = \Hilm_u I\) (see, for instance,
\cite{Kwasniewski:Cuntz-Pimsner-Doplicher}*{Lemma~1.8}).  Now
\(\Hilm|_{A/I}\defeq (\Hilm_t/\Hilm_t I)_{t\in S}\) with the
isomorphisms
\(\widetilde{\mu}_{t,u}\circ q_{t,u}\colon (\Hilm_t/\Hilm_t I)
\otimes_A (\Hilm_u/\Hilm_uI) \congto (\Hilm_{tu}/\Hilm_{tu} I)\) for
\(t,u\in S\) is an action of~\(S\) by Hilbert bimodules on~\(A/I\).

\begin{definition}
  We call the actions \(\Hilm|_I\) and~\(\Hilm|_{A/I}\) above the
  \emph{restrictions} of~\(\Hilm\) to \(I\) and~\(A/I\),
  respectively.
\end{definition}

\begin{remark}
  \label{rem:homomorphism_actions_ideal}
  Let \(I\in \Ideals^{\Hilm}(A)\).  The inclusions
  \(\Hilm_t I\subseteq \Hilm_t\) for \(t\in S\) yield an injective
  homomorphism from~\(\Hilm|_I\) to~\(\Hilm\).  The quotient maps
  \(\Hilm_t \to \Hilm_t/\Hilm_t I\) for \(t\in S\) yield a
  homomorphism from~\(\Hilm\) to~\(\Hilm|_{A/I}\).
\end{remark}

\begin{remark}
  Let~\(B\) be an \(S\)\nb-graded \(\Cst\)\nb-algebra with grading
  \((B_t)_{t\in S}\) and let \(I\in\Ideals^B(A)\).  The induced
  ideal~\(BIB\) carries the \(S\)\nb-grading \(BIB \cap B_t\) by
  Proposition~\ref{prop:invariant_ideals_in_regular}.  The
  quotient~\(B/BIB\) is \(S\)\nb-graded by the images of~\(B_t\) by
  Lemma~\ref{lem:quotients_of_regular}.
\end{remark}

\begin{remark}
  \label{rem:dual_ideal}
  If~\(I\) is \(\Hilm\)\nb-invariant, then the dual actions for the
  restrictions \(\Hilm|_I\) and~\(\Hilm|_{A/I}\) are equal to the
  restrictions of the action \((\dual{\Hilm}_t)_{t\in S}\) to
  \(\dual{I}\) and \(\dual{A}\setminus \dual{I}\), respectively.
  Let \(\dual{I}\rtimes S\) and
  \(\dual{A}\setminus\dual{I}\rtimes S\) denote the transformation
  groupoids dual to \(\Hilm|_I\) and~\(\Hilm|_{A/I}\), respectively.
  Then
  \[
    \dual{I}\rtimes S
    = (\dual{A}\rtimes S)|_{\dual{I}},\qquad
    \dual{A}\setminus\dual{I}\rtimes S
    = (\dual{A}\rtimes S)|_{\dual{A}\setminus\dual{I}},
  \]
  where the right hand sides in the above equalities mean the
  restrictions of the transformation groupoid \(\dual{A}\rtimes S\)
  to the invariant subsets \(\dual{I}\)
  and~\(\dual{A}\setminus\dual{I}\) of the unit space~\(\dual{A}\)
  of \(\dual{A}\rtimes S\).  In particular, if the space of units in
  \(\dual{A}\rtimes S\) is closed, then the same holds for the
  transformation groupoids dual to the restrictions \(\Hilm|_I\)
  and~\(\Hilm|_{A/I}\).
\end{remark}

\begin{proposition}
  \label{pro:closed_unit_isg_extension}
  If the action~\(\Hilm\) on~\(A\) is closed, then so are the
  restrictions \(\Hilm|_I\) and~\(\Hilm|_{A/I}\) for each
  \(\Hilm\)\nb-invariant ideal \(I\idealin A\).
\end{proposition}

\begin{proof}
  This follows from Remarks \ref{rem:closed_actions_vs_groupoids}
  and~\ref{rem:dual_ideal}.
\end{proof}

\subsection{Exact actions}

\begin{proposition}
  \label{prop:exactness_of_universal_rep}
  Let~\(\Hilm\) be an action of~\(S\) by Hilbert bimodules on a
  \(\Cst\)\nb-algebra~\(A\) and let \(I\in \Ideals^{\Hilm}(A)\).
  Let~\(\iota\) be the injective homomorphism from~\(\Hilm|_I\)
  into~\(\Hilm\) and let~\(\kappa\) be the quotient homomorphism
  from~\(\Hilm\) onto~\(\Hilm|_{A/I}\) in
 Remark~\textup{\ref{rem:homomorphism_actions_ideal}}. 
They induce an exact   
  sequence
 \[
    0 \to I\rtimes S
    \xrightarrow{\iota \rtimes S}
    A\rtimes S \xrightarrow{\kappa \rtimes S}
    A/I\rtimes S \to 0.
\]
  It descends to a sequence
  \begin{equation}
    \label{eq:sequence_to_be_exact}
    0 \to I\rtimes_\red S
    \xrightarrow{\iota \rtimes_\red S}
    A\rtimes_\red S \xrightarrow{\kappa\rtimes_\red S}
    A/I\rtimes_\red S \to 0,
 \end{equation}

 which may fail to be exact only in the middle: \(\iota \rtimes_\red S\)
  is injective, \(\kappa \rtimes_\red S\) is surjective and the
  range of \(\iota \rtimes_\red S\) is contained in the kernel of
  \(\kappa \rtimes_\red S\). 
\end{proposition}

\begin{proof} 

Note that \(\iota(I_{1,t}^{\Hilm I} )= I_{1,t}^{\Hilm}\cdot I =\iota(s(\Hilm_t I))\cdot I_{1,t}^{\Hilm} \)
and
\(\kappa(I_{1,t}^{\Hilm} )=I_{1,t}^{{\Hilm/\Hilm I}}\), for all \(t\in S\). Hence~\(\iota\) and~\(\kappa\) satisfy the equivalent  conditions  in Proposition \ref{prop:functoriality}
by the last part of Remark \ref{rem:some_conditions_implying_conditions}.
Thus, by Proposition \ref{prop:functoriality}, not only the \Star{}homomorphisms \(\iota \rtimes S\), \(\kappa \rtimes S\) but also \(\iota \rtimes_{\red} S\), \(\kappa \rtimes_{\red} S\) exist, 
and \(\iota \rtimes_\red S\) is injective. 
The maps 	\(\kappa \rtimes S\) and
  \(\kappa \rtimes_\red S\) are surjective because their images
  contain the dense \Star{}subalgebra \(A/I\rtimes_\alg S\).

  We prove that \(\iota \rtimes S\) is injective.  Let~\(\pi\) be a
  faithful, nondegenerate representation of \(I\rtimes S\) on a
  Hilbert space~\(\Hils\).  Since~\(I\) is nondegenerate in
  \(I\rtimes S\), the representation~\(\pi|_I\) is also
  nondegenerate.  Therefore, for each \(\xi_t\in \Hilm_t\), the
  formula
  \[
    \widetilde{\pi}(\xi_t)\pi(a)h \defeq \pi(\xi_t a)h
  \]
  for \(a\in I\) and \(h\in \Hils\) defines a bounded operator
  on~\(\Hils\).  Alternatively, \(\widetilde{\pi}(\xi_t)\) could be
  defined using an approximate identity~\((\mu_\lambda)\) for~\(I\),
  as the limit of the strongly convergent
  net~\(\pi(\xi_t \mu_\lambda)\).  A standard proof shows that
  \((\widetilde{\pi}_t)_{t\in S}\) is a
  representation~\(\widetilde{\pi}\)
  of~\(\Hilm\).  The integrated representation
  \(\widetilde{\pi}\rtimes S\colon A\rtimes S \to \Bound(\Hils)\)
  satisfies
  \((\widetilde{\pi}\rtimes S) \circ (\iota\rtimes S)= \pi\).  Hence
  \(\iota\rtimes S\) is injective.

  The composite maps \((\kappa \rtimes S) \circ (\iota \rtimes S)\)
  and \((\kappa \rtimes_\red S) \circ (\iota \rtimes_\red S)\)
  vanish.  Hence the range of \(\iota \rtimes S\) is contained in
  \(\ker(\kappa\rtimes S)\) and the range of
  \(\iota \rtimes_\red S\) is contained in
  \(\ker(\kappa\rtimes_\red S)\).  Conversely, we claim that the
  range of \(\iota \rtimes S\) contains \(\ker(\kappa\rtimes S)\).
  We identify \(I\rtimes S\) with its image in \(A\rtimes S\).  Let
  \(q\colon A\rtimes S \to (A\rtimes S)\mathbin/(I\rtimes S)\) be
  the quotient map.  For \(t\in S\), the restriction of~\(q\)
  to~\(\Hilm_t I\) vanishes.  Hence~\(q\) induces maps
  \[
    \psi_t\colon \Hilm_t/ \Hilm_t I \to (A\rtimes S)/(I\rtimes S).
  \]
  They form a representation of~\(\Hilm_{A/I}\).  It integrates to a
  homomorphism
  \[
    \psi\rtimes S\colon
    (A/I)\rtimes S \to (A\rtimes S) / (I\rtimes S).
  \]
  We have \((\psi\rtimes S) \circ (\kappa\rtimes S)=q\).  Hence
  \(\ker(\kappa\rtimes S) \subseteq \ker(q)=I\rtimes S\).
\end{proof}

\begin{definition}
  The action~\(\Hilm\) is \emph{exact} if the
  sequence~\eqref{eq:sequence_to_be_exact} is exact for each
  \(I\in \Ideals^{\Hilm}(A)\).
\end{definition}

\begin{example}[Exactness of twisted groupoids]
  \label{ex:exactness_twisted_groupoids}
  An action~\(\Hilm\) of an inverse semigroup~\(S\) on a commutative
  \(\Cst\)\nb-algebra \(A\cong \Cont_0(X)\) corresponds to a twisted
  étale groupoid \((\Sigma,\Gr)\) with unit space~\(X\) (see
  \cite{BussExel:Fell.Bundle.and.Twisted.Groupoids}).  Here
  \(\Gr= X\rtimes S\) is the dual groupoid of~\(\Hilm\).  The
  action~\(\Hilm\) is exact if and only if the corresponding \emph{twisted
  groupoid \((\Gr,\Sigma)\) is exact} in the sense that
  for any open invariant subset \(U\subseteq X\), the sequence of
  reduced twisted groupoid \(\Cst\)\nb-algebras
  \[
    \Cst_\red(\Gr|_U,\Sigma|_U)
    \into \Cst_\red(\Gr,\Sigma)
    \onto \Cst_\red(\Gr|_{X\setminus U},\Sigma|_{X\setminus U})
  \]
  is exact.  If the twist is trivial, the name \emph{inner exact} for
  such groupoids is introduced
  in~\cite{AnantharamanDelaroch:Weak_containment}.  This example
  generalises as follows.
\end{example}

\begin{example}[Exactness of Fell bundles over groupoids]
  \label{ex:exactness_fell_bundles_groupoids}
  Let \(\A=(A_\gamma)_{\gamma\in \Gr}\) be a Fell bundle over an
  étale groupoid~\(\Gr\) with locally compact Hausdorff unit
  space~\(X\).  Let \(S\subseteq \Bis(\Gr)\) be a wide inverse
  semigroup of bisections and turn the Fell bundle~\(\A\) into an
  action~\(\Hilm\) of~\(\tilde{S}\) on the section
  \(\Cst\)\nb-algebra \(A=A|_X\), such that the associated universal
  and reduced \(\Cst\)\nb-algebras remain the same (see Examples
  \ref{ex:fell_bundles_over_groupoids}
  and~\ref{ex:full_reduced_section_algebras}).  This natural
  correspondence extends to invariant ideals and the associated
  algebras.  Indeed, by Lemma~\ref{lem:invariance_vs_duals} an ideal
  \(I\) in \(A\) is \(\Hilm\)-invariant if and only if it is
  \(\Gr\)-invariant in the sense
  of~\cite{Ionescu-Williams:Remarks_ideal_structure} (\(\check{I}\)
  is invariant under the dual groupoid
  \(\check{A}\rtimes \Gr\cong \widecheck{A}\rtimes \tilde{S}\)).
  The restricted actions \(\Hilm|_I\) and~\(\Hilm|_{A/I}\)
  of~\(\tilde{S}\) correspond to\emph{ restricted Fell bundles}
  \(\A|_{I}=(A_\gamma I_{s(\gamma)})_{\gamma\in \Gr}\)
  and~\(\A_{A/I}=(A_{\gamma}/(A_\gamma I_{s(\gamma)}))_{\gamma\in
    \Gr}\) over~\(\Gr\) as defined in
  \cite{Ionescu-Williams:Remarks_ideal_structure}*{Propositions 3.3
    and~3.4}.  The authors
  of~\cite{Ionescu-Williams:Remarks_ideal_structure} consider
  separable Fell bundles over locally compact Hausdorff groupoids.
  However, neither separability nor Hausdorffness are used in the
  construction of \(\A|_{I}\) and~\(\A_{A/I}\).
  Proposition~\ref{prop:exactness_of_universal_rep} implies that
 the sequence
  \[
    \Cst(\A|_{I})
    \into \Cst(\A)
    \onto \Cst(\A|_{A/I})
  \]
  is exact.  This extends the main result
  of~\cite{Ionescu-Williams:Remarks_ideal_structure} to \'etale
  groupoids that are not separable or not Hausdorff.  We will say
  that the \emph{Fell bundle~\(\A\) is exact} if for every
  \(\Gr\)\nb-invariant ideal~\(I\) in~\(A\), the sequence
  \[
    \Cst_\red(\A|_{I})
    \into \Cst_\red(\A)
    \onto \Cst_\red(\A|_{A/I})
  \]
  is exact.  Equivalently, the action~\(\Hilm\) corresponding
  to~\(\A\) is exact.
\end{example}

\subsection{Exactness for essential crossed products}

The essential crossed product is not functorial (see
\cite{Kwasniewski-Meyer:Essential}*{Remark~4.8}).  This complicates
the definition of ``exactness'' for essential crossed products.
Only the quotient maps cause extra problems:

\begin{lemma}\label{lem:essential_crossed_product_ideal_embedding}
  Let~\(\Hilm\) be an action of~\(S\) by Hilbert bimodules on a
  \(\Cst\)\nb-algebra~\(A\) and let \(I\in \Ideals^{\Hilm}(A)\).
  The injective homomorphism~\(\iota\) from~\(\Hilm|_I\)
  into~\(\Hilm\) induces an injective \Star{}homomorphism
  \(\iota\rtimes_\ess S\colon I\rtimes_\ess S \to A\rtimes_\ess S\).
  Its image is the ideal in \(A\rtimes_\ess S\) generated by~\(I\).
\end{lemma}

\begin{proof}
  If~\(J\) is an essential ideal in~\(I\), then \(J\oplus I^\bot\)
  is an essential ideal in~\(A\).  The obvious inclusions
  \(\Mult(J) \to \Mult(J\oplus I^\bot)\) for the essential ideals
  \(J\subseteq I\) induce a natural isomorphism from \(\Locmult(I)\)
  onto an ideal in~\(\Locmult(A)\).  Let
  \(EL_I\colon I\rtimes S\to \Locmult(I)\) be the canonical
  essential expectation.  We are going to prove below that the
  following diagram commutes:
  \begin{equation}
    \label{eq:diagram_for_ideals_and_essential}
    \begin{tikzcd}
      I\rtimes  S \ar[r, hookrightarrow, "\iota\rtimes S"] \ar[d,"EL_I"] &
      A\rtimes S \ar[d,"EL"] & \\
      \Locmult(I) \ar[r, hookrightarrow,] &
      \Locmult(A)
    \end{tikzcd}
  \end{equation}
  Then
  \(\Null_{EL}\cap \iota\rtimes S(I\rtimes S)= \iota\rtimes
  S(\Null_{EL_I})\) because \(\iota\rtimes S(I\rtimes S)\) is an
  ideal in \(A\rtimes S\).  This, in turn, implies that the
  injective homomorphism \(\iota\rtimes S\) factors through an
  injective homomorphism
  \(\iota\rtimes_\ess S\colon I\rtimes_\ess S= I\rtimes
  S/\Null_{EL_I} \to A\rtimes_\ess S=A\rtimes S/ \Null_{EL}\).

  To check~\eqref{eq:diagram_for_ideals_and_essential}, let
  \(t\in S\).  If
  \(I_{t,1} = \cl{\sum_{v \le t,1} \s(\Hilm_v)}\) then
  \(I_{t,1} I = \cl{\sum_{v \le t,1} \s(\Hilm_v I)}\) and the
  restriction of the map
  \(\vartheta_{t,1}\colon \Hilm_t\cdot I_{t,1} \congto I_{t,1}\)
  (defined in~\eqref{eq:Def-thetas}) to \(\Hilm_t I \cdot I_{t,1}\)
  coincides with the corresponding map defined for the restricted
  action \((\Hilm_s I)_{s\in S}\).  Hence for each
  \(\xi\in \Hilm_t I\) the element
  \(EL(\xi)\in \Mult(I_{1,t}\oplus I_{1,t}^\bot)\subseteq
  \Mult(I_{1,t} I\oplus I_{1,t}^\bot I \oplus I_{1,t} I^\bot\oplus
  I_{1,t}^\bot I^\bot)\) acts on \(u\in I_{1,t}I\) in the same way
  as \(EL_I(\xi)\):
  \[
    EL(\xi)u=\vartheta_{1,t}(\xi u)=EL_I(\xi) u.
  \]
  Since
  \(EL(\xi) (I_{1,t}^\bot I \oplus I_{1,t} I^\bot\oplus I_{1,t}^\bot
  I^\bot) =0\), the embedding \(\Locmult(I)\subseteq \Locmult(A)\)
  maps \(EL_I(\xi)\) to \(EL(\xi)\).  This
  proves~\eqref{eq:diagram_for_ideals_and_essential}.
\end{proof}

\begin{example}[\cite{Kwasniewski-Meyer:Essential}*{Example~4.7}]
  \label{exa:group_bundle}
  Let \(S \defeq G \cup \{0\}\) be the inverse semigroup obtained by
  adjoining a zero element to an amenable discrete group~\(G\).
  Let~\(G\) act on \(A=\Cont[0,1]\) by \(\Hilm_g = \Cont[0,1]\) for
  \(g\in G\) and \(\Hilm_0 = \Cont_0(0,1]\), equipped with the usual
  involution and multiplication maps.  Then \(A\rtimes_\ess S= A\).
  Every ideal~\(I\) in~\(A\) is \(\Hilm\)\nb-invariant and
  \(I\rtimes_\ess S= I\), and so
  \(I\rtimes_\ess S\subseteq A\rtimes_\ess S\).  However, if
  \(I\defeq \Cont_0(0,1]\), then
  \(A/I\rtimes_\ess S =\C\rtimes G=\Cst(G)\) and the quotient
  homomorphism~\(\kappa\) from~\(\Hilm\) onto~\(\Hilm|_{A/I}\) does
  not induce a map from~\(A\) to~\(\Cst(G)\).
\end{example}

\begin{definition}\label{defn:essential exactness}
  We call the action~\(\Hilm\) \emph{essentially exact} if for each
  \(I\in \Ideals^{\Hilm}(A)\) there is a \Star{}homomorphism
  \(\kappa\rtimes_\ess S\colon A\rtimes_\ess S \to A/I\rtimes_\ess
  S\) whose restriction to each fibre~\(\Hilm_t\) is the quotient
  map onto \(\Hilm_t/\Hilm_tI\), \(t\in S\), and the kernel of
  \(\kappa\rtimes_\ess S\) is
  \(\iota\rtimes_\ess S(I\rtimes_\ess S)\).
\end{definition}
	\begin{remark}\label{rem:essential_for_closed_actions}
  When the action~\(\Hilm\) is closed, then the reduced and
  essential crossed products coincide for all restrictions of
  \(\Hilm\) (see Proposition~\ref{pro:closed_unit_isg_extension}).
  Thus essential exactness is the same as exactness for closed
  actions of inverse semigroups and for Fell bundles over Hausdorff
  groupoids.
\end{remark}
\begin{example}[Essentially exact Fell bundles]\label{exm:essential_exact_Fell_bundle}
  Consistently with Example
  \ref{ex:exactness_fell_bundles_groupoids}, we will call a Fell
  bundle \(\A=(A_\gamma)_{\gamma\in \Gr}\) over an étale groupoid
  essentially exact if the corresponding action~\(\Hilm\) is
  essentially exact.  More specifically, by
  Lemma~\ref{lem:essential_crossed_product_ideal_embedding}, for any
  \(\Gr\)\nb-invariant ideal~\(I\) in~\(A\), the inclusion
  \(\Contc(\A|_{I}) \subseteq \Contc(\A)\) extends to an injective
  \Star{}homomorphism \(\Cst_\ess(\A|_{I}) \into \Cst_\ess(\A)\).
  So \(\A\) is \emph{essentially exact} if and only if, for every
  \(\Gr\)\nb-invariant ideal~\(I\) in~\(A\), restriction of sections
  gives a well-defined \Star{}homomorphism
  \(\Cst_\ess(\A) \onto \Cst_\ess(\A|_{A/I})\) and the following
  sequence is exact:
  \[
    \Cst_\ess(\A|_{I})
    \into \Cst_\ess(\A)
    \onto \Cst_\ess(\A|_{A/I}).
  \]
\end{example}
If the \(S\)\nb-action on~\(A\) is residually aperiodic, then~\(A\)
separates ideals in~\(A\rtimes_\ess S\) if and only if the
\(S\)\nb-action on~\(A\) is essentially exact (see
Theorem~\ref{thm:residual_aperiodic_action} below).  In
Example~\ref{exa:group_bundle}, however, \(A\) separates ideals
in~\(A\rtimes_\ess S\) although the \(S\)\nb-action on~\(A\) is not
essentially exact.  The following proposition shows that the
\(S\)\nb-action on~\(A\) must be essentially exact if~\(A\)
separates ideals in \(A\rtimes_\red S\).

\begin{proposition}
  \label{prop:separation_of_ideals_in_reduced}
  The following are equivalent:
  \begin{enumerate}
  \item \label{enu:separation_of_ideals_in_reduced1}%
    \(A\) separates
    ideals in the reduced crossed product \(A\rtimes_\red S\);
  \item \label{enu:separation_of_ideals_in_reduced2}%
    the action is exact and for each \(I\in\Ideals^{\Hilm}(A)\),
    \(A/I\) detects ideals in \(A/I\rtimes_\red S\).
  \end{enumerate}
  If these equivalent conditions hold, then
  \(A/I\rtimes_\red S=A/I\rtimes_\ess S\) for every
  \(I\in\Ideals^{\Hilm}(A)\) and the action is essentially exact.
\end{proposition}

\begin{proof}
  Lemma~\ref{lem:separate_induced} and
  Proposition~\ref{prop:invariant_ideals_in_regular} show that~\(A\)
  separates ideals in \(A\rtimes_\red S\) if and only if~\(A/I\)
  detects ideals in
  \(A\rtimes_\red S/\iota_\red \rtimes S(I\rtimes_\red S)\) for all
  \(I\in\Ideals^{\Hilm}(A)\).  Then the action is exact because the
  kernel of \(\kappa \rtimes_{\red} S\) has to be
  \(\iota_\red \rtimes S(I\rtimes_\red S)\), and then
  \(A\rtimes_\red S/\iota_\red \rtimes S(I\rtimes_\red S) \cong
  A/I\rtimes S\) for all \(I\in\Ideals^{\Hilm}(A)\).  Thus
  \ref{enu:separation_of_ideals_in_reduced1}
  and~\ref{enu:separation_of_ideals_in_reduced2} are equivalent.
  Condition~\ref{enu:separation_of_ideals_in_reduced2} implies
  \(A/I\rtimes_\red S=A/I\rtimes_\ess S\) because otherwise
  \(A/I\rtimes_\ess S\) would be a quotient of \(A/I\rtimes_\red S\)
  by a nonzero ideal not detected by \(A/I\).
\end{proof}

We illustrate by an example what can go wrong with the exactness of
essential crossed products.  Our example is closely related to the
Reeb foliation or, more precisely, to its restriction to a
transversal.

\begin{example}
  Let \(\vartheta\colon \R\to\R\) be a homeomorphism with
  \(\vartheta(t) = t\) for \(t\le 0\) and \(\vartheta(t) > t\) for
  all \(t>0\).  Let~\(G\) be the germ groupoid of the transformation
  groupoid \(\R\times_\vartheta\Z\).  We claim that
  \(\Cst_\red(G) \cong \Cst_\ess(G)\).  To see this, we use that the
  restrictions of~\(G\) to \([0,\infty)\) and~\((-\infty,0)\) are
  Hausdorff.  Indeed, \(G|_{(0,\infty)}\) is the Hausdorff
  transformation groupoid \((0,\infty) \rtimes_\vartheta \Z\)
  because~\(\Z\) acts freely (and properly) on~\((0,\infty)\).
  Similarly, \(G|_{[0,\infty)}\) is the Hausdorff transformation
  groupoid \([0,\infty) \rtimes_\vartheta \Z\); the action of~\(\Z\)
  on~\([0,\infty)\) fixes~\(0\), but the germ of any \(n\in\Z\)
  at~\(0\) is nontrivial because~\(\vartheta^n\) acts nontrivially
  on \((0,\infty)\).  Therefore, the support of any nonzero element
  of \(\Cst_\red(G)\) must intersect \([0,\infty)\) and
  \((-\infty,0)\) in relatively open subsets.  Hence the support
  cannot be meagre, and this proves our claim.  This equality of
  reduced and essential groupoid \(\Cst\)\nb-algebras for~\(G\) is
  not inherited by the restriction to the closed invariant
  subset~\((-\infty,0]\).  Indeed, the restriction of~\(G\) to
  \((-\infty,0]\) is the non-Hausdorff group bundle with trivial
  fibre over \((-\infty,0)\) and the fibre~\(\Z\) at~\(0\) (see
  \cite{Kwasniewski-Meyer:Essential}*{Example~4.7} and
  Example~\ref{exa:group_bundle}).  In this case, the full and
  reduced groupoid \(\Cst\)\nb-algebras are obtained by gluing
  together \(\Cont_0((-\infty,0))\) and the fibre \(\Cst(\Z)\)
  at~\(0\).  However,
  \( \Cst_\ess(G|_{(-\infty,0]})=\Cont_0((-\infty,0])\), so
  \[
    \Cst_\red(G|_{(-\infty,0]}) \not\cong \Cst_\ess(G|_{(-\infty,0]}).
  \]
  Since the essential and reduced crossed products coincide for
  \(G\) and~\(G|_{(0,\infty)}\), but not for~\(G|_{(-\infty,0]}\),
  the following sequence of essential crossed products exists, but fails to be
  exact:
  \[
    0 \to \Cst_\ess(G|_{(0,\infty)}) \to \Cst_\ess(G)
    \to \Cst_\ess(G|_{(-\infty,0]}) \to 0.
  \]
  This is one way how essential crossed products may fail to be
  exact.  The restriction~\(G|_{\{0\}}\) is simply the group~\(\Z\).
  So
  \(\Cst_\red(G|_{\{0\}}) = \Cst_\ess(G|_{\{0\}}) \cong \Cst(\Z)\).
  The restriction \Star{}homomorphism
  \(\Cst(G|_{(-\infty,0]}) \to \Cst(G|_{\{0\}})\) does not descend
  to the essential crossed products.  That is, there is no canonical
  map from \(\Cst_\ess(G|_{(-\infty,0]})\) to
  \(\Cst_\ess(G|_{\{0\}})\).  This is the second way how essential
  crossed products may fail to be exact.

  Since \(G|_{[0,\infty)}\cong [0,\infty) \rtimes_\vartheta \Z\), it
  follows that
  \(\Cst(G|_{[0,\infty)}) \cong \Cst_\red(G|_{[0,\infty)})\).  By a
  diagram, it follows that the sequence
  \[
    0 \to \Cst_\red(G|_{(-\infty,0)})
    \to \Cst_\red(G)
    \to \Cst_\red(G|_{[0,\infty)})
    \to 0
  \]
  is exact.  Since also
  \(\Cst(G|_{(-\infty,0)}) \cong \Cst_\red(G|_{(-\infty,0)})\), the
  Five Lemma shows that \(\Cst(G) \cong \Cst_\red(G)\).  A similar
  argument shows \(\Cst(G|_U) \cong \Cst_\red(G|_U)\) for all
  locally closed \(G\)\nb-invariant subsets \(U\subseteq\R\).
  Therefore, \(G\) is inner exact.
\end{example}

\subsection{Amenability vs exactness}
\label{sec:amenability}

\begin{definition}
  Let~\(\Hilm\) be an \(S\)\nb-action by Hilbert bimodules on a
  \(\Cst\)\nb-algebra~\(A\).  We call the action~\(\Hilm\)
  \emph{amenable} if the regular representation
  \(\Lambda\colon A\rtimes S \to A\rtimes_\red S\) is an
  isomorphism.
\end{definition}

\begin{remark}
  By~\eqref{eq:kernel}, the action~\(\Hilm\) is amenable if and only
  if the weak conditional expectation \(E\colon A\rtimes S \to A''\)
  is faithful.
\end{remark}

\begin{lemma}
  \label{lem:permanence_of_amenability}
  Let~\(\Hilm\) be an \(S\)\nb-action by Hilbert bimodules on a
  \(\Cst\)\nb-algebra~\(A\) and let \(I\in\Ideals^{\Hilm}(A)\) be an
  invariant ideal.  If \(\Hilm|_{A/I}\) is amenable, then
  \[
    I\rtimes_\red S \into A\rtimes_\red S \onto A/I\rtimes_\red S
  \]
  is exact.  If this sequence is exact, then~\(\Hilm\) is amenable
  if and only if \(\Hilm|_I\) and~\(\Hilm|_{A/I}\) are amenable.
\end{lemma}

\begin{proof}
  There is a commutative diagram
  \[
    \begin{tikzcd}
      I\rtimes S \ar [r, rightarrowtail, "\iota \rtimes S"] \ar[d, twoheadrightarrow,"\Lambda^I"] &
      A\rtimes S  \ar[r, twoheadrightarrow ,"\kappa \rtimes S"] \ar[d, twoheadrightarrow,"\Lambda^A"] &
      A/I\rtimes S   \ar[d, twoheadrightarrow,"\Lambda^{A/I}"] 
      \\
      I\rtimes_\red S \ar[r, rightarrowtail, "\iota \rtimes_\red S"'] &
      A\rtimes_\red S  \ar[r, twoheadrightarrow ,"\kappa \rtimes_\red S"'] &
      A/I\rtimes_\red S,
    \end{tikzcd}
  \]
  where the top horizontal sequence is exact, and \(\Lambda^I\),
  \(\Lambda^A\), \(\Lambda^{A/I}\) denote the respective regular
  representations.  If \(\Lambda^{A/I}\) is injective, then
  \[
    \ker(\kappa \rtimes_\red S)=\Lambda^A\left(\ker(\kappa \rtimes S)\right)=\Lambda^A\left(\iota \rtimes S(I\rtimes S )\right)
    =\iota \rtimes_\red S(I\rtimes_\red S ).
  \]
  Hence
  \(I\rtimes_\red S \into A\rtimes_\red S \onto A/I\rtimes_\red S\)
  is exact.  In general, if this sequence is exact, the Snake Lemma
  from homological algebra yields a short exact sequence
  \[
    \ker(\Lambda^I)
    \into \ker(\Lambda^A)
    \onto \ker(\Lambda^{A/I}).
  \]
  Hence \(\ker(\Lambda^A)=0\) if and only if \(\ker(\Lambda^I)=0\)
  and \(\ker(\Lambda^{A/I})=0\).
\end{proof}

\begin{remark}
  It is unclear whether the quotient action~\(\Hilm|_{A/I}\) for
  \(I\in\Ideals^{\Hilm}(A)\) is amenable if~\(\Hilm\) is (the proof
  of \cite{Kwasniewski-Szymanski:Pure_infinite}*{Lemma~3.9} is
  incorrect).  If so, then by
  Lemma~\ref{lem:permanence_of_amenability}, the amenable action is
  also exact.  Let \(S=G\) be a group.  If~\(\Hilm\) satisfies the
  approximation property in
  \cite{Exel:Amenability}*{Definition~4.4}, then also every
  restriction~\(\Hilm|_{A/I}\) satisfies the approximation property.
  Hence the approximation property of a Fell bundle over a group
  implies both amenability and exactness.  In particular, if there
  is an amenable group action~\(\Hilm\) such~that \(\Hilm|_{A/I}\)
  is not amenable for some \(I\in\Ideals^{\Hilm}(A)\), then the
  approximation property is strictly stronger then amenability (see
  also \cite{Exel:Amenability}*{Page~169}).
\end{remark}

\begin{example}
  \label{ex:amenable_groupoids_actions}
  Let \(\A=(A_\gamma)_{\gamma\in \Gr}\) be a Fell bundle over an
  étale, locally compact, Hausdorff groupoid~\(\Gr\).  Suppose
  that~\(\Gr\) is second countable and that each fibre~\(A_\gamma\)
  is separable.  If~\(\Gr\) is amenable, then the regular
  representation \(\Cst(\A)\onto \Cst_\red(\A)\) is an isomorphism
  by
  \cite{Sims-Williams:Amenability_for_Fell_bundles_over_groupoids}*{Theorem~1}.
  This applies also to every Fell bundle~\(\A_{A/I}\) over~\(\Gr\)
  where~\(I\) is a \(\Gr\)\nb-invariant ideal in \(A\).  Thus
  if~\(\Gr\) is amenable, then~\(\A\) is exact in the sense
  described in Example~\ref{ex:exactness_fell_bundles_groupoids}.
\end{example}

\begin{example}
  \label{ex:nuclearity_implies_amenabilty}
  If~\(\Hilm\) is a closed action of a countable inverse
  semigroup~\(S\) on a separable, commutative \(\Cst\)\nb-algebra
  \(A=\Cont_0(X)\) such that \(A\rtimes_\red S\) is nuclear, then
  the action~\(\Hilm\) is amenable and exact.  Indeed, the
  action~\(\Hilm\) corresponds to a twisted étale, locally compact,
  Hausdorff groupoid \((\Gr,\Sigma)\), where \(\Gr= X\rtimes S\) is
  the dual groupoid of~\(\Hilm\).  Equivalently, this corresponds to
  a Fell line bundle over~\(\Gr\).  Then~\(\Gr\) is amenable by
  \cite{Takeishi:Nuclearity}*{Theorem~5.4}.  Hence our claim follows
  from Example~\ref{ex:amenable_groupoids_actions}.  In particular,
  a twisted étale, locally compact, second countable, Hausdorff
  groupoid \((\Gr,\Sigma)\) with nuclear \(\Cst_\red(\Gr,\Sigma)\)
  is exact in the sense of
  Example~\ref{ex:exactness_twisted_groupoids}.
\end{example}

\section{Ideals and pure infiniteness for inverse semigroup crossed
  products}
\label{sec:crossed_infinite}

In this section we present efficient criteria for separation of
ideals and pure infiniteness in essential crossed products.

\subsection{Residual  aperiodicity and ideal structure}

Let \(\Hilm=(\Hilm_t)_{t\in S}\) be an action of a unital inverse
semigroup~\(S\) by Hilbert bimodules on a \(\Cst\)\nb-algebra~\(A\).

\begin{definition}[\cite{Kwasniewski-Meyer:Essential}*{Definition~6.1}]
  The action \(\Hilm=(\Hilm_t)_{t\in S}\) is \emph{aperiodic} if the
  Hilbert \(A\)\nb-bimodules \(\Hilm_t \cdot I_{1,t}^\bot\) are
  aperiodic for all \(t\in S\), where~\(I_{1,t}\) is defined
  in~\eqref{eq:Itu}.
\end{definition}

\begin{proposition}[\cite{Kwasniewski-Meyer:Essential}*{Proposition~6.3}]
  \label{pro:aperiodic_actions_vs_inclusions}
  Let~\(B\) be an \(S\)\nb-graded \(\Cst\)\nb-algebra with a grading
  \(\Hilm=(\Hilm_t)_{t\in S}\).  Let \(A\defeq \Hilm_1\subseteq B\)
  and turn~\(\Hilm\) into an \(S\)\nb-action on~\(A\).  If this
  action is aperiodic, then the inclusion \(A\subseteq B\) is
  aperiodic.  The converse holds if the grading on~\(B\) is
  topological, that is, the canonical quotient map
  \(A\rtimes S \onto A\rtimes_\ess S\) factors through \(A\onto B\).
\end{proposition}

\begin{definition}[\cite{Kwasniewski-Meyer:Essential}*{Definition~2.20}]
  Let~\(\Gr\) be an étale groupoid and \(X\subseteq \Gr\) its unit
  space.  The \emph{isotropy group} of a point \(x\in X\) is
  \(\Gr(x)\defeq \s^{-1}(x)\cap \rg^{-1}(x)\subseteq \Gr\).  We
  call~\(\Gr\) \emph{topologically free} if, for every open
  \(U\subseteq \Gr\setminus X\), the set
  \(\setgiven{x\in X}{\Gr(x)\cap U\neq \emptyset}\) has empty
  interior.
\end{definition}

\begin{remark}
  A groupoid \(\Gr\) is \emph{effective} if any open subset
  \(U\subseteq \Gr\) with \(\rg|_U = \s|_U\) is contained in~\(X\).
  Effective groupoids are topologically free.  The converse
  implication holds if~\(X\) is closed in~\(\Gr\), but not in
  general.
\end{remark}

Aperiodicity is related in \cites{Kwasniewski-Meyer:Aperiodicity,
  Kwasniewski-Meyer:Aperiodicity_pseudo_expectations,
  Kwasniewski-Meyer:Essential} to topological freeness and several
other conditions.  We now introduce residual versions of
aperiodicity and topological freeness.

\begin{definition}
  \label{def:isg_residual_aperiodic}
  The action~\(\Hilm\) is \emph{residually aperiodic} if, for each
  \(I\in\Ideals^{\Hilm}(A)\), the restricted action \(\Hilm|_{A/I}\)
  is aperiodic.
\end{definition}

\begin{definition}
  An étale groupoid~\(\Gr\) with unit space \(X\subseteq \Gr\) is
  \emph{residually topologically free} if, for each nonempty closed
  \(\Gr\)\nb-invariant subset \(Y\subseteq X\), the restricted
  groupoid \(r^{-1}(Y)=s^{-1}(Y)\) is topologically free.
\end{definition}

The following lemma may help to show that a transformation groupoid
is (residually) topologically free.

\begin{lemma}
  \label{lem:top_free_inheritance}
  Let~\(\Gr\) be an étale groupoid.  Let \(X\) and~\(Y\) be
  topological spaces with \(\Gr\)\nb-actions and let
  \(f\colon X\to Y\) be a continuous \(\Gr\)\nb-equivariant map.
	\begin{enumerate}
	\item\label{enu:top_free_inheritance1} if \(Y\) is closed in \(\Gr\ltimes Y\), then \(X\) is closed in  \(\Gr\ltimes X\)
	\item\label{enu:top_free_inheritance2} if \(f\colon X\to Y\) is open and \(\Gr\ltimes Y\) is topologically free, then so is
  \(\Gr\ltimes X\).
	\item\label{enu:top_free_inheritance3} if \(f\colon X\to Y\) is open, maps closed \(\Gr\)-invariant sets to closed sets, and \(\Gr\ltimes Y\) is residually topologically free, then so is
  \(\Gr\ltimes X\).
	\end{enumerate}
\end{lemma}

\begin{proof}
  The map~\(f\) induces a continuous  groupoid homomorphism
  \(f_*\colon \Gr\ltimes X \to \Gr \ltimes Y\) such that \(f_*^{-1}(Y)=X\).  So
	\ref{enu:top_free_inheritance1} follows. The map \(f_*\) is open if and only if~\(f\) is open.
	Assume this.
	To show \ref{enu:top_free_inheritance2} assume that \(\Gr\ltimes Y\) is topologically free and let
  \(U\subseteq (\Gr\ltimes X)\setminus X\).  Then \(f_*(U)\) is open
  and contained in \((\Gr\ltimes Y)\setminus Y\).  Let
  \(I_U \defeq \setgiven{x\in X}{\Gr(x) \cap U \neq \emptyset}\) and
  \(I_{f^*(U)} \defeq \setgiven{y\in Y}{\Gr(y) \cap f_*(U) \neq
    \emptyset}\).  We claim that \(f(I_U) = I_{f^*(U)}\).  To see
  this, let \(x\in I_U\).  Then there is an arrow \((x,g,x) \in U\);
  here \(g\in\Gr\) is such that \(\varrho(x) = \s(g)\) and
  \(g\cdot x = x\).  Since~\(f\) is \(\Gr\)\nb-equivariant, then
  \((f(x),g,f(x)) \in f_*(U)\).  This witnesses that
  \(f(x) \in I_{f^*(U)}\).  Since \(\Gr\ltimes Y\) is topologically
  free, \(I_{f^*(U)}\) has empty interior.  Since~\(f\) is open, the
  preimage of \(I_{f^*(U)}\) in~\(X\) has empty interior as well.  It follows
  that~\(I_U\) has empty interior.  This witnesses
  that~\(\Gr\ltimes X\) is topologically free and proves \ref{enu:top_free_inheritance2}.
	
	Finally, for any \(\Gr\)-invariant set \(D\subseteq X\) the set  \(f(D)\) is \(\Gr\)-invariant. If we assume \(f(D)\) is closed  in \(Y\) and \(\Gr\ltimes Y\) is residually topologically free,
		then  \(\Gr\ltimes f(D)\) is topologically free. Since the restriction of \(f_*\) to \(\Gr\ltimes D\) is a continuous  open map onto \(\Gr\ltimes f(D)\),
		\ref{enu:top_free_inheritance2} implies that \(\Gr\ltimes D\) is topologically free. This
		proves \ref{enu:top_free_inheritance3}.
\end{proof}

\begin{theorem}
  \label{the:residually_aperiodic_characterisations}
  Let~\(\Hilm\) be an action of a unital inverse semigroup~\(S\) on
  a \(\Cst\)\nb-algebra~\(A\) by Hilbert bimodules.  If~\(A\) is
  separable or of Type~I, then the following are equivalent:
  \begin{enumerate}
  \item \label{en:residually_aperiodic_characterisations1}%
    the dual groupoid \(\dual{A}\rtimes S\) is residually
    topologically free;
  \item \label{en:residually_aperiodic_characterisations2}%
    the action~\(\Hilm\) is residually aperiodic;
  \item \label{en:residually_aperiodic_characterisations3}%
    for any \(I\in\Ideals^{\Hilm}(A)\), the full crossed product for
    the restricted action~\(\Hilm|_{A/I}\) has a unique
    pseudo-expectation namely, the canonical
    \(\Locmult\)-expectation;
  \item \label{en:residually_aperiodic_characterisations4}%
    for any \(I\in\Ideals^{\Hilm}(A)\), \((A/I)^+\) supports~\(C\)
    for each intermediate \(\Cst\)-subalgebra
    \(A/I\subseteq C \subseteq A/I\rtimes_\ess S\) for the
    restricted action~\(\Hilm|_{A/I}\);
  \item \label{en:residually_aperiodic_characterisations5}%
    for any \(I\in\Ideals^{\Hilm}(A)\), \(A/I\) detects ideals
    in each intermediate \(\Cst\)\nb-subalgebra
    \(A/I\subseteq C \subseteq A/I\rtimes_\ess S\) for the
    restricted action~\(\Hilm|_{A/I}\).
  \end{enumerate}
  For general~\(A\),
  \ref{en:residually_aperiodic_characterisations1}\(\Rightarrow\)%
  \ref{en:residually_aperiodic_characterisations2}\(\Rightarrow\)%
  \ref{en:residually_aperiodic_characterisations3}\(\Rightarrow\)%
  \ref{en:residually_aperiodic_characterisations5}
  and
  \ref{en:residually_aperiodic_characterisations2}\(\Rightarrow\)%
  \ref{en:residually_aperiodic_characterisations4}\(\Rightarrow\)%
  \ref{en:residually_aperiodic_characterisations5}.
\end{theorem}

\begin{proof}
  A subset of~\(\dual{A}\) is closed and invariant if and only if it
  is of the form \(X = \dual{A/I}\) for an \(\Hilm\)\nb-invariant
  ideal~\(I\).  The dual groupoid of the induced action on~\(A/I\)
  is the restriction \(X\rtimes S\) of the dual groupoid to~\(X\).
  Hence the implications
  \ref{en:residually_aperiodic_characterisations1}%
  \(\Rightarrow\)\ref{en:residually_aperiodic_characterisations2}\(\Rightarrow\)%
  \ref{en:residually_aperiodic_characterisations3},
  \ref{en:residually_aperiodic_characterisations4} follow from
  \cite{Kwasniewski-Meyer:Aperiodicity_pseudo_expectations}*{Corollary~4.8
    and Theorem~3.6}, applied to the quotients~\(A/I\) and
  intermediate \(\Cst\)\nb-algebras.
  Lemma~\ref{lem:support_implies_separate} shows
  that~\ref{en:residually_aperiodic_characterisations4}
  implies~\ref{en:residually_aperiodic_characterisations5}, and
  \cite{Pitts-Zarikian:Unique_pseudoexpectation}*{Theorem~3.5} shows
  that~\ref{en:residually_aperiodic_characterisations3}
  implies~\ref{en:residually_aperiodic_characterisations5}.
  If~\(A\) is separable or of Type~I, so are its quotients~\(A/I\),
  and then
  \cite{Kwasniewski-Meyer:Aperiodicity_pseudo_expectations}*{Proposition~6.1}
  shows that~\ref{en:residually_aperiodic_characterisations5}
  implies~\ref{en:residually_aperiodic_characterisations1}.
\end{proof}

The following result justifies introducing the notion of essential
exactness -- it is an instance of
condition~\ref{it:residual_faithful+aperiodic_gives_supports1} in
Theorem~\ref{thm:residual_faithful+aperiodic_gives_supports}.

\begin{theorem}
  \label{thm:residual_aperiodic_action}
  Let~\(B\) be an \(S\)\nb-graded \(\Cst\)\nb-algebra with a grading
  \(\Hilm=(\Hilm_t)_{t\in S}\) that forms a residually aperiodic
  action of~\(S\) on \(A\defeq \Hilm_1\) \textup{(}this holds if the dual
  groupoid \(\dual{A}\rtimes S\) is residually topologically free\textup{)}.
  The following are equivalent:
  \begin{enumerate}
  \item \label{en:residual_aperiodic_action1}%
    \(B\cong A\rtimes_\ess S\) and~\(\Hilm\) is essentially exact;
  \item \label{en:residual_aperiodic_action2}%
    \(A\) separates ideals in~\(B\), so \(\Ideals(B)\cong \Ideals^{\Hilm}(A)\);
  \item \label{en:residual_aperiodic_action3}%
    \(A^+\) residually supports~\(B\);
  \item \label{en:residual_aperiodic_action4}%
    \(A^+\) fills~\(B\).
  \end{enumerate}
  If the above equivalent conditions hold and the primitive ideal
  space~\(\check{A}\) is second countable, then the quasi-orbit map
  induces a homeomorphism \(\check{B}\cong \check{A}/{\sim}\), where
  \(\check{A}/{\sim}\) is the quasi-orbit space of the dual groupoid
  \(\check{A}\rtimes S\), that is, \(\prid_1, \prid_2\in\check{A}\)
  satisfy \(\prid_1 \sim \prid_2\) if and only if
  \(\cl{(\check{A}\rtimes S)\cdot \prid_1} = \cl{(\check{A}\rtimes
    S)\cdot \prid_2}\).
\end{theorem}

\begin{proof}
  The \(\Cst\)\nb-inclusion \(A\subseteq B\) is symmetric and
  residually aperiodic by Propositions
  \ref{prop:invariant_ideals_in_regular} and
  \ref{pro:aperiodic_actions_vs_inclusions}.  Hence by
  Theorem~\ref{thm:residual_faithful+aperiodic_gives_supports}
  conditions
  \ref{en:residual_aperiodic_action2}--\ref{en:residual_aperiodic_action4}
  are equivalent to the condition that for each
  \(I\in \Ideals^B(A)=\Ideals^{\Hilm}(A)\) the unique
  pseudo-expectation \(E^I\colon B/BIB \to A/I\) is almost faithful.
  Let us assume this.  Then there is a commutative diagram
  \[
    \begin{tikzcd}
      A\rtimes  S/I\rtimes S\cong A/I\rtimes S  \ar[r, "\Psi"]
      \ar[d,"EL^I"] &
      B/BIB \ar[d,"E^I"] & \\
      \Locmult(A/I) \ar[r, "\hookrightarrow"'] &
      \hull(A/I),
    \end{tikzcd}
  \]
  where~\(\Psi\) is the homomorphism that exists by universality of
  \(A/I\rtimes S\) because \(B/BIB\) is graded by \(\Hilm|_{A/I}\)
  by Lemma \ref{lem:quotients_of_regular}, \(EL^I\) is the canonical
  essential expectation for \(A/I\rtimes S\), and
  \(\Locmult(A/I)\hookrightarrow \hull(A/I)\) is the canonical
  embedding.  The diagram commutes because the inclusion
  \(A/I\subseteq A/I\rtimes S\) is aperiodic and hence there is a
  unique pseudo-expectation by
  \cite{Kwasniewski-Meyer:Aperiodicity_pseudo_expectations}*{Theorem
    3.6}.  As a consequence,
  \(\ker\Psi={\Psi}^{-1}(0)={\Psi}^{-1}(\Null_{E^I})=\Null_{EL}\)
  and thus \(\Psi\) factors through an isomorphism
  \(A/I\rtimes_{\ess} S\cong B/BIB\).  Since this holds for every
  \(I\in \Ideals^{\Hilm}(A)\) we get that \(B\cong A\rtimes_\ess S\)
  and that~\(\Hilm\) is essentially exact
  (\(A\rtimes_\ess S/ I\rtimes_{\ess} S\cong B/BIB \cong
  (A/I)\rtimes_{\ess} S\) for \(I\in \Ideals^{\Hilm}(A)\)).

  Theorem~\ref{thm:residual_faithful+aperiodic_gives_supports}
  implies easily that~\ref{en:residual_aperiodic_action1} implies
  \ref{en:residual_aperiodic_action2}.  This finishes the proof that
  the four conditions are equivalent.  The remaining claims follow
  mostly from the last part of
  Theorem~\ref{thm:residual_faithful+aperiodic_gives_supports}.
  That the quasi-orbit space has the asserted form follows from
  \cite{Kwasniewski-Meyer:Stone_duality}*{Theorem~6.22}.
\end{proof}

\begin{corollary}
  \label{cor:Ideals_in_section_algebras}
  Let \(\A=(A_\gamma)_{\gamma\in \Gr}\) be a Fell bundle over an
  étale groupoid~\(\Gr\) with locally compact, Hausdorff unit
  space~\(X\), and put \(A=\Cont_0(\A|_X)\).  Define the dual
  groupoid~\(\dual{A}\rtimes \Gr\) as in
  Example~\textup{\ref{exa:transformation_groupoid_Fell}}.  Assume
  that it is residually topologically free \textup{(}this holds, for
  instance, if~\(G\) is residually topologically free and the base
  map for the \(\Cst\)-bundle~\(\A\) is open and closed\textup{)}.
  Assume also one of the following
  \begin{enumerate}
  \item \label{enu:Ideals_in_section_algebras1}%
    \(B\defeq \Cst_\ess(\A)\) and~\(\A\) is essentially exact;
  \item \label{enu:Ideals_in_section_algebras2}%
    \(B\defeq \Cst_\red(\A)\), \(\A\) is exact, and the unit space
    in~\(\check{A}\rtimes \Gr\) is closed \textup{(}the latter is
    automatic if~\(G\) is Hausdorff\textup{)};
  \item \label{enu:Ideals_in_section_algebras3}%
    \(B\defeq \Cst(\A)\), \(\A\) is separable, and~\(G\) is amenable
    and Hausdorff.
  \end{enumerate}
  Then~\(A\) separates ideals in~\(B\) and, even more, \(A^+\)
  fills~\(B\).  The lattice \(\Ideals(B)\) is naturally isomorphic
  to the lattice of \(G\)\nb-invariant ideals in~\(A\).  If, in
  addition, \(\check{A}\) is second countable, then
  \(\check{B}\cong \check{A}/{\sim}\), where \(\check{A}/{\sim}\) is
  the quasi-orbit space of the dual groupoid
  \(\check{A}\rtimes \Gr\), that is,
  \(\prid_1, \prid_2\in\check{A}\) satisfy \(\prid_1 \sim \prid_2\)
  if and only if
  \(\cl{(\check{A}\rtimes \Gr)\cdot \prid_1} = \cl{(\check{A}\rtimes
    \Gr)\cdot \prid_2}\).
\end{corollary}

\begin{proof}
The claims in brackets follow from Lemma \ref{lem:top_free_inheritance}.
  The claim in case~\ref{enu:Ideals_in_section_algebras1}
  follows from Theorem~\ref{thm:residual_aperiodic_action} (see also
  Example~\ref{exm:essential_exact_Fell_bundle}).
  Case~\ref{enu:Ideals_in_section_algebras2} follows
  from~\ref{enu:Ideals_in_section_algebras1} and Remarks
  \ref{rem:closed_actions_vs_groupoids}
  and~\ref{rem:essential_for_closed_actions}.
  Example~\ref{ex:amenable_groupoids_actions} explains why
  \ref{enu:Ideals_in_section_algebras3} is a special case
  of~\ref{enu:Ideals_in_section_algebras2}.
\end{proof}

Theorem~\ref{thm:residual_aperiodic_action} allows us to describe
the ideal structure of~\(B\) in terms of~\(A\) under the following
assumptions:
\begin{equation}
  \label{standing_assumption}
  \Hilm\text{ is  an essentially exact, residually aperiodic
    action and } B= A\rtimes_\ess S.
\end{equation}
We are going to study whether~\(B\) is purely infinite using the
same assumption.  Before we do this, we
simplify~\eqref{standing_assumption} in the presence of a
conditional expectation.

\begin{proposition}
  \label{prop:residual_aperiodic_Cartan}
  If there is a genuine conditional expectation
  \(E\colon B\to A\subseteq B\), then \eqref{standing_assumption} is
  equivalent to
  \begin{equation}
    \label{standing_assumption0}
    \text{\(\Hilm\)  is closed, exact, residually aperiodic  and
      \(B= A\rtimes_\red S\).}
  \end{equation}
  For any \(\Cst\)\nb-inclusion \(A\subseteq B\), an
  action~\(\Hilm\) as in~\eqref{standing_assumption0}
  exists if and only if \(A\subseteq B\) is regular, residually
  aperiodic and there is a conditional expectation
  \(E\colon B\to A\) which is residually faithful in the sense
  that~\(E\) descends to a faithful conditional expectation
  \(E^I\colon B/BIB \to A/I\) for any \(I\in \Ideals^B(A)\)
  \textup{(}see
  Lemma~\textup{\ref{lem:symmetric_quotients_conditional_expectations})}.
\end{proposition}

\begin{proof}
  \eqref{standing_assumption0} implies~\eqref{standing_assumption}
  by Remark~\ref{rem:essential_for_closed_actions}.  Conversely,
  assume~\eqref{standing_assumption}.  Then \(E\colon B\to A\) is
  the unique pseudo-expectation for \(A\subseteq B\) by
  \cite{Kwasniewski-Meyer:Aperiodicity_pseudo_expectations}*{Theorem~3.6}.
  Hence \(E=EL\), \(\Hilm\) is closed and
  \(A\rtimes_\red S=A\rtimes_\ess S\).  The same argument works for
  all the restrictions~\(\Hilm^I\) and quotient inclusions
  \(A/I\subseteq B/BIB\) for
  \(I\in \Ideals^{\Hilm}(A)=\Ideals^B(A)\).  This
  gives~\eqref{standing_assumption0}.  Combining this reasoning with
  \cite{Kwasniewski-Meyer:Cartan}*{Theorem~6.3} also gives the
  second part of the assertion.
\end{proof}

\begin{corollary}
  \label{cor:residual_aperiodic_Cartan}
  If~\(A\) is type~\(I\), then an action~\(\Hilm\) as
  in~\eqref{standing_assumption0} exists if and only if
  \(A\subseteq B\) is residually Cartan, that is, for each
  \(I\in \Ideals^B(A)\), \(A/I\subseteq B/BIB\) is a noncommutative
  Cartan subalgebra in the sense of Exel~\cite{Exel:noncomm.cartan}.
\end{corollary}

\begin{proof}
  Combine the second part of
  Proposition~\ref{prop:residual_aperiodic_Cartan} and
  \cite{Kwasniewski-Meyer:Cartan}*{Theorem~6.3}.
\end{proof}

\subsection{Pure infiniteness criteria}
\label{subsection:Pure_infiniteness_criteria}

In this section, we assume that~\(B\) is an \(S\)\nb-graded
\(\Cst\)\nb-algebra with a grading \(\Hilm=(\Hilm_t)_{t\in S}\)
and~\(A\) is the unit fibre of the grading.  We give pure
infiniteness criteria for~\(B\) under the
assumption~\eqref{standing_assumption}.  This
covers~\eqref{standing_assumption0} and the assumptions in
Corollary~\ref{cor:Ideals_in_section_algebras} as special cases.  In
view of Theorem~\ref{thm:residual_aperiodic_action}, the following
two theorems are immediate corollaries of Theorems
\ref{thm:Kirchberg_Sierakowski}
and~\ref{thm:pure_infiniteness_criteria}:

\begin{theorem}
  \label{thm:Kirchberg_Sierakowski2}
  Assume~\eqref{standing_assumption}.  Let \(\mathcal{F}\subseteq A^+\)
  fill~\(A\)
  and be invariant under \(\varepsilon\)\nb-cut-downs.
  Then~\(B\)
  is strongly purely infinite if and only if each pair of elements
  \(a,b\in \mathcal{F}\)
  has the matrix diagonalisation property in~\(B\).
\end{theorem}

\begin{theorem}
  \label{the:pure_infiniteness_for_crossed_products}
  Assume~\eqref{standing_assumption}.  Let
  \(\mathcal{F}\subseteq A^+\) residually support~\(A\).  Suppose
  that \(\Ideals^{\Hilm}(A)\) is finite or the projections in
  \(\mathcal{F}\) separate the ideals in~\(\Ideals^{\Hilm}(A)\).
  Then~\(B\) is strongly purely infinite \textup{(}with the ideal
  property\textup{)} if and only if every element in
  \(\mathcal{F}\setminus\{0\}\) is properly infinite in~\(B\).
\end{theorem}

In order to use these results, we need conditions that suffice for
\(a,b\in A^+\) to have the matrix diagonalisation property in~\(B\)
or for \(a\in A^+\setminus \{0\}\) to be properly infinite in~\(B\).
Checking the matrix diagonalisation property is usually difficult.
Nevertheless, the following lemma may be useful (see
\cites{Kirchberg-Sierakowski:Strong_pure,
  Kwasniewski-Meyer:Aperiodicity}):

\begin{lemma}
  \label{lem:from separation to diagonalisation}
  Let \(a,b\in A^+\setminus\{0\}\).  Suppose that for each
  \(c\in \Hilm_t\), \(t\in S\), and each \(\varepsilon >0\) there
  are \(n,m\in\N\) and \(a_i\in a \Hilm_{s_i}\), \(s_i\in S\), for
  \(i=1,\dotsc,n\) and \(b_j\in b\Hilm_{t_j}\), \(t_j\in S\), for
  \(j=1,\dotsc,m\) such that
  \[
    a \approx_\varepsilon  \sum_{i=1}^n a_i^*a_i,
    \quad
    b\approx_\varepsilon \sum_{i=1}^m b_i^* b_i,
    \quad
    \sum_{i,j=1, i\neq j}^{n}
    a_i^* a_j\approx_{\varepsilon} 0,
    \quad
    \sum_{i,j=1, i\neq j}^{n,m}
    b_i^* b_j\approx_{\varepsilon} 0
  \]
  and \(\sum_{i=1,j=1}^{n,m} a_i^* c b_j\approx_{\varepsilon} 0\).
  Then \(a,b\in A^+\setminus\{0\}\) have the matrix diagonalisation
  property in~\(B\).
  \end{lemma}

\begin{proof}
  Let \(\CC\defeq \bigcup_{t\in S} \Hilm_t\)
  and \(\mathcal{S}\defeq \bigcup_{t\in S} \Hilm_t\).
  We claim that \(a, b\in A^+\setminus\{0\}\)
  have the matrix diagonalisation property with respect to \(\CC\)
  and~\(\mathcal{S}\)
  as introduced in
  \cite{Kirchberg-Sierakowski:Filling_families}*{Definition~4.6}.
  Indeed, let \(x\in \Hilm_t\)
  be such that
  \(\left(\begin{smallmatrix} a&x^*\\x&b \end{smallmatrix}\right) \in
  M_2(B)^+\)
  and let \(\varepsilon >0\).
  Let \(a_i\in a \Hilm_{s_i}\)
  and \(b_j\in b \Hilm_{t_j}\)
  satisfy the conditions described in
  the assertion
  with \(c\defeq a^{\nicefrac12}x b^{\nicefrac12}\).
  We may write \(a_i=a^{\nicefrac12}x_i\)
  and \(b_j=b^{\nicefrac12}y_j\) for some \(x_i\), \(y_j\).  Let
  \(d_1\defeq \sum_{i=1}^n x_i\) and
  \(d_2\defeq \sum_{j=1}^n y_j\).
  The assumed estimates 
  imply that
  \[
    d_1^*a d_1 \approx_{2\varepsilon} a,\qquad
    d_2^* b d_2\approx_{2\varepsilon} b,\qquad
    d_1^* x d_2\approx_\varepsilon 0.
  \]
  This proves our claim.  Clearly, \(\mathcal{S}\) is a
  multiplicative subsemigroup of~\(B\),
  \(\mathcal{S}^* \mathfrak{S}\mathcal{S}\subseteq \mathfrak{S}\),
  \(A \mathcal{S} A\subseteq \mathcal{S}\), and the closed linear
  span of \(\CC\) is \(B\).  Thus \(a, b\in A^+\setminus\{0\}\) have
  the matrix diagonalisation property in~\(B\) by
  \cite{Kirchberg-Sierakowski:Filling_families}*{Lemma~5.6}.
\end{proof}

Now we will focus on ways to check whether
\(a\in A^+\setminus \{0\}\) is properly infinite in~\(B\).  The
following definition generalises
\cite{Kwasniewski-Szymanski:Pure_infinite}*{Definition~5.1} and
\cite{Kwasniewski-Meyer:Aperiodicity}*{Definition~5.5} from groups
to inverse semigroups.  A crucial point is that the properties
depend only on the Fell bundle \(\Hilm=(\Hilm_t)_{t\in S}\), not on
the norm in~\(B\).

\begin{definition}
  \label{def:infinite_elements_B}
  An element \(a\in A^+\setminus\{0\}\) is called
  \begin{enumerate}
  \item \label{en:infinite_elements_B1}%
    \emph{\(\Hilm\)\nb-infinite} if there is
    \(b\in A^+\setminus\{0\}\) such that for each
    \(\varepsilon >0\), there are \(n,m\in\N\), \(t_i\in S\) and
    \(a_i \in a\Hilm_{t_i}\) for \(1 \le i \le n+m\), such that
    \[
      a \approx_\varepsilon  \sum_{i=1}^{n} a_i^*a_i,
      \quad
      b\approx_\varepsilon \sum_{i=n+1}^{n+m} a_j^* a_j,
      \quad
      \sum_{\twolinesubscript{i,j=1}{i\neq j}}^{n+m}
      \norm{a_i^* a_j}\le \varepsilon;
    \]
  \item \label{en:infinite_elements_B1res}%
    \emph{residually \(\Hilm\)\nb-infinite} if \(a+I\) is
    \(\Hilm|_{A/I}\)\nb-infinite for all
    \(I\in \Ideals^{\Hilm}(A)\) with \(a\notin I\);
  \item \label{en:infinite_elements_B2}%
    \emph{properly \(\Hilm\)\nb-infinite} if for all
    \(\varepsilon >0\) there are \(n,m\in\N\), \(t_i\in S\) and
    \(a_i \in a\Hilm_{t_i}\) for \(1 \le i \le n+m\), such that
    \[
      a \approx_\varepsilon  \sum_{i=1}^{n} a_i^*a_i,
      \quad
      a\approx_\varepsilon \sum_{j=n+1}^m a_j^* a_j,
      \quad
      \sum_{\twolinesubscript{i,j=1}{i\neq j}}^{n+m}
      \norm{a_i^* a_j}\le \varepsilon;
    \]
  \item \emph{\(\Hilm\)\nb-paradoxical} if the condition
    in~\ref{en:infinite_elements_B2} holds with \(\varepsilon =0\),
    that is, there are \(n,m\in\N\), \(t_i\in S\), and
    \(a_i \in a\Hilm_{t_i}\) for \(1 \le i \le n+m\), such that
    \[
      a = \sum_{i=1}^n a_i^* a_i,\qquad
      a = \sum_{j=n+1}^{n+m} a_j^* a_j,\qquad
      a_i^* a_j =0 \quad\text{for }i\neq j.
    \]
  \end{enumerate}
\end{definition}

\begin{lemma}
  \label{lem:Fell_infiniteness}
  If \(a\in A^+\setminus\{0\}\) is \(\Hilm\)\nb-infinite, then it is
  infinite in~\(B\).  If \(a\in A^+\setminus\{0\}\) is
  properly \(\Hilm\)\nb-infinite, then it is properly infinite
  in~\(B\).
\end{lemma}

\begin{proof}
  First let \(a\in A^+\setminus\{0\}\) be \(\Hilm\)\nb-infinite.
  Let \(b\in A^+\setminus\{0\}\) be as in Definition
  \ref{def:infinite_elements_B}.\ref{en:infinite_elements_B1}.  For
  \(\varepsilon >0\), there are \(n,m\in\N\) and \(t_i\in S\) and
  \(a_i\in a \Hilm_{t_i}\) for \(i=1,\dotsc,n+m\), \(n,m\in\N\) as
  in Definition
  \ref{def:infinite_elements_B}.\ref{en:infinite_elements_B1}.  Let
  \(x\defeq \sum_{i=1}^n a_i\) and \(y\defeq\sum_{i=n+1}^m a_i\).
  Then $x, y \in aB$.  Simple estimates such as
  \[
    \norm{x^* x - a} = \norm[\bigg]{\sum_{j=1}^n a_j^* a_j - a +
      \sum_{i,j=1, i\neq j}^{n} a_i^* a_j} \le
    \varepsilon + \varepsilon
  \]
  show that \(x^*x\approx_{2\varepsilon} a\),
  \(y^*y\approx_{2\varepsilon} b\) and
  \(y^*x \approx_{\varepsilon} 0\).  Hence~\(a\) is infinite
  in~\(B\) (see Definition~\ref{def:kinds_of_elements}).  The proof
  when~\(a\) is properly \(\Hilm\)\nb-infinite is the same with
  \(b=a\).
\end{proof}

Let us compare the definitions of infinite and properly infinite
elements in Definition~\ref{def:kinds_of_elements} to the
definitions of \(\Hilm\)\nb-infinite and properly
\(\Hilm\)\nb-infinite elements in
Definition~\ref{def:infinite_elements_B}.  There are two
differences.  First, we now choose the elements $x, y \in aB$ in the
subalgebra \(A\rtimes_\alg S\), so that we may write them as a
finite sum \(\sum a_i\) with \(a_i \in a\Hilm_{t_i}\).  Secondly, we
estimate each product~\(a_i^* a_j\) for \(i\neq j\) separately.  The
first change does not achieve much because \(A\rtimes_\alg S\) is
dense in~\(B\) and we only aim for approximate equalities anyway.
The second change simplify the estimates a lot because
\(\norm{a_i^* a_j}\) is computed in the Hilbert \(A\)\nb-bimodule
\(\Hilm_{t_i^* t_j}\), whereas the norm estimates in
Definition~\ref{def:kinds_of_elements} involve the \(\Cst\)\nb-norm
of~\(B\).  For an \(\Hilm\)\nb-paradoxical element, we even assume
the products~\(a_i^* a_j\) for \(i\neq j\) to vanish exactly.  This
is once again much easier to check.  Paradoxical elements are also
important because they are related to paradoxical decompositions,
which were studied already by Banach and Tarski.  In the setting of
purely infinite crossed products, their importance was highlighted
by R\o{}rdam and
Sierakowski~\cite{Rordam-Sierakowski:Purely_infinite}.  The
implications among our infiniteness conditions hinted at above are
summarised in the following proposition:

\begin{proposition}
  \label{prop:proper_infiniteness_conditions}
  Assume that~\(A\) separates ideals in~\(B\).  Consider the
  following conditions \(a\in A^+\setminus \{0\}\) may satisfy:
  \begin{enumerate}
  \item \label{en:proper_infiniteness_conditions1}%
    \(a\) is properly infinite in \(B\);
  \item \label{en:proper_infiniteness_conditions2}%
    for each \(\varepsilon >0\) there are \(n,m\in\N\),
    \(t_i\in S\), and \(a_i \in a\Hilm_{t_i}\) for
    \(1 \le i \le n+m\), such that
    \[
      a \approx_\varepsilon  \sum_{i,j=1}^{n} a_i^*a_j,
      \qquad
      a\approx_\varepsilon \sum_{i,j=n+1}^m a_i^* a_j,
      \qquad
      \sum_{i,j=1, i\neq j}^{n+m}
      a_i^* a_j\approx_{\varepsilon} 0;
    \]
  \item \label{en:proper_infiniteness_conditions3}%
    \(a\) is residually \(\Hilm\)\nb-infinite;
  \item \label{en:proper_infiniteness_conditions4}%
    \(a\) is properly \(\Hilm\)\nb-infinite;
  \item \label{en:proper_infiniteness_conditions41}%
    \(a\) is \(\Hilm\)\nb-paradoxical.
  \end{enumerate}
  Then \ref{en:proper_infiniteness_conditions1}\(\Leftrightarrow\)%
  \ref{en:proper_infiniteness_conditions2}\(\Leftarrow\)%
  \ref{en:proper_infiniteness_conditions3}\(\Leftarrow\)%
  \ref{en:proper_infiniteness_conditions4}\(\Leftarrow\)%
  \ref{en:proper_infiniteness_conditions41}.
\end{proposition}

\begin{proof}
  The implications
  \ref{en:proper_infiniteness_conditions41}\(\Rightarrow\)%
  \ref{en:proper_infiniteness_conditions4}\(\Rightarrow\)%
  \ref{en:proper_infiniteness_conditions3} are straightforward.  By
  \cite{Kirchberg-Rordam:Non-simple_pi}*{Proposition 3.14}, \(a\) is
  properly infinite if and only if it is residually infinite.
  Since~\(A\) separates ideals in~\(B\), any ideal in~\(B\) comes
  from an invariant ideal in~\(A\), as in the definition that~\(a\)
  is residually \(\Hilm\)\nb-infinite.  Together with
  Lemma~\ref{lem:Fell_infiniteness}, this shows
  that~\ref{en:proper_infiniteness_conditions3}
  implies~\ref{en:proper_infiniteness_conditions1}.

  According to Definition~\ref{def:kinds_of_elements},
  \(a\in A^+\setminus\{0\}\) is properly infinite in~\(B\) if and
  only if, for all \(\varepsilon >0\), there are \(x,y\in a\cdot B\)
  with \(x^*x\approx_\varepsilon a\), \(y^*y\approx_\varepsilon b\)
  and \(x^*y\approx_\varepsilon 0\).  Without loss of generality, we
  may pick \(x,y\in a \cdot (\sum_{t\in S}\Hilm_t)\) because
  \(\sum_{t\in S}\Hilm_t\) is dense in \(B\).  So
  \(x = \sum_{j=1}^n a_j\) and \( y = \sum_{j=n+1}^{n+m} a_j\) for
  some \(n,m\in\N\), \(t_i\in S\), and \(a_i \in a\Hilm_{t_i}\) for
  \(1 \le i \le n+m\).  The relations \(x^*x\approx_\varepsilon a\),
  \(y^*y\approx_\varepsilon b\) and \(x^*y\approx_\varepsilon 0\)
  translate to those described in
  \ref{en:proper_infiniteness_conditions2}.  This proves that
  \ref{en:proper_infiniteness_conditions1}
  and~\ref{en:proper_infiniteness_conditions2} are equivalent.
\end{proof}

It is unclear whether the implications in
Proposition~\ref{prop:proper_infiniteness_conditions} may be
reversed.

\begin{remark}
  The example of graph \(\Cst\)\nb-algebras shows that it may be
  much easier to check that an element is residually
  \(\Hilm\)\nb-infinite than that it is properly
  \(\Hilm\)\nb-infinite (see also
  \cite{Kwasniewski-Szymanski:Pure_infinite}*{Remark~7.10}).
\end{remark}

\begin{corollary}
  \label{cor:pure_infiniteness_for_crossed_products}
  Assume~\eqref{standing_assumption}.  Let
  \(\mathcal{F}\subseteq A^+\) residually support~\(A\).  Suppose
  that \(\Ideals^{\Hilm}(A)\) is finite or that~\(\mathcal{F}\)
  consists of projections, or that the projections
  in~\(\mathcal{F}\) separate the ideals in~\(\Ideals^{\Hilm}(A)\).
  If every element in \(\mathcal{F}\setminus\{0\}\) is residually
  \(\Hilm\)\nb-infinite, then \(A\rtimes_\ess S\) is purely infinite
  and has the ideal property.
\end{corollary}

\begin{proof}
  Combine Theorem~\ref{the:pure_infiniteness_for_crossed_products}
  and Proposition~\ref{prop:proper_infiniteness_conditions}.
\end{proof}

We may simplify our conditions further if~\(A\) is commutative.
Then \(A\rtimes_\ess S \cong \Cst_\ess(\Gr,\Sigma)\) for a twisted
\'etale groupoid~\(\Gr\) with object space~\(\hat{A}\).  The
twist~\(\Sigma\) is always locally trivial.  Therefore, the
bisections that trivialise the twist~\(\Sigma\) form a wide inverse
subsemigroup~\(S'\) among all bisections of~\(\Gr\) (see
\cite{BussExel:Regular.Fell.Bundle}*{Theorem~7.2}).  Then
\(\Cst_\ess(\Gr,\Sigma) \cong A\rtimes_\ess S'\).  The action
of~\(S'\) on~\(A\) is equivalent to a twisted action as in
\cite{BussExel:Regular.Fell.Bundle}*{Definition~4.1}, that is,
each~\(\Hilm_t\) for \(t\in S\) comes from an isomorphism between
two ideals in~\(A\).  We assume this because it allows us to
identify elements of~\(\Hilm_t\) with \(\Cont_0\)\nb-functions on
\(\s(\Hilm_t) \subseteq \hat{A}\).  This discussion  shows how
to turn any inverse semigroup action on a commutative
\(\Cst\)\nb-algebra into a twisted action by partial automorphisms.

\begin{lemma}
  \label{lem:strict_infiniteness}
  Assume that~\(A\) is commutative and that~\(S\) acts on~\(A\) by a
  twisted action by partial automorphisms as in
  Example~\textup{\ref{ex:twisted actions}}.  Equip~\(\dual{A}\)
  with the dual action of~\(S\).  Let \(a\in A^+\) and
  \(V \defeq \setgiven{x\in\dual{A}}{a(x)\neq0}\).  Consider the
  following conditions:
  \begin{enumerate}
  \item \label{enu:strict_infiniteness1}%
    the condition in
    Definition~\textup{\ref{def:infinite_elements_B}.%
      \ref{en:infinite_elements_B1}}
    holds with \(\varepsilon=0\);
  \item \label{enu:strict_infiniteness2}%
    there are \(b\in (aAa)^+\setminus\{0\}\), \(n\in\N\) and
    \(t_i\in S\), \(a \in a \Hilm_{t_i}\) for \(1\le i \le n\) such
    that
    \[
      a =\sum_{i=1}^n a_i^*a_i,
      \quad\text{ and }\quad
      a_i^* a_j=0, \quad  a_i^*b=0 \qquad
      \text{for  all }\, i,j=1,\dotsc,n, i\neq j;
    \]
  \item \label{enu:strict_infiniteness3}%
    there are \(n\in\N\), \(t_1,\dotsc,t_{n}\in S\), and open
    subsets \(V_1,\dotsc,V_{n}\subseteq V\) such that
    \begin{enumerate}
    \item \(V_i\)
      is contained in the domain of~\(t_i\) for \(1\le i\le n\);
    \item \((t_i\cdot V_i)\cap (t_j\cdot V_j)=\emptyset\)
      if \(1\le i<j\le n\);
    \item \(V=\bigcup_{i=1}^n V_i\)
      and \(\cl{\bigcup_{i=1}^{n} t_i\cdot V_i}\subsetneq V\);
    \end{enumerate}
  \item \label{enu:strict_infiniteness4}%
    \(a\) is \(\Hilm\)\nb-infinite.
  \end{enumerate}
  Then \ref{enu:strict_infiniteness1}\(\Leftrightarrow\)%
  \ref{enu:strict_infiniteness2}\(\Rightarrow\)%
  \ref{enu:strict_infiniteness3}\(\Rightarrow\)\ref{enu:strict_infiniteness4}.
  The implications
  \ref{enu:strict_infiniteness1}\(\Leftrightarrow\)%
  \ref{enu:strict_infiniteness2}\(\Rightarrow\)\ref{enu:strict_infiniteness4}
  hold in full generality.
\end{lemma}

\begin{proof}
  If~\ref{enu:strict_infiniteness2} holds, then taking
  \(m\defeq 1\), \(t_{n+1}\defeq1\) and \(a_{n+1}\defeq\sqrt{b}\) we
  get~\ref{enu:strict_infiniteness1}.  Conversely,
  if~\ref{enu:strict_infiniteness1} holds, then there are
  \(t_i\in S\), \(a_i \in a \Hilm_{t_i}\) for \(i=1,\dotsc,n+m\),
  such that \(a =\sum_{i=1}^n a_i^*a_i\),
  \(\sum_{i=n+1}^{n+m}a_i^*a_i\neq 0\) and \(a_i^* a_j=0\) for
  \(i\neq j\).  Thus putting \(b\defeq\sum_{i=n+1}^{n+m}a_ia_i^*\)
  gives~\ref{enu:strict_infiniteness2}.  This shows
  that~\ref{enu:strict_infiniteness1}
  and~\ref{enu:strict_infiniteness2} are equivalent.

  Now let \(b\in (aAa)^+\setminus\{0\}\), \(n\in\N\), and
  \(t_i\in S\), \(a_i \in a \Hilm_{t_i}\) for \(i=1,\dotsc,n\) be as
  in~\ref{enu:strict_infiniteness2}.  We identify the
  fibres~\(\Hilm_t\) for \(t\in S\) with spaces of sections of the
  associated line bundle over~\(\Gr\).  Put
  \(U_i\defeq \setgiven{\gamma\in \Gr}{ \norm{a_i(\gamma)} >0}\) for
  \(i=1,\dotsc,n\) and
  \(W\defeq \setgiven{x\in X}{\norm{b(\gamma)} >0}\).  Then
  \(V_i\defeq s(U_i)=\setgiven{x\in X}{(a_i^* a_i)(x)>0} \subseteq
  V\) is contained in the domain of~\(t_i\) for \(1\le i\le n\).
  The equality \(a =\sum_{i=1}^n a_i^*a_i\) implies that
  \(V=\bigcup_{i=1}^n s(U_i)=\bigcup_{i=1}^n V_i\).  Similarly,
  \(a_i^* a_j=0\) holds if and only if
  \((t_i\cdot V_i)\cap (t_j\cdot V_j)=r(U_{i})\cap
  r(U_{j})=\emptyset\) for all \(i\neq j\) and \(a_i^*b=0\) holds if
  and only if
  \(W\subseteq V \setminus \bigcup_{i=1}^n r(U_i)=V \setminus
  \bigcup_{i=1}^n (t_i\cdot V_i)\); here we identify functions in
  \(\Hilm_t\) with \(\Cont_0\)\nb-functions on bisections.  Such
  a~\(W\) exists if and only if
  \(\cl{\bigcup_{i=1}^n t_i\cdot V_i}\subsetneq V\).
  Hence~\ref{enu:strict_infiniteness2}
  implies~\ref{enu:strict_infiniteness3}.

  Next we show that~\ref{enu:strict_infiniteness3}
  implies~\ref{enu:strict_infiniteness4}.  Let
  \(t_1,\dotsc,t_{n}\in S\), and \(V_1,\dotsc,V_{n}\subseteq V\) be
  as in~\ref{enu:strict_infiniteness3}.  Let
  \(b\in A^+\setminus\{0\}\) be any function that vanishes outside
  the open set
  \(V \bigm\backslash \cl{\bigcup_{i=n+1}^{n} t_i\cdot V_i}\).  Fix
  \(\varepsilon>0\).  Let
  \(K \defeq \setgiven{x\in\dual{A}}{a(x)\ge\varepsilon}\).  Let
  \(w_1,\dotsc, w_n\in A\) be a partition of unity subordinate to
  the open covering \(K\subseteq \bigcup_{i=1}^n V_i\).  Let
  \(a_i \defeq (a-\varepsilon)_+^{1/2} \cdot w_i^{1/2}\) for
  \(i=1,\dotsc,n\).  These functions vanish outside~\(K\),
  and~\(a_i\) belongs to the domain of~\(t_i\).
  Since~\(\Hilm_{t_i}\) comes from a partial automorphism, we may
  view~\(a_i\) as an element of~\(\Hilm_{t_i}\).  It belongs to
  \(a\cdot \Hilm_{t_i}\) because the support of~\(a_i\) is contained
  in~\(V\).  The product~\(a_i^* a_j\) is defined using the Fell
  bundle structure.  If \(i\neq j\), then \(a_i^* a_j=0\) because
  \((t_i\cdot V_i)\cap (t_j\cdot V_j)=\emptyset\).  Similarly, we
  get~\(a_i^* b=0\).  And
  \[
  \sum_{i=1}^n a_i^* a_i
  = \sum_{i=1}^n (a-\varepsilon)_+ \cdot w_i
  = (a-\varepsilon)_+ \approx_\varepsilon a.
  \]
  Hence \(a\)   is \(\Hilm\)-infinite.
\end{proof}

\begin{remark}
  For strongly boundary group actions
  (see~\cite{Laca-Spielberg:Purely_infinite}) and, more generally,
  for filling actions
  (see~\cite{Jolissaint-Robertson:Simple_purely_infinite})
  condition~\ref{enu:strict_infiniteness3} in
  Lemma~\ref{lem:strict_infiniteness} holds for every nonempty open
  subset~\(V\).  Thus if~\(\Hilm\) comes from such an action, then
  every element in \(A^+\setminus \{0\}\) is \(\Hilm\)\nb-infinite.
  This also holds when~\(A\) is noncommutative (see
  \cite{Kwasniewski-Szymanski:Pure_infinite}*{Lemma~5.12}).
\end{remark}

\begin{remark}
  An \'etale, Hausdorff, locally compact groupoid~\(H\) is
  \emph{locally contracting} if for each nonempty open set~\(U\) in
  the unit space~\(H^0\) of~\(H\) there is a bisection
  \(B\subseteq H\) with \(\cl{\rg(B)} \subsetneq \s(B) \subseteq U\)
  (see~\cite{Anantharaman-Delaroche:Purely_infinite}).  Given a wide
  inverse subsemigroup \(S\subseteq \Bis(H)\), we may strengthen
  this criterion by requiring \(B\subseteq t\) for some \(t\in S\).
  Then we may rewrite \(\cl{\rg(B)} \subsetneq \s(B) \subseteq U\)
  as follows: there is \(t\in S\) and \(V\subseteq U\) contained in
  the domain of~\(t\) with \(\cl{t\cdot V} \subsetneq V\).  This is
  the case \(n=1\) of condition~\ref{enu:strict_infiniteness3} in
  Lemma~\ref{lem:strict_infiniteness}.  As a result, if the dual
  groupoid \(\dual{A} \rtimes S\) is locally contracting, then for
  any \(a\in A^+ \setminus \{0\}\) there is \(0 \neq a_2 \le a\)
  that is \(\Hilm\)\nb-infinite; namely, choose \(U=\supp a\) and
  then~\(a_2\) with \(\supp a_2 = V\) and \(a_2 \le a\) for~\(V\) as
  above.
\end{remark}

Condition~\ref{enu:strict_infiniteness3} in
Lemma~\ref{lem:strict_infiniteness} could be relaxed so that it
still implies \(\Hilm\)\nb-infiniteness, by using compact subsets
of~\(V\).  We formulate the relevant condition implying
\(\Hilm\)\nb-proper infiniteness:

\begin{lemma}
  \label{lem:properly_infinite_A_comm}
  Retain the assumptions of
  Lemma~\textup{\ref{lem:strict_infiniteness}}.  In particular, let
  \(a\in A^+\) and \(V \defeq \setgiven{x\in\dual{A}}{a(x)\neq 0}\).
  If for each compact subset \(K\subseteq V\) there are
  \(n,m\in\N\), \(t_1,\dotsc,t_{n+m}\in S\), and open subsets
  \(V_1,\dotsc,V_{n+m}\subseteq V\) such that
  \((t_i\cdot V_i)\cap (t_j\cdot V_j)=\emptyset\) if
  \(1\le i<j\le n+m\), \(K\subseteq \bigcup_{i=1}^n V_i\) and
  \(K\subseteq \bigcup_{i=n+1}^{n+m} V_i\), then \(a\) is
  \(\Hilm\)-properly infinite.
 \end{lemma}

\begin{proof}
  Fix \(\varepsilon>0\).  Let
  \(K \defeq \setgiven{x\in\dual{A}}{a(x)\ge\varepsilon}\).  Choose
  \(n,m,t_i,V_i\) as in the assumption of the lemma.  Let
  \(w_1,\dotsc, w_n\in A\) and \(w_{n+1},\dotsc, w_{n+m}\in A\) be
  partitions of unity subordinate to the open coverings
  \(K\subseteq \bigcup_{i=1}^n V_i\) and
  \(K\subseteq \bigcup_{i=n+1}^{n+m} V_i\), respectively.  Let
  \(a_i \defeq (a-\varepsilon)_+^{1/2} \cdot w_i^{1/2}\) for
  \(i=1,\dotsc,n+m\).  As in the proof of the implication
  \ref{enu:strict_infiniteness3}\(\Rightarrow\)\ref{enu:strict_infiniteness4}
  in Lemma~\ref{lem:strict_infiniteness} one sees that
  treating~\(a_i\) as an element of~\(\Hilm_{t_i}\), the
  elements~\(a_i\) satisfy the relations in
  Definition~\ref{def:infinite_elements_B}.%
  \ref{en:infinite_elements_B2}.
\end{proof}

Now we assume, in addition, that the spectrum~\(\dual{A}\) is
totally disconnected.  This implies that the compact open bisections
form a basis for the topology and that~\(A\) is spanned by
projections.  We are going to see that a projection is
\(\Hilm\)\nb-paradoxical if and only if its support is
\((2,1)\)-paradoxical as defined in~\cite{Boenicke-Li:Ideal}.  Such
open subsets give purely infinite elements in the type semigroup
considered in \cites{Boenicke-Li:Ideal, Rainone-Sims:Dichotomy,
  Ma:Purely_infinite_groupoids}.

\begin{definition}[\cite{Boenicke-Li:Ideal}]
  \label{def:two_one_paradoxical}
  Let~\(\Gr\) be an ample groupoid.  We say that a compact open set
  \(V\subseteq G^0\) is \emph{\((2,1)\)-paradoxical} if there are
  \(n,m\in\N\) and compact open bisections \(U_i\subseteq \Gr\) for
  \(1\le i \le n+m\) such that \(\rg(U_i) \subseteq V\) for
  \(1\le i \le n+m\) and
  \[
    V = \bigsqcup_{i=1}^n \s(U_i),\qquad
    V = \bigsqcup_{i=n+1}^{n+m} \s(U_i),\qquad
    \rg(U_i) \cap \rg(U_j) = \emptyset \quad\text{for }i\neq j.
  \]
\end{definition}

\begin{proposition}
  \label{prop:paradoxicality}
  Let~\(\Hilm\) be an action of an inverse semigroup~\(S\) by
  Hilbert bimodules on a commutative \(\Cst\)\nb-algebra~\(A\) with
  totally disconected spectrum~\(\dual{A}\); equivalently, the dual
  groupoid \(\Gr\defeq \dual{A}\rtimes S\) is ample.  A projection
  \(a\in A^+\) is \(\Hilm\)\nb-paradoxical if and only if its support
  \(V \defeq \setgiven{x\in\dual{A}}{a(x)\neq 0}\) is \((2,1)\)-paradoxical.
\end{proposition}

\begin{proof}
  Suppose first that \(a\in A^+\setminus\{0\}\) is
  \(\Hilm\)\nb-paradoxical.  That is, there are \(n,m\in\N\),
  \(t_1,\dotsc,t_{n+m}\in S\), and \(a_i \in a\Hilm_{t_i}\) such
  that \(a = \sum_{i=1}^n a_i^* a_i = \sum_{i=n+1}^{n+m} a_i^* a_i\)
  and \(a_i^* a_j =0\) for \(i\neq j\).  Let \(1\le i\le n+m\).
  Recall that we may treat~\(\Hilm_{t_i}\) as spaces of
  sections~\(A_{U_i}\) of a line bundle over
  \(\Gr=\dual{A}\rtimes S\) that are supported on open bisections
  \(U_i\in \Bis(\Gr)\).  Thus
  \(U_i\defeq \setgiven{\gamma\in \Gr}{\norm{a_i(\gamma)} >0}\) is
  an open bisection of~\(\Gr\) contained in~\(U_i\).  Since
  \(a_i\in aA_{U_i}\), we have
  \(r(U_i)=\setgiven{x\in \dual{A}}{(a_i a_i^*)(x)>0} \subseteq V\).
  And \(a_i^* a_j=0\) implies that
  \(r(U_{i})\cap r(U_{j})=\emptyset\) for all \(i\neq j\).  Since
  \(\setgiven{x\in X}{(a_i^* a_i)(x)>0} = s(U_i)\), the equalities
  \(a =\sum_{i=1}^n a_i^*a_i\) and
  \(a =\sum_{i=n+1}^{n+m} a_i^*a_i\) imply
  \(V = \bigsqcup_{i=1}^n \s(U_i)\) and
  \(V = \bigsqcup_{i=n+1}^{n+m} \s(U_i)\).  Hence the family
  \(U_i\in \Bis(\Gr)\) for \(1\le i \le n+m\) has all the desired
  properties, except that~\(U_i\) need not be compact.  However,
  since~\(\Gr\) is ample, every~\(U_i\) is a union of some compact
  open bisections.  Since~\(V\) is compact and
  \(V = \bigsqcup_{i=1}^n \s(U_i) = \bigsqcup_{i=n+1}^{n+m}
  \s(U_i)\), we may, in fact, replace each~\(U_i\) for
  \(1\le i \le n+m\) by a finite union of compact open bisections.
  This gives a compact open bisection.

  Conversely, let \(U_i\subseteq \Gr\) for \(1\le i \le n+m\) be a
  family of bisection as in Definition~\ref{def:two_one_paradoxical}.  Let
  \(S'\subseteq \Bis(\Gr)\) be the family of open compact bisections
  that trivialise the twist, that is, the restrictions of the
  associated line bundle over~\(\Gr\) to sets in~\(S'\) are trivial.
  Note that~\(S'\) forms an inverse semigroup and a basis for the
  topology of~\(\Gr\); this holds for the family of all open
  bisections that trivialise the twist, by the proof of
  \cite{BussExel:Regular.Fell.Bundle}*{Theorem~7.2}, and for the
  family of all compact open bisections because~\(\Gr\) is ample.
  Since
  \(V=\bigsqcup_{i=1}^n \s(U_i)= \bigsqcup_{j=n+1}^{n+m} \s(U_j)\)
  is compact, for each \(i=1,\dotsc,n\) we may find a finite family
  of sets \((U_{i,j})_{j=1}^{n_i} \subseteq S'\) such that
  \(\bigcup_{j=1}^{n_i} U_{i,j}\subseteq U_i\) and
  \(V=\bigcup_{i=1}^n \bigcup_{j=1}^{n_i} s(U_{i,j})\).  Since the
  bisections \((U_{i,j})_{j=1}^{n_i}\subseteq U_i\) are closed and
  open, we may arrange that the sets \((s(U_{i,j}))_{j=1}^{n_i}\)
  are pairwise disjoint.  Then the sets
  \((U_{i,j})_{i=1,j=1}^{n,n_i}\) are pairwise disjoint, and since
  \(\bigcup_{j=1}^{n_i} U_{i,j}\subseteq U_i\), for
  \(i=1,\dotsc,n\), also \((r(U_{i,j}))_{i=1,j=1}^{n,n_i}\) are
  pairwise disjoint.  We put \(a_{i,j}\defeq 1_{U_{i,j}}\), for
  \(i=1,\dots ,n\), \(j=1,\dotsc,n_i\).  By the choice of bisections
  in~\(S'\), we may treat~\(a_{i,j}\) as an element of the space
  \(\Contc(U_{ij})\) of sections of the line bundle over~\(\Gr\).
  By the construction of the Fell bundle over \(\dual{A}\rtimes S\),
  by passing if necessary to smaller sets, we may assume that each
  space \(\Contc(U_{i,j})\) is contained in~\(\Hilm_{t_{ij}}\) for
  some \(t_{ij}\in S\).  Hence \(a_{i,j}\in \Hilm_{t_{ij}}\) for all
  \(i,j\).  Using the Fell bundle structure, we get
  \[
    \sum_{i=1,j=1}^{n, n_i} a_{i,j}^* \cdot a_{i,j}
    = \sum_{i=1,j=1}^{n, n_i} 1_{s(U_{i,j})} = 1_V=a.
  \]
  Similarly, we get
  \(\sum_{i=n+1,j=1}^{n+m, n_i} a_{i,j}^* \cdot a_{i,j}=a\) and
  \(a_{i,j}^* a_{i',j'}=0\) for all \((i,j)\neq (i',j')\).  Hence
  \(a\) is \(\Hilm\)\nb-paradoxical.
\end{proof}

\begin{corollary}
  \label{cor:purely_infinite_groupoid}
  Let \((\Gr,\Sigma)\) be an essentially exact twisted groupoid
  where~\(\Gr\) is ample and residually topologically free with
  locally compact Hausdorff \(X\defeq \Gr^0\).  If every compact
  open subset of~\(X\) is \((2,1)\)-paradoxical, then the essential
  \(\Cst\)\nb-algebra \(\Cst_{\ess}(\Gr,\Sigma)\) is purely infinite
  \textup{(}and has the ideal property\textup{)}.
\end{corollary}

\begin{proof}
  View \(\Cst_{\ess}(\Gr,\Sigma)\) as the essential crossed product
  by an inverse semigroup action~\(\Hilm\) on \(\Cont_0(X)\) as in
  Examples \ref{ex:essential_groupoid_algebras}
  and~\ref{ex:essential_partial_algebras}.  The assertion follows
  from Proposition~\ref{prop:paradoxicality} and
  Corollary~\ref{cor:pure_infiniteness_for_crossed_products}.
\end{proof}

\begin{remark}
  \label{rem:paradoxicality_for_groupoids}
  When~\(\Gr\) is Hausdorff, then
  \(\Cst_{\ess}(\Gr,\Sigma)=\Cst_{\red}(\Gr,\Sigma)\) and
  \((\Gr,\Sigma)\) is essentially exact if and only if it is inner
  exact.  Thus Corollary~\ref{cor:purely_infinite_groupoid}
  generalises the pure infiniteness criteria in
  \cites{Boenicke-Li:Ideal, Rainone-Sims:Dichotomy}, where the
  authors considered Hausdorff ample groupoids without a twist.
  They proved, in addition, that if the type semigroup associated
  to~\(G\) is almost unperforated, then the implication in
  Corollary~\ref{cor:purely_infinite_groupoid} may be reversed.  We
  will generalise this and some other results of
  Ma~\cite{Ma:Purely_infinite_groupoids} to étale twisted groupoids
  in the forthcoming paper~\cite{Kwasniewski-Meyer:Type_semigroups}.
\end{remark}

\end{document}